\documentclass[10pt,twoside,a4paper,english,leqno]{article}

\title{Well-posedness of the Green-Naghdi and Boussinesq-Peregrine systems}
\author{Vincent Duch\^ene%
\thanks{IRMAR - UMR6625, CNRS and Univ. Rennes 1, Campus de Beaulieu, F-35042 Rennes cedex, France.
 VD is partially supported by the project Dyficolti ANR-13-BS01-0003-01 of the Agence Nationale de la Recherche.
}
\and Samer Israwi%
\thanks{Math\'ematiques, Facult\'e des sciences I et Ecole doctorale des sciences et technologie, Universit\'e Libanaise, Beyrouth, Liban. 
SI is partially supported by the Lebanese University research program (MAA group project).
}
}
\date{\today}

\usepackage{babel}          
\usepackage[utf8]{inputenc}  
\usepackage[T1]{fontenc}   
\usepackage{amsmath, amsthm,amsfonts,amssymb,fancyhdr,hyperref,mathtools}      

\numberwithin{equation}{section}

\newtheorem{Theorem}{Theorem}
\newtheorem{Proposition}{Proposition}[section]
\newtheorem{Lemma}[Proposition]{Lemma}
\newtheorem{Remark}[Proposition]{Remark}

\usepackage{titling}  
\makeatletter
\let\Title\@title
\let\Author\@author
\makeatother
 
\fancypagestyle{mystyle}{%
	\fancyfoot{}
	\fancyhead[CO]{\Title} 
	\fancyhead[CE]{\scriptsize V. Duch\^{e}ne and S. Israwi} 
	\fancyhead[RE]{\scriptsize\today}  
	\fancyhead[LO]{}  
	\fancyhead[LE,RO]{\scriptsize\thepage}
	}
\fancypagestyle{plain}{\fancyhead{} } 

\newcommand{\RR}{\mathbb{R}}
\newcommand{\NN}{\mathbb{N}}
\renewcommand{\t}{\widetilde}

\renewcommand{\S}{\mathcal{S}}

\newcommand{\T}{\mathcal{T}}
\newcommand{\Q}{\mathcal{Q}}
\newcommand{\R}{\mathcal{R}}
\renewcommand{\O}{\mathcal{O}}
\newcommand{\E}{\mathcal{E}}
\newcommand{\F}{\mathcal{F}}
\newcommand{\mfT}{\mathfrak{T}}
\newcommand{\mfa}{\mathfrak{a}}
\newcommand{\dd}{{\rm d}}
\newcommand{\Id}{{\rm Id}}

\renewcommand{\u}{{\bf u}}
\renewcommand{\v}{\mathbf{v}}
\newcommand{\w}{\mathbf{w}}

\renewcommand{\r}{{\bf r}}

\newcommand{\eqdef}{\stackrel{\rm def}{=}}
\newcommand{\nn}{\nonumber}
\newcommand{\ie}{{\em i.e.}~}
\newcommand{\eg}{{\em e.g.}~}
\newcommand{\id}[1]{\left\vert_{_{#1}}\right.}

\DeclareMathOperator{\curl}{curl}
\DeclareMathOperator*{\esssup}{ess\,sup}

\DeclarePairedDelimiter\abs{\lvert}{\rvert}

\DeclarePairedDelimiter\norm{\big\lvert}{\big\rvert}

\DeclarePairedDelimiter\bra{\big\langle}{\big\rangle}
\DeclarePairedDelimiter\Par{\big(}{\big)}

\begin{document}
\maketitle

\begin{abstract}
In this paper we address the Cauchy problem for two systems modeling the propagation of long gravity waves in a layer of homogeneous, incompressible and inviscid fluid delimited above by a free surface, and below by a non-necessarily flat rigid bottom. Concerning the Green-Naghdi system, we improve the result of Alvarez-Samaniego and Lannes~\cite{Alvarez-SamaniegoLannes08a} in the sense that much less regular data are allowed, and no loss of derivatives is involved. Concerning the Boussinesq-Peregrine system, we improve the lower bound on the time of existence provided by Mésognon-Gireau~\cite{Mesognon-Gireaub}. The main ingredient is a physically motivated change of unknowns revealing the quasilinear structure of the systems, from which energy methods are implemented.
\end{abstract}

\vfill

\tableofcontents

\vfill

\newpage

\section{Introduction}\label{S.Intro}

\subsection{Motivation}

The Green-Naghdi system\footnote{The Boussinesq-Peregrine system can be viewed as a simplification of the Green-Naghdi system for small-amplitude waves, which is particularly relevant for numerical purposes; see~\cite{Peregrine67,Lannes,Mesognon-Gireaub,BellecColinRicchiuto}. This work is dedicated to the Green-Naghdi system and we only remark incidentally that our strategy may also be favorably applied to the Boussinesq-Peregrine system.} (sometimes called Serre or fully nonlinear Boussinesq system) is a model for the propagation of gravity waves in a layer of homogeneous incompressible inviscid fluid with rigid bottom and free surface. It has been formally derived several times in the literature, in particular in~\cite{Serre53,SuGardner69,GreenNaghdi76,MilesSalmon85,Seabra-SantosRenouardTemperville87}, using different techniques and various hypotheses. For a clear and modern exposition, it is shown in~\cite{LannesBonneton09} that the Green-Naghdi system can be derived as an asymptotic model from the water waves system (namely the ``exact'' equations for the propagation of surface gravity waves), by assuming that the typical horizontal length of the flow is much larger than the depth of the fluid layer ---that is in the shallow-water regime---  and that the flow is irrotational. Roughly speaking, a Taylor expansion with respect to the small ``shallow-water parameter'' yields at first order the Saint-Venant system, and at second order the Green-Naghdi system (see for instance~\cite{Matsuno16} for higher order systems). As a relatively simple fully nonlinear model (that is without restriction on the amplitude of the waves) formally improving the precision of the Saint-Venant system, the Green-Naghdi system is widely used to model and numerically simulate the propagation of surface waves, in particular in coastal oceanography. It would be impossible to review the vast literature on the subject, and we only let the reader refer to~\cite{Barthelemy04,BonnetonChazelLannesEtAl11} for an introduction and relevant references.

In this work, we are interested in the structural properties and rigorous justification of the Green-Naghdi system. The derivation through formal Taylor expansions of Bonneton and Lannes~\cite{LannesBonneton09} can be made rigorous~\cite[Prop.~5.8]{Lannes}: roughly speaking, any sufficiently smooth solution of the water waves system satisfies the Green-Naghdi system up to a quantifiable (small) remainder. This consistency result is only one step towards the full justification of the model in the following sense: the solution of the water waves system and the solution of the Green-Naghdi system with corresponding initial data remain close on a relevant time interval. In order for such result to hold, one needs of course to ensure the existence and uniqueness of a solution to the Green-Naghdi system in the aforementioned time interval for a large class of initial data; one also needs a stability property ensuring that the two solutions are close. These two results typically call for robust energy estimates on exact and approximate solutions. 

Somewhat surprisingly, the well-posedness theory concerning the Cauchy problem for the Green-Naghdi system is in some sense less satisfactory than the corresponding one for the water waves system. Again, the literature on the latter problem is too vast to summarize, and we only mention the result of~\cite{Alvarez-SamaniegoLannes08,Iguchi09} and~\cite[Theorem 4.16]{Lannes}. Indeed, the latter pay attention to the various dimensionless parameters of the system, and in particular obtain results which hold uniformly with respect to the shallow-water parameter. The outcome of these results is that provided that the initial data is sufficiently regular (measured through Sobolev spaces) and satisfy physical assumptions ---the so-called non-cavitation and Rayleigh-Taylor criteria--- then there exists a unique solution of the water waves system preserving the regularity of the initial data. Moreover, the maximal time of existence may be bounded from below uniformly with respect to the shallow-water parameter; see details therein. Such result is very much nontrivial as the limit of small shallow-water parameter is singular in some sense. Similar results have been proved for the Green-Naghdi system in horizontal dimension $d=1$ in~\cite{Li06} (for flat bottom) and~\cite{Israwi11} (for general bathymetries), but are not yet available in dimension $d=2$. Alvarez-Samaniego and Lannes~\cite{Alvarez-SamaniegoLannes08a} proved an existence and uniqueness result on the correct time-scale but their proof relies on a Nash-Moser scheme, and as such involves a loss of derivatives between the regularity of the initial data and the control of the solution at positive times. {\em The main result of this paper is to show that this loss of derivatives is in fact not necessary, and that the Cauchy problem for the Green-Naghdi system is well-posed in the sense of Hadamard in Sobolev-type spaces.}

\subsection{Strategy}

Let us now introduce the system of equations at stake. In order to ease the discussion and notations, we restrict the study to the horizontal space $X\in\RR^d$ with $d=2$, although the results are easily adapted to the situation of $d=1$, thus yielding another proof of the result in~\cite{Israwi11}. The non-dimensionalized Green-Naghdi system may be expressed (see~\cite{LannesBonneton09,Lannes}) as
\begin{equation}\label{GN-u}
\left\{\begin{array}{l}
\partial_t\zeta+\nabla\cdot(h\u)=0,\\ \\
\big(\Id+\mu \T[h,\beta b]\big)\partial_t\u +\nabla\zeta+\epsilon(\u\cdot\nabla)\u+\mu\epsilon \big(\Q[h,\u]+\Q_b[h,\beta b,\u]\big)=0,
\end{array}\right.
\end{equation}
with $h=1+\epsilon\zeta-\beta b$ and
  \begin{align}
\label{def-T}
     \mathcal T[h,\beta b]\u &\eqdef \frac{-1}{3h}\nabla(h^3\nabla\cdot \u)+\frac1{2h}\Big(\nabla\big(h^2(\beta\nabla b)\cdot \u\big)-h^2(\beta\nabla b)\nabla\cdot \u\Big)+\beta^2(\nabla b\cdot \u)\nabla b,\\
\Q[h,\u]&\eqdef \frac{-1}{3h}\nabla\Big(h^3\big((\u\cdot\nabla)(\nabla\cdot\u)-(\nabla\cdot\u)^2\big)\Big),\nn \\
    \Q_b[h,\beta b,\u]&\eqdef\frac{\beta}{2h}\Big(\nabla\big( h^2(\u\cdot\nabla)^2 b\big) -h^2\big((\u\cdot\nabla)(\nabla\cdot\u)-(\nabla\cdot\u)^2\big)\nabla b\Big)+\beta^2\big((\u\cdot\nabla)^2 b\big) \nabla b.\nn
   \end{align}
Here, the unknowns are $\zeta(t,X)\in \RR$ and $\u(t,X)\in\RR^d$ (representing respectively the dimensionless surface deformation and layer-averaged horizontal velocity), $b(X)\in\RR$ is the fixed bottom topography (so that $h(t,X)$ represents the depth of the fluid layer) and $\epsilon,\beta,\mu$ are dimensionless parameters. 
   
As aforementioned, by setting $\mu=0$ in~\eqref{GN-u}, one recovers the Saint-Venant system, which is an archetype of first-order quasilinear systems of conservation laws. Our strategy in the following is to adapt to the Green-Naghdi equations the well-known techniques --- and in particular {\em a priori} energy estimates --- developed for such systems. Such energy estimates are obviously not guaranteed, due to the presence of the additional third-order nonlinear operators. The key ingredient of this work is the extraction of a quasilinear structure of~\eqref{GN-u}, form which energy estimates can be deduced, and eventually a standard Picard iteration scheme can be set up.

Such a ``quasilinearization'' is also the key ingredient in the proof of the local existence of solutions to the water waves system~\cite[Theorem 4.16]{Lannes}. However, the structure of the water waves system and the one of the Green-Naghdi system look different, due to the fact that they use different unknowns. Indeed, the second equation in~\eqref{GN-u} describes the time-evolution of the layer-averaged horizontal velocity, while the Zakharov/Craig-Sulem formulation of the water waves system involves the trace of the velocity potential at the surface. To our opinion, the main contribution of this work is the demonstration that when expressed in a different set of variables, the Green-Naghdi system possesses a structure which is very similar to the water waves one; and that one can take advantage of this fact to adapt the proof of the local well-posedness of the latter to the one of the former.

To be more precise, our work is based on another formulation of system~\eqref{GN-u}, namely
  \begin{equation}\label{GN-v}
   \left\{ \begin{array}{l}
   \partial_t\zeta+\nabla\cdot(h\u) =0,\\ \\
\big(\partial_t+\epsilon\u^\perp \curl\big) \v+\nabla\zeta+\frac\epsilon2\nabla(\abs{\u}^2)=\mu\epsilon\nabla \big(\R[h,\u]+\R_b[h,\beta b,\u]\big),
      \end{array}\right.
      \end{equation}
      where we denote $\curl (v_1,v_2)\eqdef \partial_1v_2-\partial_2v_1$ and $(u_1,u_2)^\perp\eqdef(-u_2,u_1)$,
   and
      \begin{align}\label{def-R}\R[h,\u]&\eqdef \frac{\u}{3h}\cdot\nabla(h^3\nabla\cdot\u)+\frac12 h^2(\nabla\cdot\u)^2, \\
      \label{def-Rb} \R_b[h,\beta b,\u]&\eqdef -\ \frac12  \left(\frac{\u}{h}\cdot\nabla\big(h^2(\beta\nabla b\cdot\u)\big)+ h(\beta\nabla b\cdot\u) \nabla\cdot\u+(\beta\nabla b\cdot\u)^2\right),
      \end{align}
      and $\v$ is defined (recalling the definition of the operator $\T$ in~\eqref{def-T}) by
   \begin{equation}\label{def-mfT}
  h\v \ = \  h\u+\mu h\T[h,\beta b]\u \eqdef\  \mfT[h,\beta b]\u.
   \end{equation}
In the following, we study system~\eqref{GN-v} as evolution equations for the variables $\zeta$ and $\v$, with $\u=\u[h,\beta b,\v]$ being uniquely defined (see Lemma~\ref{L.T-invertible} thereafter) by~\eqref{def-mfT}, and deduce the well-posedness of~\eqref{GN-u} from the one of~\eqref{GN-v}.

That system~\eqref{GN-v} is equivalent to~\eqref{GN-u} is certainly not straightforward, and we detail the calculations in Section~\ref{S.Formulations-comparison}.  Physically speaking, the variable $\v$ approximates (in the shallow-water regime, $\mu\ll1$) $\v_{\rm ww}=\nabla\psi_{\rm ww}$ where $\psi_{\rm ww}$ is the trace of the velocity potential at the surface, and thus system~\eqref{GN-v} is more directly comparable to the water waves system. In particular, the Hamiltonian structure of system~\eqref{GN-v}, as brought to light in~\cite{Holm88,Li02}, is a direct counterpart of the celebrated one of the water waves system, and we show in Appendix~\ref{S.Formulations-Hamiltonian} how system~\eqref{GN-v} can be quickly derived thanks to the Hamiltonian formalism. This allows to obtain preserved quantities of the system in a straightforward way (see~\cite{SiriwatKaewmaneeMeleshko16} and references therein). Most importantly for our purposes, this allows us to follow the strategy of the proof for the local existence of a solution to the water waves system in~\cite{Lannes}, and to obtain the corresponding local existence result for system~\eqref{GN-v}. More precisely, the variables $\zeta,\v$ allow to define the analogue of Alinhac's ``good unknowns''~\cite{Alinhac89} for the water waves system (see \eg~\cite{AlazardMetivier09}) on which energy estimates can be established. We then extend the analysis so as to prove the well-posedness of the Cauchy problem, in the sense of Hadamard.

After the completion of this work, it was pointed to us that the use of system~\eqref{GN-v} was not necessary (and may be viewed as an over-complicated strategy) to derive {\em a priori} energy estimates. For the sake of completeness, we sketch in Appendix~\ref{S.direct-estimates} the computations which provide such {\em a priori} estimates directly on system~\eqref{GN-u} and would yield an alternative proof of our main results. We still believe that the similarity of structure between the Green-Naghdi system and the one of the water waves system that we exhibit and exploit in this work is an interesting feature. It may serve as a pedagogical tool to get a grasp at some properties of the latter without technical difficulties related to the Dirichlet-to-Neumann operator. Incidentally, we do not claim the discovery of formulation~\eqref{GN-v}; see~\cite{LeGavrilyukHank10} and references therein, as well as in Appendix~\ref{S.Formulations-Hamiltonian}.

\subsection{Main results}

Let us now present the main results of this work.
Here and thereafter, we fix the parameters $\epsilon,\beta\geq0$ and $\mu\in(0,\mu^\star)$. The validity of the Green-Naghdi system as an asymptotic model for the water waves system stems from assuming $\mu\ll1$ while $\epsilon,\beta=\O(1)$ ---see~\cite{Lannes} for details--- but we do not make use of such restriction in this work. However, we shall
always assume that
\begin{equation}\label{cond-h0}
0<h_\star<h(\epsilon\zeta(x),\beta b(x))<h^\star<\infty, \qquad h(\epsilon\zeta,\beta b)\eqdef 1+\epsilon\zeta-\beta b.
\end{equation}
We work with the following functional spaces, defined for $n\in\NN$ by
\begin{align*}
&H^n\eqdef \{\zeta\in L^2(\RR^d) ,  & &\norm{ \zeta}_{H^n}^2\eqdef \sum_{|\alpha|=0}^n\norm{\partial^\alpha\u}_{L^2}^2<\infty\}, \qquad \dot{H}^n\eqdef \{b\in L^2_{\rm loc}(\RR^d) , \quad \nabla b\in (H^{n-1})^d\},\\
&X^n\eqdef \{\u\in L^2(\RR^d)^d ,  & &\norm{ \u}_{X^n}^2\eqdef \sum_{|\alpha|=0}^n\norm{\partial^\alpha\u}_{X^0}^2=\sum_{|\alpha|=0}^n\norm{\partial^\alpha\u}_{L^2}^2+\mu\norm{\partial^\alpha\nabla\cdot\u}_{L^2}^2<\infty\},\\
&Y^n\eqdef\{\v\in (X^0)^\prime ,   & &\norm{\v}_{Y^n}^2\eqdef \sum_{|\alpha|=0}^n\norm{\partial^\alpha\v}_{(X^0)^\prime}^2<\infty\}.
\end{align*}
Here, $\alpha\in\NN^d$ is a multi-index, $(X^0)^\prime$ is the topological dual space of $X^0$, endowed with the norm of the strong topology; and we denote by $\bra{ \v,\u}_{(X^0)^\prime}$ the $(X^0)^\prime-X^0$ duality bracket.

Given variables $\lambda_i\in\RR^+$, we denote $C(\lambda_1,\lambda_2,\dots)$ a multivariate polynomial with non-negative coefficients, and $F(\lambda_1,\lambda_2,\dots)$ a multivariate polynomial with non-negative coefficients and zero constant term. Since such notations are used for upper bounds and will take variables restricted to line segments, $C$ should be regarded as a constant and $F$ as a linear functional.

\begin{Theorem}[Well-posedness]\label{T.WP}
Let $N\geq 4$, $b\in \dot{H}^{N+2}$ and $(\zeta_0,\u_0)\in H^N\times X^N$ satisfying~\eqref{cond-h0} with $h_\star,h^\star>0$. Then there exists $T>0$ and a unique $(\zeta,\u)\in {C([0,T];H^N\times X^N)}$ satisfying~\eqref{GN-u} and $(\zeta,\u)\id{t=0}=(\zeta_0,\u_0)$. Moreover, one can restrict
\[ T^{-1}=C(\mu^\star,h_\star^{-1},h^\star)F(\beta\norm{\nabla b}_{H^{N+1}},\epsilon\norm{\zeta_0}_{H^N},\epsilon\norm{\u_0}_{X^N})>0\]
such that, for any $t\in[0,T]$,~\eqref{cond-h0} holds with $\tilde h_\star=h_\star/2,\tilde h^\star=2h^\star$, and
\[ \sup_{t\in[0,T]} \big(\norm{\zeta}_{H^N}^2+\norm{\u}_{X^N}^2 \big) \leq C(\mu^\star,h_\star^{-1},h^\star,\beta\norm{\nabla b}_{H^{N+1}},\epsilon\norm{\zeta_0}_{H^N},\epsilon\norm{\u_0}_{X^N}) \times \big(\norm{\zeta_0}_{H^N}^2+\norm{\u_0}_{X^N}^2 \big) ,
\]
and the map $(\zeta_0,\u_0)\in H^N\times X^N\mapsto  (\zeta,\u)\in {C([0,T];H^N\times X^N)}$ is continuous.
\end{Theorem}

The full justification of the Green-Naghdi system is a consequence of the local existence and uniqueness result for the  water waves system (\cite[Th.~4.16]{Lannes}), its consistency with the Green-Naghdi system (\cite[Prop.~5.8]{Lannes}), the well-posedness of the Cauchy problem for the Green-Naghdi system (Theorem~\ref{T.WP}) as well as a stability result ensuring the Lipschitz dependence of the error with respect to perturbations of the system (Proposition~\ref{P.stability}, thereafter). The following result improves~\cite[Th.~6.15]{Lannes} by the level of regularity required on the initial data.

\begin{Theorem}[Full justification]\label{T.justification} Let $N\geq 7$ and $(\zeta_{\rm ww},\psi_{\rm ww})\in C([0,T_{\rm ww}];H^N\times \dot{H}^{N+1})$ be a solution to the water waves system~\eqref{WW-psi} (recall that such solutions exist for a large class of initial data by~\cite[Th.~4.16]{Lannes}) such that~\eqref{cond-h0} holds with $b\in H^{N+2}(\RR^d)$. 

Denote $\zeta_0=(\zeta_{\rm ww})\id{t=0},\psi_0=(\psi_{\rm ww})\id{t=0}$, $h_0=1+\epsilon\zeta_0-\beta b$ and $\u_0=\mfT[h_0,\beta b]^{-1}(h_0\nabla\psi_0)$ and $ M= \sup_{t\in [0,T_{\rm ww}]} \big(\norm{\zeta_{\rm ww}}_{H^{N}}+\norm{\nabla\psi_{\rm ww}}_{H^{N}}\big)$.

Then there exists $T>0$ and $(\zeta_{\rm GN},\u_{\rm GN})\in C([0,T];H^{N}\times X^{N})$ unique strong solution to the Green-Naghdi system~\eqref{GN-u} with initial data $(\zeta_0,\u_0)$, by Theorem~\ref{T.WP}; and one can restrict
\[T^{-1}=C(\mu^\star,h_\star^{-1},h^\star)F(\beta\norm{ b}_{H^{N+2}},\epsilon M)>0\]
such that for any $t\in[0,\min(T,T_{\rm ww})]$,
\[\norm{\zeta_{\rm ww}-\zeta_{\rm GN}}_{H^{N-6}}+\norm{\nabla\psi_{\rm ww}-\v_{\rm GN}}_{Y^{N-6}}\leq  {\bf C}\, \mu^2\, t,\]
with ${\bf C}=C(\mu^\star,h_\star^{-1},h^\star,\beta\norm{ b}_{H^{N+2}},\epsilon M)$, and $\v_{\rm GN}$ is defined by~\eqref{def-mfT}.
\end{Theorem}

\begin{Remark}[Functional setting]
The functional spaces arise as natural energy spaces for the quasilinear structure of system~\eqref{GN-v} that we exhibit in this work. The regularity assumption $N\geq 4$ is most certainly not optimal. By the method of our proof, we are restricted to $N\in\NN$, and we cannot hope to obtain a lower threshold than in the pure quasilinear setting, namely $N>d/2+1$. That $N=3$ is not allowed when $d=2$ comes from technical limitations in various places.

Since our proof is based solely on the structural properties of the system and on energy estimates (in particular, no dispersive estimates are used), it may be adapted almost verbatim to the periodic situation. Although we have not checked all the technical details, we also expect that the strategy may be extended to the more general situation of Kato's uniformly local Sobolev spaces~\cite{Kato75}.
\end{Remark}
\begin{Remark}\label{R.irrot}
Applying the operator $\curl$ to $\eqref{GN-v}_2$, one observes the identity
 \[ \partial_t\curl\v+\epsilon\nabla\cdot (\u\curl\v)=0.\]
 Thus standard energy estimates on conservation laws~\cite{Benzoni-GavageSerre07} yield, for any $t\in[0,T]$,
 \[\norm{\curl \v}_{H^{N-1}}(t) \leq  \norm{\curl \v_0}_{H^{N-1}} \exp(\epsilon  \lambda t) \]
 with $\lambda\lesssim \norm{\u}_{H^N}\leq \norm{\u}_{X^N}$. Thus a byproduct of Theorem~\ref{T.WP} is that initial smallness of the ``generalized vorticity''~\cite{MilesSalmon85,GavrilyukTeshukov01}, $\curl\v$, propagates for large positive times ---and remains trivial if it vanishes initially, as in Theorem~\ref{T.justification}.
As already mentioned, the variable $\v$ physically approximates $\nabla\psi_{\rm ww}$ where $\psi_{\rm ww}$ is the trace of the velocity potential at the surface. It is therefore physically relevant (although not mathematically required) to assume
\[\curl\v=\curl \big(\u+\mu \T[h,\beta b]\u\big) \equiv   0.\]
This two-dimensional irrotationality condition is a direct consequence of the three-dimensional irrotationality assumption on the velocity flow inside the fluid layer. Outside of this irrotational framework, we believe that the Green-Naghdi system is not a valid model, and refer to~\cite{CastroLannesa} for a thorough discussion in this situation. 
\end{Remark}

\paragraph{Outline} The body of the paper is dedicated to the proof of Theorem~\ref{T.WP}. In Section~\ref{S.preliminary}, we provide some technical results concerning our functional spaces and the operator $\mfT$. We then exhibit the quasilinear system satisfied by the derivatives of any regular solution to~\eqref{GN-v} in Section~\ref{S.quasilinear}. This quasilinear formulation allows to obtain {\em a priori} energy estimates in Section~\ref{S.energy}. Finally, we make use of these energy estimates for proving the well-posedness of the Cauchy problem for system~\eqref{GN-v} in Section~\ref{S.WP}. The equivalence between formulation~\eqref{GN-v} and~\eqref{GN-u} is stated in Proposition~\ref{P.GNuvsGNv}. Theorems~\ref{T.WP} and~\ref{T.justification} follow as direct consequence of the results in Section~\ref{S.WP} and Proposition~\ref{P.GNuvsGNv}, and we complete the proofs in Section~\ref{S.Formulations-comparison}. In Appendix~\ref{S.direct-estimates}, we roughly sketch of how energy estimates could be obtained directly on formulation~\eqref{GN-u}. For motivation purposes, we also provide a quick discussion on the derivation and Hamiltonian formulation of system~\eqref{GN-v} in Appendix~\ref{S.Formulations-Hamiltonian}.

We conclude this section with some independent remarks.

\paragraph{Large time well-posedness and asymptotics} A very natural question in the oceanographic context concerns the large time asymptotic behavior of solutions to the Green-Naghdi system.
After a straightforward rescaling of~\eqref{GN-u}, the problem is naturally formulated in terms of solutions to the system
\begin{equation}\label{GN-u-LT}
\left\{\begin{array}{l}
\partial_t\zeta+\frac1{\epsilon}\nabla\cdot(h\u)=0,\qquad h\eqdef 1+\epsilon\zeta-\beta b,\\ \\
\big(\Id+\mu \T[h,\beta b]\big)\partial_t\u+\frac1\epsilon\nabla\zeta+(\u\cdot\nabla)\u+\mu \big(\Q[h,\u]+\Q_b[h,\beta b,\u]\big)=0.
\end{array}\right.
\end{equation}
Is the Cauchy problem for~\eqref{GN-u-LT} locally well-posed, uniformly with respect to the small parameter $\epsilon$? Is it globally well-posed for sufficiently small $\epsilon$? Can we exhibit ``averaged'' equations asymptotically describing the behavior of the solution?

This type of singular limit has been widely studied in particular in the context of the low Mach number limit; see \eg~\cite{Gallagher05,Alazard08} and references therein. As a matter of fact, when $\beta=\mu=0$, one recognizes the incompressible limit for the isentropic two-dimensional Euler equations, and it is tempting to elaborate on the analogy. One would then expect the solutions to~\eqref{GN-u-LT} to be asymptotically described (as $\epsilon\to0$) as the superposition of two components:
\begin{itemize}
\item The ``incompressible'' component, being defined as the solution to
\begin{equation}\label{GN-u-LT-incomp} \hspace{-5pt}
\left\{\begin{array}{l}
\nabla\cdot((1-\beta b)\u)=0,\\ \\
\big(\Id+\mu \T[1-\beta b,\beta b]\big)\partial_t\u+(\u\cdot\nabla)\u+\mu \big(\Q[1-\beta b,\u]+\Q_b[1-\beta b,\beta b,\u]\big)=-\nabla p,
\end{array}\right.
\end{equation}
where the ``pressure'' $p$ in the right-hand side of~$\eqref{GN-u-LT-incomp}_2$ is the Lagrange multiplier associated to the ``incompressibility'' constraint $\eqref{GN-u-LT-incomp}_1$.
 \item The ``acoustic'' component, being defined as the solution to
 \begin{equation}\label{GN-u-LT-acous}
 \left\{\begin{array}{l}
 \partial_t\zeta+\frac1{\epsilon}\nabla\cdot((1-\beta b)\u)=0,\\ \\
 \big(\Id+\mu \T[1-\beta b,\beta b]\big)\partial_t\u+\frac1\epsilon\nabla\zeta=0,
 \end{array}\right.
 \end{equation}
 with initial data satisfying $\curl (\u+\mu \T[1-\beta b,\beta b]\u)\id{t=0}=0$.
\end{itemize}
The system~\eqref{GN-u-LT-incomp} was derived in~\cite{CamassaHolmLevermore96,CamassaHolmLevermore97}, and is usually referred to as the {\em great-lake equations}. Its well-posedness, extending the theory concerning the two-dimensional incompressible Euler equations, was subsequently provided in~\cite{LevermoreOliverTiti96,Oliver97}. However, these authors apparently overlooked the role of the irrotationality assumption discussed in Remark~\ref{R.irrot}, as the only functions satisfying the constraint $\eqref{GN-u-LT-incomp}_1$ as well as the irrotationality condition $\curl (\u+\mu \T[1-\beta b,\beta b]\u)=0$ are in fact trivial.
In other words, in the irrotational framework that is the only one for which the Green-Naghdi system is rigorously justified, the incompressible (or rigid-lid) component vanishes; see also~\cite{Mesognon-Gireaua} for a similar discussion on the water waves system. One thus expects that the flow is asymptotically described by~\eqref{GN-u-LT-acous} only, in the limit $\epsilon\to0$.

However, when trying to adapt the usual strategy for rigorously proving such behavior, one immediately encounters a serious difficulty in the physically relevant situation of non-trivial topography, which transpires in the the fact that our lower bound for the existence time in Theorem~\ref{T.WP} depends on the size of the bottom variations in addition to the size of the initial data.
When transcribed to system~\eqref{GN-u-LT}, this means that we are not able to obtain a lower bound on the existence time of its solutions which is uniform with respect to $\epsilon$, unless $\beta=\O(\epsilon)$.

For the Saint-Venant system, that is setting $\mu=0$, Bresch and Métivier~\cite{BreschMetivier10} have obtained such a uniform lower bound without any restriction on the amplitude bathymetry. The strategy consists in estimating first the time derivatives of the solution, and then using the system to deduce estimates on space derivatives. A related strategy (in the sense that we look for operators commuting with the singular component of the system) amounts to remark that for any $n\in\NN$, one can control the $L^2$-norm of
\[  \zeta_n\eqdef  (\nabla\cdot (1-b) \nabla)^n\zeta,\qquad \u_n\eqdef  (\nabla (1-b) \nabla\cdot )^n \u\]
by exhibiting the quasilinear system satisfied by $(\zeta_n,\u_n)$ and applying simple energy estimates. This allows to control the $H^{2n}$-norm of $\zeta,\u$, provided that the initial data and bottom topography are sufficiently regular.
One expects a similar strategy to work for the water waves system (see~\cite{MelinandMesognon}; partial results have been obtained by the method of time
derivatives in~\cite{Mesognon-Gireau}), that is to control
\[  \zeta_n\eqdef  (\frac1\mu G^{\mu}[0,\beta b])^n\zeta,\qquad \psi_n\eqdef  (\frac1\mu G^{\mu}[0,\beta b])^n \psi\]
where $G^\mu$ is the Dirichlet-to-Neumann operator, recalled in Appendix~\ref{S.Formulations-Hamiltonian}. Since $G^\mu$ is an operator of order $1$, controlling $\zeta_n,\psi_n$ indeed allows to control higher regularities on $\zeta,\psi$. The strategy however fails for the Green-Naghdi system, as the corresponding operator, namely (see Appendix~\ref{S.Formulations-Hamiltonian})
\[ \frac1\mu G^{\mu}[0,\beta b] \bullet \approx -\nabla\cdot \Big( (1-\beta b)\mfT[1-\beta b,\beta b]^{-1}\big\{(1-\beta b)\nabla \bullet\big\}\Big)\]
is of order $0$. One could easily propose different systems that do not suffer from such a shortcoming, by adding the effect of surface tension as in~\cite{Mesognon-Gireau}, or modifying the system without hurting its consistency as in~\cite{Mesognon-Gireaub}. This is however out of the scope of the present work, and we leave the question of uniform lower bounds for the existence time of solutions to~\eqref{GN-u-LT} as an open problem.

\paragraph{Other models} We expect that our strategy may be of interest to other models for gravity waves. Since our result holds uniformly with respect to $\mu\in(0,1)$, we easily deduce the well-posedness of the Saint-Venant system. This result is of course a direct application of the standard theory on quasilinear hyperbolic systems~\cite{Benzoni-GavageSerre07}. In the other direction, it would be interesting to apply our strategy to the higher-order models derived by Matsuno~\cite{Matsuno16}, which enrich the Saint-Venant and Green-Naghdi systems with models of arbitrary high order, while preserving the structure of which we take advantage in this work. Similarly, one can derive models with improved frequency dispersion while preserving the structure of the Green-Naghdi model, by modifying the approximate Hamiltonian in Appendix~\ref{S.Formulations-Hamiltonian}. Such a strategy was applied by the authors in the one-dimensional and bilayer situation in~\cite{DucheneIsrawiTalhouk16}. Interestingly, since such models can be tuned to fit the dispersion relation of the water waves system, they do not suffer from the shortcoming described above, so that large-time well-posedness is expected to hold even in the presence of a non-trivial bathymetry.

Let us clarify however that our strategy does not require the Hamiltonian structure, but only exploits the ability to construct ``good unknowns'' through $\zeta,\v$ and their derivatives. Thus even models which are derived through careless approximations may preserve the quasilinear structure which is necessary for our energy estimates. Such is the case for the Boussinesq-Peregrine system (see~\cite{Peregrine67,Lannes,Mesognon-Gireaub}), which consists in a simplification of the Green-Naghdi system obtained by withdrawing contributions of size $\O(\mu\epsilon)$  while keeping $\O(\mu\beta)$ ones:
\begin{equation}\label{BP}
\left\{\begin{array}{l}
\partial_t\zeta+\nabla\cdot(h\u)=0,\\ \\
\partial_t\big(\u+\mu \T[1-\beta b,\beta b]\u\big) +\nabla\zeta+\epsilon(\u\cdot\nabla)\u=0.
\end{array}\right.
\end{equation}
Because the elliptic operator $\Id+\mu \T[1-\beta b,\beta b]$ does not depend on time, solving numerically the Boussinesq-Peregrine system is much less costly than solving the Green-Naghdi system; see~\cite{LeGavrilyukHank10,LannesMarche15}. It turns out that the proof of Theorem~\ref{T.WP} does extend to the Boussinesq-Peregrine system with straightforward modifications (all on the side of simplification). More precisely, we obtain the well-posedness of system~\eqref{BP} after applying
the change of variable
\begin{equation}
\label{BP-v}
\v\eqdef \u+\mu \T[1-\beta b,\beta b]\u
\end{equation}
and extracting the quasilinear system satisfied by $\partial^\alpha\zeta$ and $\partial^\alpha\v$.
\begin{Theorem}[Boussinesq-Peregrine]
Theorem~\ref{T.WP} holds replacing system~\eqref{GN-u} with system~\eqref{BP}. The full justification in the sense of Theorem~\ref{T.justification} holds as well, controlling the error as
\[\forall  t\in[0,\min(T,T_{\rm ww})], \qquad \norm{\zeta_{\rm ww}-\zeta_{\rm BP}}_{H^{N-6}}+\norm{\nabla\psi_{\rm ww}-\v_{\rm BP}}_{Y^{N-6}}\leq  {\bf C}\, (\mu^2+\mu\epsilon)\, t,\]
with $\zeta_{\rm BP},\u_{\rm BP}$ the unique strong solution to~\eqref{BP} and $\v_{\rm BP}$ defined by~\eqref{BP-v}.
\end{Theorem}
Compared with~\cite[Th.~2.1]{Mesognon-Gireaub}, our result provides a larger lower bound on the time of existence, and does not rely on the assumption $\epsilon=\O(\mu)$.

\section{Preliminary results}\label{S.preliminary}

In this section, we fix parameters $n\in\NN$, $\alpha\in\NN^d$
and denote $a\lesssim b$ for $a\leq C b$ where $C$ is a constant depending  (non-decreasingly) only on $n$, $d$, and possibly $|\alpha|$ and $\mu$. We denote $\bra{ A }_{n>r}=A$ if $n>r$ and $\bra{ A }_{n>r}=0$ otherwise, and $a\vee b=\max(a,b)$. The results are tailored for the dimension $d=2$ (through the repeated use of the continuous Sobolev embedding $H^2\subset L^\infty$ for instance) but hold as well when $d=1$. They are not meant to be sharp, but only sufficient for our needs. The call for ``tame'' estimates stems from the Bona-Smith technique enforcing continuity properties stated in Proposition~\ref{P.continuity}; rougher estimates would be sufficient for the existence and uniqueness of solutions, as in Proposition~\ref{P.existence}. 
\begin{Lemma}\label{L.embeddings}
One has the continuous embeddings $H^{n+1}(\RR^d)^d\subset X^n\subset H^n(\RR^d)^d$ and the corresponding $H^n(\RR^d)^d\subset Y^n \subset H^{n-1}(\RR^d)^d$. The following inequalities hold as soon as the right-hand side is finite:
\begin{align}\label{embed-X}
&\norm{\u}_{H^n}\leq \norm{\u}_{X^n}, &\qquad &\norm{\u}_{X^n}\lesssim \norm{\u}_{H^{n+1}},\\
\label{embed-Y}
&\norm{\v}_{H^{n-1}}\lesssim \norm{\v}_{Y^n}, &\qquad &\norm{\v}_{Y^n}\leq \norm{\v}_{H^n}.
\end{align}
We also have the  non-uniform continuous embedding 
\begin{equation}\label{embed-Z}
 \norm{\nabla f}_{Y^n}\lesssim \frac1{\sqrt\mu}\norm{f}_{H^n}, \qquad   \qquad \norm{\nabla \cdot \u }_{H^n}\lesssim \frac1{\sqrt\mu}\norm{\u}_{X^n}.
 \end{equation}
\end{Lemma}
\begin{proof}The continuous embeddings $H^{1}(\RR^d)^d\subset X^0\subset L^2(\RR^d)^d$ are straightforward, and the corresponding $L^2(\RR^d)^d\subset Y^0 \subset H^{-1}(\RR^d)^d$ follow by duality. 
The estimate~\eqref{embed-Z} with $n=0$ is easily checked, as for any $\u\in X^0$,
\[\abs{\bra{\nabla f,\u}_{(X^0)^\prime}}=\abs{\Par{f,\nabla\cdot\u}_{L^2}}\leq \frac1{\sqrt{\mu}}\norm{f}_{L^2}\norm{\u}_{X^0}.\]
The case $n\in \NN^\star $ is reduced to the case $n=0$ by considering $\partial^\alpha\u,\partial^\alpha\v,\partial^\alpha f$ with $0\leq |\alpha|\leq n$.
\end{proof}
\begin{Lemma}\label{L.products}
Let $f\in H^{2\vee n}(\RR^d)$ and $g\in H^n(\RR^d)$. Then $fg\in H^n(\RR^d)$
\begin{equation}\label{product-H-0} 
\norm{fg}_{H^n}\lesssim \norm{f}_{H^2}\norm{g}_{H^n} +\bra{ \norm{f}_{H^n}\norm{g}_{H^2}}_{n>2} .
\end{equation}
The above holds as well allowing exceptionally the value $n=-1$.

Let $f\in L^{\infty}(\RR^d) \cap \dot{H}^{3\vee n}(\RR^d)$ and $g\in H^n(\RR^d)$. Then $fg\in H^n(\RR^d)$ and
\begin{equation}\label{product-H} 
\norm{fg}_{H^n}\lesssim \big(\norm{f}_{L^{\infty}}+\norm{\nabla f}_{H^2}\big)\norm{g}_{H^n}+ \bra{ \norm{\nabla f}_{H^{n-1}}\norm{g}_{H^2}}_{n>2} .
\end{equation}
If, moreover, $f\geq f_0>0$, then $f^{-1}g\in H^n(\RR^d)$ and
\begin{equation}\label{product-H-1} 
  \norm{f^{-1}g}_{H^n}\leq C(f_0^{-1},\norm{f}_{L^{\infty}},\norm{\nabla f}_{H^2})\big(\norm{g}_{H^n}+ \bra{ \norm{\nabla f}_{H^{n-1}}\norm{g}_{H^2}}_{n>2}\big) .
  \end{equation}
Let $f\in L^{\infty}(\RR^d) \cap \dot{H}^{3 \vee n+1}(\RR^d)$ and $\u\in X^n$, with $n\in\NN$. Then $f\u\in X^n$ and
\begin{equation}\label{product-X} 
 \norm{f\u}_{X^n}\lesssim \big(\norm{f}_{L^{\infty}}+\norm{\nabla f}_{H^2}\big)\norm{\u}_{X^n}+ \bra{ \norm{\nabla f}_{H^{n}}\norm{\u}_{X^2}}_{n>2} .
   \end{equation}
Let $f\in L^{\infty}(\RR^d) \cap \dot{H}^{3\vee n}(\RR^d)$, and $\v\in Y^n$, with $n\in\NN$. One has $f\v\in Y^n$ and
\begin{equation}\label{product-Y} 
 \norm{f\v}_{Y^n}\lesssim \big(\norm{f}_{L^{\infty}}+\norm{\nabla f}_{H^2}\big)\norm{\v}_{Y^n}+ \bra{ \norm{\nabla f}_{H^{n-1}}\norm{\v}_{H^2}}_{n>2} .
   \end{equation}
\end{Lemma}
\begin{proof}
Estimate~\eqref{product-H-0} is well-known; see~\cite[Prop.~B.2]{Lannes} for instance. 

As for~\eqref{product-H}, the cases $n\in\{0,1,2\}$ are straightforward, using Leibniz rule and the continuous Sobolev embedding $H^2\subset L^\infty$. When $n>2$, we decompose Leibniz rule as follows: for any $1\leq |\alpha|\leq n$,
\[\partial^{\alpha}(fg)=f\partial^\alpha g+\sum_{\substack{\beta+\gamma=\alpha\\ |\beta|\geq 1,|\gamma|\geq 0}}\binom{\alpha}{\beta}(\partial^\beta f)(\partial^\gamma g).\] Since $|\beta|\geq 1$ and using the
standard bilinear estimate for (see \eg~\cite[Prop. 3.6]{TaylorIII}) , one has
\[\norm{(\partial^\beta f)(\partial^\gamma g)}_{L^2}\lesssim \norm{\nabla f}_{L^\infty}\norm{g}_{H^{|\beta|+|\gamma|-1}}+\norm{\nabla f}_{H^{|\beta|+|\gamma|-1}}\norm{g}_{L^\infty},\]
and the result follows.

Estimate~\eqref{product-H-1} is obtained in the same way, and using the induction hypothesis to control the contribution of $\norm{\nabla(f^{-1})}_{H^{|\beta|+|\gamma|-1}}$.

As for~\eqref{product-X}, we use
\[ \norm{ f\u}_{X^n}\lesssim  \norm{f \u}_{H^n}+\sqrt\mu\norm{ f\nabla\cdot\u}_{H^n}+\sqrt{\mu}\norm{\nabla f\cdot \u}_{H^n},\]
with product estimates~\eqref{product-H-0} and~\eqref{product-H}, as well as the continuous embedding~\eqref{embed-X} in Lemma~\ref{L.embeddings}.

Finally, we notice that $f\v\in Y^0=(X^0)^\prime$ and $\norm{f\v}_{Y^0}\lesssim \big(\norm{f}_{L^{\infty}}+\norm{\nabla f}_{H^{2}}\big)\norm{\v}_{(X^0)^\prime}$ since for any $\u\in X^0$,
\[ \abs{\bra{ f\v, \u}_{(X^0)^\prime}}=\abs{\bra{ \v,f\u}_{(X^0)^\prime}}\leq
\norm{\v}_{(X^0)^\prime}\norm{f \u}_{X^0}\lesssim
 \norm{\v}_{(X^0)^\prime}\big(\norm{f}_{L^{\infty}}+\norm{\nabla f}_{H^{2}}\big)\norm{\u}_{X^0}.\]
For $n\in\{0,1\}$, we write
 \[\norm{f\v}_{Y^{n+1}}\lesssim \norm{f\v}_{Y^n}+\sum_{|\alpha|=1}\norm{f \partial^\alpha \v}_{Y^n}+\norm{ (\partial^\alpha f) \v }_{Y^n}\lesssim \norm{f\v}_{Y^n}+\sum_{|\alpha|=1}\norm{f \partial^\alpha \v}_{Y^n}+\norm{ (\partial^\alpha f) \v }_{H^{n}},
 \]
 where we used~\eqref{embed-Y} in Lemma~\ref{L.embeddings} for the last estimate.
Estimate~\eqref{product-Y} follows by finite induction, using~\eqref{product-H-0} and again~\eqref{embed-Y}. For $n>2$, we use Leibniz rule:
 \[\norm{\partial^{\alpha}(f\v)}_{Y^0}\leq \norm{f\partial^\alpha \v}_{Y^0}+\sum_{\substack{\beta+\gamma=\alpha\\ |\beta|\geq 1,|\gamma|\geq 0}}\binom{\alpha}{\beta} \norm{(\partial^\beta f)(\partial^\gamma \v)}_{Y^0}.\]
 Estimate~\eqref{product-Y} follows from the above bilinear estimate 
\[\norm{(\partial^\beta f)(\partial^\gamma \v)}_{Y^0}\lesssim \norm{(\partial^\beta f)(\partial^\gamma \v)}_{L^2}\leq \norm{\nabla f}_{H^2}\norm{\v}_{H^{n-1}}+ \norm{\nabla f}_{H^{n-1}}\norm{\v}_{H^2}
\]
and using again the embedding~\eqref{embed-Y}. The proof is complete.
\end{proof}

Let us now exhibit elliptic estimates concerning the operator $\mfT$, defined in~\eqref{def-T},\eqref{def-mfT}.
\begin{Lemma}\label{L.T-invertible}
Let $b\in \dot{H}^3(\RR^d)$ and $h\in \dot{H}^2(\RR^d)$ be such that~\eqref{cond-h0} holds. Then $\mfT[h,\beta b]\in\mathcal{L}(X^0;(X^0)^\prime)$ and is symmetric:
\[\forall\u_1,\u_2\in X^0, \quad \bra{ \mfT[h,\beta b] \u_1,\u_2}_{(X^0)^\prime} \ = \ \bra{ \mfT[h,\beta b] \u_2 , \u_1 }_{(X^0)^\prime} .\]
Moreover, one has
\[\forall\u_1,\u_2\in X^0, \quad \abs{\bra{ \mfT[h,\beta b] \u_1,\u_2}_{(X^0)^\prime}} \leq C(\mu,h^\star,\beta\norm{\nabla b}_{L^\infty})\norm{\u_1}_{X^0}\norm{\u_2}_{X^0},\]
\[ \forall\u\in X^0, \quad \norm{\u}_{X^0}^2 \leq C(h_\star^{-1})\bra{ \mfT[h,\beta b] \u,\u}_{(X^0)^\prime} .\]
In particular, $\mfT[h,\beta b]:X^0\to (X^0)^\prime$ is a topological isomorphism and
\[ \forall\v\in(X^0)^\prime, \quad \norm{\mfT[h,\beta b]^{-1}\v}_{X^0}\leq C(h_\star^{-1})\norm{\v}_{(X^0)^\prime}.\]
\end{Lemma}
\begin{proof}
We establish the estimates for $\u,\u_1,\u_2\in \mathcal{S}(\RR^d)^d$ so that all the terms are well-defined, and the $((X^0)^\prime-X^0)$ duality product coincides with the $L^2$ inner product. The result for less regular functions is obtained by density of $\mathcal{S}(\RR^d)^d$ in $X^0$ and continuous linear extension.

By definition of $\mfT$ in~\eqref{def-mfT} and after integration by parts, one has
\begin{multline*}\Par{ \mfT[h,\beta b] \u_1,\u_2}_{L^2} =\int_{\RR^d} h \u_1\cdot \u_2+\frac\mu3 h^3(\nabla\cdot \u_1)(\nabla\cdot \u_2) \\
-\frac\mu{2} h^2\big((\nabla\cdot\u_2)(\beta\nabla b\cdot \u_1)+(\beta\nabla b\cdot\u_2)(\nabla\cdot \u_1)\big)  +\mu h(\beta\nabla b\cdot \u_1)(\beta\nabla b\cdot \u_2),
\end{multline*}
from which the symmetry is evident. The first estimate of the Lemma follows by Cauchy-Schwarz inequality. The second one is obvious when rewriting
\[\Par{ \mfT[h,\beta b] \u,\u}_{L^2}  = \int_{\RR^d} h \abs{\u}^2+\frac\mu{12} h^3\abs{\nabla\cdot \u}^2  +\frac\mu4 h\abs{ h\nabla\cdot\u-2\beta\nabla b\cdot \u}^2.\]

This shows that $\mfT[h,\beta b] :X^0\to(X^0)^\prime$ is continuous and coercive, so that the operator version of Lax-Milgram theorem ensures that $\mfT[h,\beta b] $ is an isomorphism. The continuity of the inverse follows from the coercivity of $\mfT[h,\beta b] $:
\[\norm{\u}_{X^0}^2\leq C(h_\star^{-1})\abs{\bra{ \mfT[h,\beta b]\u,\u}_{(X^0)^\prime}}\leq C(h_\star^{-1})\norm{\mfT[h,\beta b]\u}_{(X^0)^\prime}\norm{\u}_{X^0},\]
and setting $\u=\mfT[h,\beta b]^{-1}\v$ above.
\end{proof}

\begin{Lemma}\label{L.T-differentiable}
Let $b\in \dot{H}^{4}(\RR^d)$ and $h\in \dot{H}^{3}(\RR^d)$ be such that~\eqref{cond-h0} holds; and let $\v\in Y^n$. Then $\mfT[h,\beta b]^{-1}\v\in X^n$ and
\[ \norm{\mfT[h,\beta b]^{-1}\v}_{X^n}\leq C(\mu,h_\star^{-1},h^\star,\norm{\nabla h}_{H^{2}},\beta\norm{\nabla b}_{H^{3}})\Big(\norm{\v}_{Y^{n}}+\bra{\norm{\nabla h}_{H^{n-1}}\norm{\v}_{Y^2}}_{n>2}\Big).\]
\end{Lemma}
\begin{proof}
Let $\v\in \mathcal{S}(\RR^d)^d$, and denote for simplicity $\mfT\eqdef \mfT[h,\beta b]$. One has, for $\Lambda^n=(\Id-\Delta)^{n/2}$,
\[
\big[\Lambda^n,\mfT^{-1}\big]\v=-\mfT^{-1}\Lambda^n \mfT\mfT^{-1}\v+\mfT^{-1}\mfT\Lambda^n\mfT^{-1}\v
=-\mfT^{-1}\big[\Lambda^n,\mfT\big]\mfT^{-1}\v.\]
By definition of $\mfT$ and since $\Lambda^n$ commutes with space differentiation, we have for any $\u,\w\in \S(\RR^d)^d$,
\begin{multline*}\abs{\Par{\big[\Lambda^n,\mfT\big]\u,\w}_{L^2}}= \Par{\big[\Lambda^n,h\big]\u,\w}_{L^2}+\frac\mu3\Par{\big[\Lambda^n,h^3\big]\nabla\cdot\u,\nabla\cdot\w}_{L^2}\\
-\frac\mu{2}\Par{[\Lambda^n,h^2(\beta\nabla b)\cdot] \u,\nabla\cdot \w}_{L^2}-\frac\mu{2}\Par{[\Lambda^n, h^2(\beta\nabla b)]\nabla\cdot \u,\w}_{L^2}+\mu\Par{[\Lambda^n,h(\beta \nabla b)(\beta\nabla b)\cdot ] \u,\w}_{L^2}.
\end{multline*}
Commutator estimates (see \eg \cite[Corollary~B.9]{Lannes}) and product estimate~\eqref{product-H} yield
\[\abs{\Par{\big[\Lambda^n,\mfT\big]\u,\w}_{L^2}}\leq  C(\mu,h^\star,\norm{\nabla h}_{H^{2}},\beta\norm{\nabla b}_{H^{3}})\Big(\norm{\Lambda^{n-1}\u}_{X^0}+\bra{\norm{\nabla h}_{H^{n-1}}\norm{\Lambda^{2}\u}_{X^0}}_{n>2}\Big)\norm{\w}_{X^0}.\]
By density and continuity arguments, we infer that for any $\u\in X^{n-1}$, $\big[\Lambda^n,\mfT\big]\u\in (X^0)^\prime$ and
\[\norm{\big[\Lambda^n,\mfT\big]\u}_{(X^0)^\prime}\leq  C(\mu,h^\star,\norm{\nabla h}_{H^{2}},\beta\norm{\nabla b}_{H^{3}})\Big(\norm{\u}_{X^{n-1}}+\bra{\norm{\nabla h}_{H^{n-1}}\norm{\u}_{X^2}}_{n>2}\Big).\]
Combining the above and by Lemma~\ref{L.T-invertible}, we find 
\begin{align*}\norm{\mfT^{-1}\v}_{X^n}&= \norm{\mfT^{-1}\Lambda^n\v-\mfT^{-1}\big[\Lambda^n,\mfT\big]\mfT^{-1}\v}_{X^0}\leq C(h_\star^{-1})\norm{\Lambda^n\v-\big[\Lambda^n,\mfT\big]\mfT^{-1}\v}_{(X^0)^\prime}\\
&\leq   C(\mu,h_\star^{-1},h^\star,\norm{\nabla h}_{H^{2}},\beta\norm{\nabla b}_{H^{3}})\Big(\norm{\v}_{Y^n}+\norm{\mfT^{-1}\v}_{X^{n-1}}+\bra{\norm{\nabla h}_{H^{n-1}}\norm{\mfT^{-1}\v}_{X^2}}_{n>2}\Big).
\end{align*}
The result now follows by induction on $n\in\NN$, and by density of $ \mathcal{S}(\RR^d)^d$ in $Y^n$.
\end{proof}

We now exhibit first-order expansions of our operators which are used to obtain the quasilinear formulation of our system, in Proposition~\ref{P.quasi} below.
\begin{Lemma}\label{L.com-Hn}
For any $f,g\in \dot{H}^{3\vee n+ |\alpha|-1}(\RR^d)$ with $n\in\NN $ and $|\alpha|\geq 1$, one has
\[
\norm{\partial^\alpha(fg)-f\partial^\alpha g}_{H^n} \lesssim \norm{\nabla f}_{H^{2\vee n+|\alpha|-1}}\norm{ g}_{H^{n+|\alpha|-1}}.
\]
\end{Lemma}
\begin{proof}
Using Leibniz rule, we only have to estimate
\[\norm{(\partial^\beta f)(\partial^\gamma g)}_{L^2}\quad \text{ for $|\beta|+|\gamma|\leq n+|\alpha|$ and $|\beta|\geq 1$}.\] 
If $|\beta|=1$, then $|\gamma|\leq n+|\alpha|-1$, and we find by continuous Sobolev embedding $H^2\subset L^\infty$
\[ \norm{(\partial^\beta f)(\partial^\gamma g)}_{L^2}\leq \norm{\partial^\beta f}_{H^2}\norm{\partial^\gamma g}_{L^2}.\]
If $|\beta|=2$, then $|\gamma|\leq n+|\alpha|-2$, and standard product estimate (see \eg \cite[Prop. B.2]{Lannes}) yields
\[ \norm{(\partial^\beta f)(\partial^\gamma g)}_{L^2}\leq \norm{\partial^\beta f}_{H^1}\norm{\partial^\gamma g}_{H^1}.\]
If $|\beta|>2$, then $|\gamma|\leq n+|\alpha|-3$, and we find by continuous Sobolev embedding $H^2\subset L^\infty$
\[ \norm{(\partial^\beta f)(\partial^\gamma g)}_{L^2}\leq \norm{\partial^\beta f}_{L^2}\norm{\partial^\gamma g}_{H^2}.\]
This concludes the proof.
\end{proof}
\begin{Lemma}\label{L.com-Yn}
For any $f\in \dot{H}^{4\vee n+ |\alpha|-1}(\RR^d)$ and $\v\in Y^{n+ |\alpha|-1}$ with $n\in\NN$ and $|\alpha|\geq 1$, one has
\[
\norm{\partial^\alpha(f\v)-f\partial^\alpha \v}_{Y^n} \lesssim \norm{\nabla f}_{H^{3\vee n+|\alpha|-1}}\norm{ \v }_{Y^{n+|\alpha|-1}}.
\]
\end{Lemma}
\begin{proof}
Using Leibniz rule, we only have to estimate
\[\norm{(\partial^\beta f)(\partial^\gamma \v)}_{Y^0}\quad \text{ for $|\beta|+|\gamma|\leq n+|\alpha|$ and $|\beta|\geq 1$}.\] 
If $|\beta|=1$, then $|\gamma|\leq n+|\alpha|-1$, and we use the product estimate~\eqref{product-Y} in Lemma~\ref{L.products} to deduce
\[ \norm{(\partial^\beta f)(\partial^\gamma \v)}_{Y^0}\leq \norm{\partial^\beta f}_{H^3}\norm{\partial^\gamma \v}_{Y^0}.\]
For $|\beta|=2$, $|\beta|=3$ and $|\beta|>3$, we proceed as in the proof of Lemma~\ref{L.com-Hn}, and using~\eqref{embed-Y} in Lemma~\ref{L.embeddings}: by the continuous embedding $L^2\subset Y^0$ we fall into the $L^2$ setting, and the continuous embedding $Y^{n+|\alpha|-1}\subset H^{n+|\alpha|-2}$ allows to control $\norm{\partial^\gamma \v}_{H^{|\beta|-2}}\leq \norm{\partial^\gamma \v}_{H^{n+|\alpha|-2}}$.
\end{proof}
\begin{Lemma}\label{L.lin-Hn}
For any $f,g\in \dot{H}^{3\vee n+ |\alpha|-1}(\RR^d)$ with $n\in\{0,1\}$ and $|\alpha|\geq 1$, one has
\[
\norm{\partial^\alpha(fg)-g\partial^\alpha f-f\partial^\alpha g}_{H^n} \lesssim \norm{\nabla f}_{H^2}\norm{\nabla g}_{H^{n+|\alpha|-2}}+ \norm{\nabla f}_{H^{n+|\alpha|-2}}\norm{\nabla g}_{H^2}.
\]
\end{Lemma}
\begin{proof}
The result for $n=0$ follows from Leibniz rule:
\[\partial^{\alpha}(fg)-g\partial^\alpha f-f\partial^\alpha g=\sum_{\substack{\beta+\gamma=\alpha\\ |\beta|\geq 1,|\gamma|\geq 1}}\binom{\alpha}{\beta}(\partial^\beta f)(\partial^\gamma g),\]
and we use the standard bilinear estimate for (see~\cite[Prop. 3.6]{TaylorIII}) to estimate
\[\norm{(\partial^\beta f)(\partial^\gamma g)}_{L^2}\lesssim \norm{\nabla f}_{L^\infty}\norm{\nabla g}_{H^{|\beta|+|\gamma|-2}}+\norm{\nabla f}_{H^{|\beta|+|\gamma|-2}}\norm{\nabla g}_{L^\infty},\]
together with 
the continuous embedding $H^2\subset L^\infty$. The case $n=1$ is obtained in the same way after differentiating the above identity.
\end{proof}
\begin{Lemma}\label{L.lin-Yn}
For any $f\in \dot{H}^{4\vee n+ |\alpha|-1}(\RR^d),\v\in Y^{4\vee n+ |\alpha|-1}$ with $n\in\{0,1\}$ and $|\alpha|\geq 1$, one has
\begin{align*}
\norm{\partial^\alpha(f\v)-\v\partial^\alpha f-f\partial^\alpha \v}_{Y^n}\lesssim \norm{\nabla f}_{H^{3}}\norm{\v}_{Y^{n+|\alpha|-1}}+\norm{\nabla f}_{H^{ n+|\alpha|-2}}\norm{\v}_{Y^{4}}.
\end{align*}
\end{Lemma}
\begin{proof}If $\alpha=1$, then the result is obvious. For $|\alpha|\geq 2$, we use Leibniz rule:
\[\partial^{\alpha}(f\v )-\v\partial^\alpha f-f\partial^\alpha \v=\sum_{\substack{\beta+\gamma=\alpha\\ |\beta|\geq 1,|\gamma|\geq 1}}\binom{\alpha}{\beta}(\partial^\beta f)(\partial^\gamma \v).\]
When $|\beta|=1$, then by~\eqref{product-Y} in Lemma~\ref{L.products},
\[\norm{(\partial^\beta f)(\partial^\gamma \v)}_{Y^0}\lesssim \big(\norm{\partial^\beta f}_{L^{\infty}}+\norm{\partial^\beta \nabla f}_{H^{2}})\norm{\partial^\gamma \v}_{Y^0}\lesssim \norm{\nabla f}_{H^{3}}\norm{\v}_{Y^{|\alpha|-1}}.\]
When $|\beta|\geq 2$, we use the standard bilinear estimate as above:
\[\norm{(\partial^\beta f)(\partial^\gamma \v)}_{L^2}\lesssim \norm{\Lambda \nabla f}_{L^\infty}\norm{ \Lambda \v }_{H^{|\beta|+|\gamma|-3}}+\norm{\Lambda\nabla f}_{H^{|\beta|+|\gamma|-3}}\norm{\Lambda \v}_{L^\infty}.\]
This yields the desired estimate for $n=0$, by continuous embedding $H^2\subset L^\infty$ and~\eqref{embed-Y} in Lemma~\ref{L.embeddings}. The case $n=1$ is obtained similarly.
\end{proof}
The following Lemmata exhibit the {\em shape derivatives} of the operators $\mfT$ and $\mfT^{-1}$.
\begin{Lemma}\label{L.com-T}
Let $|\alpha|\geq 1$, $n\in\{0,1\}$. Let $b\in \dot{H}^{4\vee n+|\alpha|+1}(\RR^d)$ and $\zeta\in H^{3\vee n+|\alpha|-1}(\RR^d)$ be such that~\eqref{cond-h0} holds, and $\u\in X^{3\vee n+|\alpha|-1}$. Then one has
\begin{multline*}\norm{\partial^\alpha\big(\mfT[h,\beta b]\u\big)-\mfT[h,\beta b]\partial^\alpha\u-\epsilon\dd_h\mfT[h,\beta b](\partial^\alpha\zeta,\u)}_{Y^n}\\
\leq C(\mu,h^\star)F(\beta\norm{\nabla b}_{H^{3\vee n+|\alpha|}},\epsilon\norm{\nabla\zeta}_{H^{2}},\epsilon \norm{\u}_{X^{3}}) \Big( \norm{\nabla\zeta}_{H^{ n+|\alpha|-2}}+\norm{\u}_{X^{n+|\alpha|-1}}\Big)
\end{multline*}
with
\begin{equation}\label{def-dT}
 \dd_h\mfT[h,\beta b](f,\u)\eqdef f\u-\mu\nabla(h^2 f\nabla\cdot\u)+\mu\nabla\big(f h(\beta\nabla b)\cdot \u\big)-\mu f h (\beta\nabla b)\nabla\cdot \u.
 \end{equation}
\end{Lemma}
\begin{proof} Assume first that $\u\in \mathcal{S}(\RR^d)^d$, and denote
\[\r\eqdef \partial^\alpha\big(\mfT[h,\beta b]\u\big)-\mfT[h,\beta b]\partial^\alpha\u-\epsilon\dd_h\mfT[h,\beta b](\partial^\alpha\zeta,\u).\]
By definition of $\mfT$ in~\eqref{def-mfT}, integration by parts and since $\Lambda^n$ commutes with spatial differentiation, one has for any $\w\in X^0$,
\begin{align*}\bra{\Lambda^n\r,\w}_{(X^0)^\prime}&=\Par{\Lambda^n\{\partial^\alpha(h\u)-h\partial^\alpha\u-\epsilon\partial^\alpha\zeta\u\},\w}_{L^2}\\
&\qquad +\mu\frac13 \Par{ \Lambda^n\{\partial^\alpha(h^3\nabla\cdot\u)-h^3\partial^\alpha\nabla\cdot\u-3\epsilon h^2(\partial^\alpha\zeta)\nabla\cdot\u\},\nabla\cdot\w}_{L^2}\\
&\qquad -\mu\frac1{2}\Par{ \Lambda^n\{\partial^\alpha (h^2(\beta\nabla b)\cdot \u)-h^2(\beta\nabla b)\cdot\partial^\alpha\u-2\epsilon h(\partial^\alpha\zeta)(\beta\nabla b)\cdot\u \},\nabla\cdot\w}_{L^2}\\
&\qquad -\mu\frac1{2}\Par{\Lambda^n\{\partial^\alpha(h^2(\beta\nabla b)\nabla\cdot \u)-h^2(\beta\nabla b)\nabla\cdot \partial^\alpha\u-2\epsilon h(\partial^\alpha\zeta)(\beta\nabla b)\nabla\cdot \u\},\w}_{L^2}\\
&\qquad +\mu\Par{\Lambda^n\{\partial^\alpha(h\beta^2(\nabla b\cdot \u)\nabla b)-h\beta^2(\nabla b\cdot \partial^\alpha \u)\nabla b-\epsilon(\partial^\alpha\zeta)\beta^2(\nabla b\cdot \u)\nabla b\},\w}_{L^2}\\
&\eqdef \Par{\Lambda^n\r^1,\w}_{L^2}+\sqrt\mu\Par{\Lambda^n\r^2,\nabla\cdot\w}_{L^2}
\end{align*}
The two contributions are estimated using Lemmata~\ref{L.com-Hn} and~\ref{L.lin-Hn}, product estimates~\eqref{product-H-0} and~\eqref{product-H} and the continuous embedding $H^2\subset L^\infty$. For instance, by Lemma~\ref{L.lin-Hn}, one has
\[\norm{\partial^\alpha(\zeta\u)-\zeta\partial^\alpha\u-\partial^\alpha\zeta\u}_{H^n}\lesssim \norm{\nabla \zeta}_{H^{2}}\norm{\u}_{H^{n+|\alpha|-1}}+\norm{\nabla \zeta}_{H^{ n+|\alpha|-2}}\norm{\u}_{H^{3}}\]
and by Lemma~\ref{L.com-Hn},
\[\norm{\partial^\alpha( b\u)-b\partial^\alpha\u}_{H^n}\lesssim \norm{\nabla b}_{H^{2\vee n+|\alpha|-1}}\norm{\u}_{H^{n+|\alpha|-1}},\]
and~\eqref{embed-X} in Lemma~\ref{L.embeddings} allows to complete the estimate of the first contribution. The other terms are treated similarly, using first Lemma~\ref{L.com-Hn} to commute the differentiation operator with bottom contributions, product estimates~\eqref{product-H-0} and~\eqref{product-H} to estimate the commutator, and then Lemma~\ref{L.lin-Hn} to deal with surface contributions.

Altogether, one finds that $\norm{\r^1}_{H^n}$ and $\norm{\r^2}_{H^n}$ are estimated as in the statement.
By Cauchy-Schwarz inequality and duality, we deduce that $\r\in Y^n$ and $\norm{\r}_{Y^n}$ satisfies the same estimate. The result for $\u\in X^{3\vee n+|\alpha|-1}$ follows by density and continuous linear extension.
\end{proof}

\begin{Lemma}\label{L.com-T-1} Let $|\alpha|\geq 1$, $n\in\{0,1\}$. Let $b\in \dot{H}^{4\vee n+|\alpha|+1}(\RR^d)$ and $\zeta\in H^{3\vee n+|\alpha|-1}(\RR^d)$ be such that~\eqref{cond-h0} holds, and $\v\in Y^{3\vee n+|\alpha|-1}$. Then one has
\begin{multline*}\norm{\partial^\alpha\big(\mfT[h,\beta b]^{-1}\v\big)-\mfT[h,\beta b]^{-1}\partial^\alpha\v+\epsilon\mfT[h,\beta b]^{-1}\big\{\dd_h\mfT[h,\beta b](\partial^\alpha\zeta, \mfT[h,\beta b]^{-1}\v)\big\}}_{X^n}\\
\leq
C(\mu,h_\star^{-1},h^\star)F(\beta\norm{\nabla b}_{H^{3\vee n+|\alpha|}},\epsilon\norm{\nabla\zeta}_{H^{2}},\epsilon \norm{\u}_{X^{3}}) \Big( \norm{\nabla\zeta}_{H^{ n+|\alpha|-2}}+\norm{\v}_{Y^{n+|\alpha|-1}}\Big)
\end{multline*}
with $\dd_h\mfT[h,\beta b]$ defined in~\eqref{def-dT}.
\end{Lemma}
\begin{proof}
Denote $\mfT=\mfT[h,\beta b]$, and $\r\eqdef  \partial^\alpha\big(\mfT^{-1}\v\big)-\mfT^{-1}\partial^\alpha\v+\epsilon\mfT^{-1}\big\{\dd_h\mfT(\partial^\alpha\zeta, \mfT^{-1}\v)\big\}$. One has
\begin{align*}
\r&=\mfT^{-1}\mfT\partial^\alpha\big(\mfT^{-1}\v\big)-\mfT^{-1}\partial^\alpha \big(\mfT\mfT^{-1}\v\big)+\epsilon\mfT^{-1}\big\{\dd_h\mfT(\partial^\alpha\zeta, \mfT^{-1}\v)\big\}\\
& =-\mfT^{-1}\Big\{\partial^\alpha \big(\mfT \u\big)-\mfT\partial^\alpha\u-\epsilon \dd_h\mfT(\partial^\alpha\zeta,\u)\Big\}\eqdef -\mfT^{-1}\tilde\r
\end{align*}
with the notation $\u\eqdef \mfT^{-1}\v $. By Lemma~\ref{L.T-differentiable} and Lemma~\ref{L.com-T}, one has

\begin{align*}
\norm{\r}_{X^n}&\leq C(\mu,h_\star^{-1},h^\star,\norm{\nabla h}_{H^{2}},\beta\norm{\nabla b}_{H^{3}})\norm{\tilde\r}_{Y^{n}}\\
&\leq C(\mu,h_\star^{-1},h^\star)F(\beta\norm{\nabla b}_{H^{3\vee n+|\alpha|}},\epsilon\norm{\nabla\zeta}_{H^{2}},\epsilon \norm{\u}_{X^{3}}) \Big( \norm{\nabla\zeta}_{H^{ n+|\alpha|-2}}+\norm{\u}_{X^{n+|\alpha|-1}}\Big)
\end{align*}
and for $m= 3$ or $m= n+|\alpha|-1$,
\[ \norm{\u}_{X^m}\leq C(\mu,h_\star^{-1},h^\star,\beta\norm{\nabla b}_{H^{3}},\epsilon\norm{\nabla \zeta}_{H^{2}})\Big(\norm{\v}_{Y^{m}}+\bra{\norm{\nabla h}_{H^{m-1}}\norm{\v}_{Y^2}}_{m>2}\Big).\]
This concludes the proof.
\end{proof}
We conclude this section with the following result which is essential for estimating the difference between two approximate solutions (in later Proposition~\ref{P.stability} for instance).
\begin{Lemma}\label{L.diff-T-1} Let $n\in\NN$, $b\in \dot{H}^{4\vee n+1}(\RR^d)$ and $\zeta,\t\zeta\in H^{3\vee n}(\RR^d)$ be such that~\eqref{cond-h0} holds, and $\v\in Y^{2\vee n}$. Then, denoting
$h=1+\epsilon\zeta-\beta b$ and $\t h=1+\epsilon\t\zeta-\beta b$, one has
\[\norm{\mfT[h,\beta b]^{-1}\v-\mfT[\t h,\beta b]^{-1}\v}_{X^n}
\leq C(\mu,h_\star^{-1},h^\star,\norm{\nabla b}_{H^{3\vee n}},\norm{\nabla h}_{H^{2\vee {n-1}}},\norm{\nabla \t h}_{H^{2\vee {n-1}}})\times \norm{\v}_{Y^{2\vee n}}\norm{\zeta-\t\zeta}_{H^n}.\]
\end{Lemma}
\begin{proof}
We rewrite
\[\mfT[h,\beta b]^{-1}\v-\mfT[\t h,\beta b]^{-1}\v=\mfT[\t h,\beta b]^{-1}\big(\mfT[\t h,\beta b]-\mfT[h,\beta b]\big)\mfT[h,\beta b]^{-1}\v.\]
By Lemma~\ref{L.T-differentiable}, one has 
\[ \norm{\mfT[h,\beta b]^{-1}\v}_{X^n}\leq C(\mu,h_\star^{-1},h^\star,\norm{\nabla h}_{H^{2\vee n-1}},\beta\norm{\nabla b}_{H^{3}})\norm{\v}_{Y^{n}}.\]
Using the definition~\eqref{def-mfT}, one has
\begin{multline*}
\mfT[\t h,\beta b]\u-\mfT[h,\beta b]\u=\epsilon(\zeta-\t\zeta)\u-\frac\mu3 \nabla \big( (h^3-\t h^3)\nabla\cdot\u\big)\\
+\frac\mu2\nabla \big((h^2-\t h^2)(\beta\nabla b\cdot\u)\big)-\frac\mu2 (h-\t h)(\beta\nabla b)\nabla\cdot\u.
\end{multline*}
By~\eqref{embed-Y} and~\eqref{embed-X} in Lemma~\ref{L.embeddings} and~\eqref{product-H-0} in Lemma~\ref{L.products}, one has immediately
\[ \norm{\epsilon(\zeta-\t\zeta)\u}_{Y^n}\lesssim \epsilon\norm{\u}_{X^{2\vee n}}\norm{\zeta-\t\zeta}_{H^n}.\]
Using now~\eqref{embed-Z} in Lemma~\ref{L.embeddings}, we find
\begin{align*}
\mu\norm{\nabla \big( (h^3-\t h^3)\nabla\cdot\u\big)}_{Y^n}&\leq\sqrt{\mu}\norm{(h^3-\t h^3)\nabla\cdot\u}_{H^n}\leq \norm{\u}_{X^{2\vee n}}\norm{h^3-\t h^3}_{H^n}\\
&\leq \epsilon\, C(\norm{\nabla h}_{H^{2\vee n-1}},\norm{\nabla \t h}_{H^{2\vee n-1}})\norm{\u}_{X^{2\vee n}}\norm{\zeta-\t\zeta}_{H^n}.\end{align*}
The other terms are treated similarly, and we find
\[ \norm{\mfT[\t h,\beta b]\u-\mfT[h,\beta b]\u}_{Y^n}\leq C(\mu,h_\star^{-1},h^\star,\beta\norm{\nabla b}_{H^{3}},\norm{\nabla h}_{H^{2\vee n-1}},\norm{\nabla \t h}_{H^{2\vee n-1}})\norm{\u}_{X^{2\vee n}}\norm{\zeta-\t\zeta}_{H^n}.\]
The proof is completed when collecting the above and applying once again Lemma~\ref{L.T-differentiable}.
\end{proof}

\section{The quasilinear system}\label{S.quasilinear}

The result below is the key ingredient of our proof. We extract the quasilinear structure of system~\eqref{GN-v} in terms of ``good unknowns'', from which energy estimates will be deduced in subsequent Section~\ref{S.energy}. The strategy consists in differentiating system~\eqref{GN-v} several times, and estimate lower order contributions thanks to the formulas obtained in Section~\ref{S.preliminary}. However, in order not too break the structure of the first equation, we are led (as for the water waves system~\cite{Lannes}) to introduce an appropriate velocity variable which is a combination of the original variables and their derivatives.  

\begin{Proposition}\label{P.quasi}
Let $\alpha$ be a non-zero multi-index and $\zeta\in H^{4\vee|\alpha|}(\RR^d),b\in \dot{H}^{5\vee|\alpha|+2}(\RR^d)$ be such that~\eqref{cond-h0} holds, and $\v\in Y^{4\vee|\alpha|}$, satisfying~\eqref{GN-v}. Denote 
\begin{equation}\label{def-w}
\zeta_{(\alpha)}\eqdef \partial^\alpha\zeta \quad \text{ and } \quad \v_{(\alpha)}\eqdef \partial^\alpha\v -\mu\epsilon\, \nabla(w \partial^\alpha\zeta) \quad \text{ where } \quad w\eqdef -h \nabla\cdot\u+ \beta \nabla b\cdot \u.\end{equation}
Then $\zeta_{(\alpha)},\v_{(\alpha)}$ satisfy
\begin{equation}\label{GN-quasi}
   \left\{ \begin{array}{l}
   \partial_t\zeta_{(\alpha)}+\epsilon\nabla\cdot(\u \zeta_{(\alpha)}) +\nabla\cdot(h\u_{(\alpha)}) =  r_{(\alpha)}\\ \\
   (\partial_t + \epsilon \u^\perp \curl)\v_{(\alpha)} + \nabla \zeta_{(\alpha)} +\epsilon \nabla( \u\cdot\v_{(\alpha)})=  \r_{(\alpha)}
   \end{array}\right.
\end{equation}
where we denote
\[ \u\eqdef \mfT[h,\beta b]^{-1}(h\v) \quad \text{ and } \quad \u_{(\alpha)}\eqdef \mfT[h,\beta b]^{-1}(h\v_{(\alpha)}),\]
and $r_{(\alpha)},\r_{(\alpha)}$ satisfy the estimates
\begin{equation}\label{est-r1r2}
\norm{r_{(\alpha)}}_{L^2}+\norm{\r_{(\alpha)}}_{Y^0}\leq  {\bf F} \ \left(\norm{\zeta}_{H^{ |\alpha|}}+\norm{\v}_{Y^{ |\alpha|}}+\norm{\curl \v}_{H^{|\alpha|-1}}\right)
\end{equation}
with ${\bf F}=C(\mu,h_\star^{-1},h^\star)F\big(\beta\norm{\nabla b}_{H^{4\vee |\alpha|+1}},\epsilon\norm{\nabla\zeta}_{H^{3}},\epsilon\norm{\v}_{Y^{4}},\epsilon\norm{\curl\v}_{H^{3}}\big)$.

Moreover, if we denote $\t\zeta,\t\v$ satisfying the same assumptions and $\t r_{(\alpha)},\t \r_{(\alpha)}$ the corresponding residuals, then one has
\begin{equation}\label{est-r1r2-lipsch}
\norm{r_{(\alpha)}-\t r_{(\alpha)}}_{L^2}+\norm{\r_{(\alpha)}-\t\r_{(\alpha)}}_{Y^0}\leq \t {\bf F}\ \left(\norm{\zeta-\t\zeta}_{H^{|\alpha|}}+\norm{\v-\t\v}_{Y^{|\alpha|}}+\norm{\curl\v-\curl\t\v}_{H^{|\alpha|-1}}\right)
\end{equation}
with 
\begin{multline*}
\t {\bf F}=C(\mu,h_\star^{-1},h^\star)F\big(\beta\norm{\nabla b}_{H^{4\vee |\alpha|+1}},\epsilon\norm{\zeta}_{H^{4\vee|\alpha|}},\epsilon\norm{\v}_{Y^{4\vee|\alpha|}},\epsilon\norm{\curl\v}_{H^{3\vee|\alpha|-1}},\\
\epsilon\norm{\t\zeta}_{H^{4\vee|\alpha|}},\epsilon\norm{\t\v}_{Y^{4\vee |\alpha|}},\epsilon\norm{\curl\t\v}_{H^{3\vee|\alpha|-1}}\big).\end{multline*}
\end{Proposition}
\begin{Remark}
One recovers the quasilinear structure of the water waves system, as exhibited in~\cite[Sec.~4.2]{Lannes}, up to two differences. Firstly, the advection velocity is $\u$ instead of $U$ (using the notation introduced in Appendix~\ref{S.Formulations-Hamiltonian}). This was to be expected, comparing the formulation of the Green-Naghdi system~\eqref{GN-vter} with the corresponding formulation of the water waves system~\eqref{WW-vbis}. 

Secondly, from the comparison of the aforementioned formulations, one would expect~$\eqref{GN-quasi}_2$ to exhibit a component of the form
\[ \nabla (\mfa\zeta) \quad \text{ with } \quad \mfa\eqdef 1+\mu\epsilon \partial_t w +\mu\epsilon^2 V\cdot\nabla w ,\]
with the hyperbolicity condition $\mfa>0$ accounting for the Rayleigh-Taylor criterion, ${ (-\partial_z P)_{z=\epsilon\zeta}>0}$; see the discussion in~\cite[Sec. 4.3.5]{Lannes}.
In our system, we simply set $\mfa\equiv 1$ since the additional contributions can be discarded, thanks to our energy norm and the $\mu$ prefactor, by~\eqref{embed-Z} in Lemma~\ref{L.embeddings}. In other words, using the notation of the proof below, one can check that
\[\nabla\zeta_{(\alpha)}\sim_{Y^0}  \nabla \Big(\big(1+\mu\epsilon ( \partial_t +\epsilon \u\cdot\nabla)w\big)\zeta_{(\alpha)}\Big).\]
Thus the Rayleigh-Taylor criterion for water waves, which is automatically satisfied for small values of the parameter $\mu$, disappears in the Green-Naghdi system (for any values of $\mu$).
\end{Remark}
\begin{proof} 
Let us first remark that by~\eqref{product-Y} in Lemma~\ref{L.products}, one has for any $\v\in Y^n$ and $n\in\NN$,
\begin{equation}\label{bound-v}
\forall n\in\NN, \qquad \norm{h\v }_{Y^n} \leq C(h^\star)(1+\norm{\nabla h}_{H^2})\norm{\v}_{Y^n}+\bra{\norm{\nabla h}_{H^{n-1}}\norm{\v}_{Y^2}}_{n>2}.
\end{equation}
Using Lemma~\ref{L.T-differentiable}, it follows $\u= \mfT[h,\beta b]^{-1}(h\v) \in X^n$ and
\begin{equation}\label{bound-u}
\norm{\u}_{X^n}
\leq 
C(\mu,h_\star^{-1},h^\star,\beta\norm{\nabla b}_{H^{3}},\epsilon\norm{\nabla\zeta}_{H^{2}})\big(\norm{\v}_{Y^n}+\bra{\norm{\nabla h}_{H^{n-1}}\norm{\v}_{Y^2}}_{n>2}\big).
\end{equation}
We now enter into the detailed calculations. For simplicity, we denote
\begin{align*}
a\sim_{L^2} b \quad  \Longleftrightarrow \quad  & a-b= r,\\
a\sim_{Y^0} b \quad  \Longleftrightarrow \quad  & a-b= \r\end{align*}
with $\norm{ r}_{L^2}$ and $\norm{\r}_{Y^0}$ satisfying~\eqref{est-r1r2}. 
\medskip

\noindent{\em First equation.} 
We start by differentiating $\alpha$-times the first equation of~\eqref{GN-v}:
\[ \partial_t\zeta_{(\alpha)}+\partial^\alpha\nabla\cdot( h\u)=0.\]
By Lemma~\ref{L.lin-Hn} for the surface contribution and Lemma~\ref{L.com-Hn} for the bottom contribution (with $n=1$), and using~\eqref{embed-X} in Lemma~\ref{L.embeddings} and~\eqref{bound-u}, one finds
\[
\partial^\alpha\nabla\cdot (h\u)\sim_{L^2} \nabla\cdot (h\partial^\alpha \u)+  \epsilon \nabla\cdot( \u\partial^\alpha \zeta).
\]
Now, we use~\eqref{product-H} in Lemma~\ref{L.products},~\eqref{embed-X} in Lemma~\ref{L.embeddings}, Lemma~\ref{L.com-T-1} and~\eqref{bound-v}, so that
\[ \nabla\cdot (h\partial^\alpha \u ) \sim_{L^2} \nabla\cdot \big(h\mfT[h,\beta b]^{-1}\big\{ \partial^\alpha (h\v)-\epsilon\dd_h\mfT[h,\beta b](\partial^\alpha\zeta, \u)\big\} \big).\]
Using~\eqref{embed-X} and~\eqref{embed-Y} in Lemma~\ref{L.embeddings}, Lemma~\ref{L.T-differentiable}, and proceeding as above but with Lemma~\ref{L.lin-Yn} for the surface contribution and Lemma~\ref{L.com-Yn} for the bottom contribution, we find
\[\nabla\cdot \big(h \mfT[h,\beta b]^{-1}\big\{ \partial^\alpha (h\v)\big\}\big)\sim_{L^2} \nabla\cdot \big(h\mfT[h,\beta b]^{-1}\big\{ h\partial^\alpha \v + \epsilon(\partial^\alpha\zeta)\v\big\}\big).\]
In order to deal with the second contribution, we use the identity (see the definition~\eqref{def-mfT})
\[
   \v =\ \u -\frac{\mu}{3h}\nabla(h^3\nabla\cdot \u)+\frac\mu{2h}\nabla\big(h^2(\beta\nabla b)\cdot \u\big)-\frac\mu2 h(\beta\nabla b)\nabla\cdot \u+\mu \beta^2(\nabla b\cdot \u)\nabla b.
\]
By the extra $\mu$ prefactor, we can use now~\eqref{embed-Z}  and~\eqref{embed-Y} in Lemma~\ref{L.embeddings}, Lemmata~\ref{L.products} and~\ref{L.T-invertible} as well as~\eqref{bound-u} to deduce
\[ \epsilon \nabla\cdot \big(h\mfT[h,\beta b]^{-1}\big\{  (\partial^\alpha\zeta)\v\big\}\big) \sim_{L^2} \epsilon \nabla\cdot \big(h\mfT[h,\beta b]^{-1}\big\{  (\partial^\alpha\zeta)\u\big\}\big).\]
Finally, recalling~\eqref{def-dT}, and proceeding as above, we find
\begin{align*} 
&\epsilon\nabla\cdot \big(h \mfT[h,\beta b]^{-1}\big\{\dd_h\mfT[h,\beta b](\partial^\alpha\zeta, \u)\big\}\big) \\
&\qquad \sim_{L^2} \epsilon\nabla\cdot\big(h \mfT[h,\beta b]^{-1} \big\{(\partial^\alpha\zeta)\u-\mu \nabla(h^2 (\partial^\alpha\zeta )\nabla\cdot\u)+\mu \nabla\big(h(\partial^\alpha\zeta) (\beta\nabla b)\cdot \u\big)\big\}\big)\\
 &\qquad \sim_{L^2} \epsilon\nabla\cdot\big(h \mfT[h,\beta b]^{-1} \big\{(\partial^\alpha\zeta)\u-\mu h\nabla \big((\partial^\alpha\zeta )(h \nabla\cdot\u- \beta \nabla b\cdot \u)\big)\big\}\big).
 \end{align*}
Collecting the above information and using the definition of $\u_{(\alpha)}$, we obtain, as desired,
\begin{equation}\label{lin-eq1}
 \partial^\alpha\partial_t\zeta=-\partial^\alpha\nabla\cdot (h\u) \sim_{L^2} -\nabla\cdot (h \u_{(\alpha)} )-\epsilon \nabla\cdot( \u \zeta_{(\alpha)}).
 \end{equation}
\medskip

\noindent{\em Second equation.} 
Now we differentiate the second equation of~\eqref{GN-v}:
\[ \partial_t \partial^\alpha\v+\partial^\alpha\big(\nabla\zeta+ \epsilon \u^\perp \curl\v+\frac\epsilon2\nabla(\abs{\u}^2)\big)= \mu\epsilon\partial^\alpha\nabla\big(\R[h,\u]+\R_b[h,\beta b,\u]\big).\]
By~\eqref{embed-Y} and~\eqref{embed-X} in Lemma~\ref{L.embeddings}, Lemma~\ref{L.lin-Hn} with $n=1$ and~\eqref{bound-u}, one has
\[ \frac\epsilon2\partial^\alpha\nabla(\abs{\u}^2)\sim_{Y^0} \epsilon\nabla(\u\cdot\partial^\alpha \u).\]
Recalling the definition~\eqref{def-R}-\eqref{def-Rb}, we use~\eqref{embed-Z} in Lemma~\ref{L.embeddings}, Lemma~\ref{L.lin-Hn} with $n=0$ and Sobolev embedding $H^2\subset L^\infty$, product estimates~\eqref{product-H} and~\eqref{product-H-1} in Lemma~\ref{L.products} and finally~\eqref{bound-u}, to deduce
\[ \mu\epsilon\partial^\alpha \nabla\big(\R[h,\u]+\R_b[h,\beta b,\u]\big)\sim_{Y^0} \mu\epsilon\nabla\left( \frac{\u}{3h}\cdot \partial^\alpha \nabla(h^3\nabla\cdot\u)- \frac12  \frac{\u}{h}\cdot \partial^\alpha \nabla\big(h^2(\beta\nabla b\cdot\u)\big)\right).\]
Recalling once again the identity
\[
   \v =\ \u -\frac{\mu}{3h}\nabla(h^3\nabla\cdot \u)+\frac\mu{2h}\nabla\big(h^2(\beta\nabla b)\cdot \u\big)-\frac\mu2 h(\beta\nabla b)\nabla\cdot \u+\mu \beta^2(\nabla b\cdot \u)\nabla b,
\]
and proceeding as previously for the remainder terms, we deduce
\[ \epsilon\partial^\alpha \nabla\big(\frac12|\u|^2-\mu\R[h,\u]-\mu\R_b[h,\beta b,\u]\big)\sim_{Y^0} \epsilon \nabla(\u\cdot\partial^\alpha\v).\]
Now, by~\eqref{embed-Y} in Lemma~\ref{L.embeddings} and Lemma~\ref{L.lin-Hn} with $n=0$, Sobolev embedding $H^2\subset L^\infty$ and~\eqref{bound-u}, we find
\[ \epsilon\partial^\alpha (\u^\perp \curl \v)\sim_{Y^0} \epsilon \u^\perp (\curl\partial^\alpha\v),\]
and the combination yields
\[\partial_t\partial^\alpha \v  \sim_{Y^0} -\partial^\alpha\nabla \zeta-\epsilon \nabla(\u\cdot\partial^\alpha\v)-\epsilon \u^\perp (\curl\partial^\alpha\v).\]

In order to conclude, we consider
\[\partial_t (\v_{(\alpha)}-\partial^\alpha\v) =- \mu \epsilon \partial_t\nabla( w \partial^\alpha\zeta).\]
By continuous Sobolev embedding $H^2\subset L^\infty$, one has
\[\norm{\partial_t w }_{L^\infty}\leq C(h^\star,\beta\norm{\nabla b}_{H^2})(\norm{\partial_t \u}_{X^2}+\norm{\partial_t\zeta}_{H^2}).\]
Using the identity
\[\partial_t\u =\mfT[h,\beta b]^{-1}\partial_t(h\v)-\epsilon\mfT[h,\beta b]^{-1}\big\{\dd_h\mfT[h,\beta b](\partial_t\zeta, \u)\big\},\]
plugging the expressions of $\partial_t\zeta,\partial_t\v$ as given by~\eqref{GN-v}, and by
 Lemmata~\ref{L.embeddings} and~\ref{L.products}, Lemma~\ref{L.diff-T-1} with $n=2$, as well as~\eqref{bound-u}, we deduce
\[\partial_t (\v_{(\alpha)}-\partial^\alpha\v)\sim_{Y^0}- \mu\epsilon \nabla (w \partial_t \partial^\alpha \zeta).\]
Finally, using~\eqref{lin-eq1} and~\eqref{embed-Z} in Lemma~\ref{L.embeddings}, we find after straightforward manipulations and proceeding as above,
\[ \mu\epsilon \nabla (w \partial_t \partial^\alpha\zeta ) \sim_{Y^0} -\mu\epsilon^2 \nabla( w\nabla\cdot(\u \zeta_{(\alpha)}) )-\mu\epsilon \nabla(w\nabla\cdot(h\u_{(\alpha)}) )\sim_{Y^0} -\mu\epsilon^2 \nabla(\u\cdot \nabla(w\zeta_{(\alpha)}) ) .\]
Collecting the above information yields
\begin{align}
\partial_t\v_{(\alpha)}  &\sim_{Y^0} -\partial^\alpha\nabla\zeta-\epsilon \nabla( \u\cdot \partial^\alpha\v)
-\epsilon \u^\perp (\curl\partial^\alpha\v) +\mu\epsilon^2\nabla(\u\cdot \nabla(w\zeta_{(\alpha)})) \nonumber\\
&=-\partial^\alpha\nabla\zeta -\epsilon \nabla(\u\cdot \v_{(\alpha)})  -\epsilon \u^\perp (\curl\v_{(\alpha)}).\label{lin-eq2}
\end{align}
Estimates~\eqref{lin-eq1} and~\eqref{lin-eq2} provide the first estimate of the statement, namely~\eqref{est-r1r2}.
\medskip

The second estimate of the statement is obtained identically. Since all contributions on the remainders involve either products or the operator $\mfT^{-1}$, we can express the difference as a sum of terms of the same form but involving at least once $\zeta-\t\zeta$ or $\v-\t\v$, or estimated by Lemma~\ref{L.diff-T-1}. For instance, we find the corresponding estimate for~\eqref{bound-v} and~\eqref{bound-u} as follows. Denote $h=1+\epsilon\zeta-\beta b$ and $\t h=1+\epsilon\t\zeta-\beta b$. By~\eqref{embed-Y} in Lemma~\ref{L.embeddings} and~\eqref{product-H-0} and~\eqref{product-Y} in Lemma~\ref{L.products},
\begin{multline*}\norm{h\v-\t h\t\v}_{Y^n}\leq \norm{h(\v-\t\v)}_{Y^n}+\epsilon\norm{(\zeta-\t\zeta)\t\v}_{Y^n}\leq C(\mu,h^\star,\norm{\nabla h}_{H^{2\vee { n-1}}},\epsilon \norm{\t\v}_{Y^{3\vee n}})\big(\norm{\v-\t\v}_{Y^n}+\norm{ \zeta-\t\zeta}_{H^n}\big) .
\end{multline*}
Applying this to $\u= \mfT^{-1}[h,\beta b](h\v) $, $\t\u= \mfT^{-1}[\t h,\beta b](\t h\t\v)$ and using Lemma~\ref{L.T-differentiable} and Lemma~\ref{L.diff-T-1}, it follows 
\begin{multline*}
\norm{\u-\t\u }_{X^n}\leq 
C(\mu,h_\star^{-1},h^\star,\norm{\nabla b}_{H^{3\vee n}},\norm{\nabla h}_{H^{2\vee {n-1}}},\norm{\nabla \t h}_{H^{2\vee {n-1}}},\epsilon\norm{\v}_{Y^{2\vee n}},\epsilon\norm{\t\v}_{Y^{3\vee n}})\\
\times\big(\norm{\v-\t\v}_{Y^n}+\norm{ \zeta-\t\zeta}_{H^n}\big).
\end{multline*}
It is now a tedious but straightforward task to follow the steps of the proof and check that the estimate~\eqref{est-r1r2-lipsch} holds.
\end{proof}

\section{A priori energy estimates}\label{S.energy}

This section is dedicated to {\em a priori} energy estimates on the quasilinearized system~\eqref{GN-quasi} (or rather a mollified version; see below), which we make use in the proof of the existence and uniqueness of solutions to system~\eqref{GN-v} in subsequent Section~\ref{S.WP}. It is convenient to introduce the following notation, for $n\in\NN$:
\[\E^n(\zeta,\v)\eqdef  \norm{\zeta}_{H^n}^2+\norm{\v}_{Y^n}^2.\]
We wish to show that regularity induced by $\E^n$ is propagated by the flow of the Green-Naghdi system~\eqref{GN-v} provided $n$ is chosen sufficiently large. However, the natural energy of our system is determined by the symmetrizer associated with the quasilinear system~\eqref{GN-quasi} satisfied by $\zeta_{(\alpha)},\v_{(\alpha)}$. This leads us to define
\[\F^n(\zeta,\v)\eqdef \sum_{0\leq |\alpha|\leq n} \F[h,\beta b](\zeta_{(\alpha)},\v_{(\alpha)}).\]
with $\zeta_{(\alpha)},\v_{(\alpha)}$ given by~\eqref{def-w} and
\[\F[h,\beta b](\zeta_{(\alpha)},\v_{(\alpha)})\eqdef \norm{\zeta_{(\alpha)}}_{L^2}^2+\bra{\v_{(\alpha)}, h \u_{(\alpha)}}_{(X^0)^\prime},\]
where we recall that $\u_{(\alpha)}\eqdef \mfT[h,\beta b]^{-1}(h\v_{(\alpha)})$. By convention, we denote $\zeta_{(0)}\eqdef \zeta$ and $\v_{(0)}\eqdef \v$. In the following, we simply denote $\F(\zeta_{(\alpha)},\v_{(\alpha)})=\F[h,\beta b](\zeta_{(\alpha)},\v_{(\alpha)})$ for the sake of readability.

It is not obvious that controlling $\F(\zeta_{(\alpha)},\v_{(\alpha)})$ for $0\leq |\alpha|\leq n$ allows to control $\E^n(\zeta,\v)$ and conversely, and this is what we investigate in the following Lemmas.
\begin{Lemma}\label{L.EvsF0} 
Let $b\in\dot{H}^3(\RR^d)$ and $\zeta\in H^3(\RR^d)$ be such that~\eqref{cond-h0} holds, and let $\underline\zeta\in L^2(\RR^d),\underline\v\in Y^0$. Then $\underline\u\eqdef \mfT[h,\beta b]^{-1}(h\underline\v)\in X^0$ is uniquely defined, and one has
\[
{\bf C}^{-1}\,\norm{\underline\u}_{X^0}\leq  \norm{\underline\v}_{Y^0}\leq {\bf C}\, \norm{\underline\u}_{X^0},\]
and
\[{\bf C}^{-1}\, \F(\underline\zeta,\underline\v)\leq \E^0(\underline\zeta,\underline\v)\leq{\bf C}\, \F(\underline\zeta,\underline\v),\]
with ${\bf C}=C(\mu,h_\star^{-1},h^\star,\beta\norm{\nabla b}_{H^2},\epsilon\norm{\nabla\zeta}_{H^2})$.
\end{Lemma}
\begin{proof}
The estimates follow from Lemma~\ref{L.T-invertible} and~\eqref{product-H},\eqref{product-H-1} and~\eqref{product-Y} in Lemma~\ref{L.products}.
\end{proof}
\begin{Lemma}\label{L.EvsF}Let $n\in\NN^\star$ and $b\in\dot{H}^{4}(\RR^d)$, $\zeta\in H^{4\vee n}(\RR^d)$ be such that~\eqref{cond-h0} holds. Assume $\v\in Y^{4\vee n}$ and
define $\zeta_{(\alpha)},\v_{(\alpha)}$ as in~\eqref{def-w}. Then one has
\[\F^n(\zeta,\v) \leq C(\mu,h_\star^{-1},h^\star,\beta\norm{\nabla b}_{H^{3}},\epsilon^2\E^4(\zeta,\v))\ \E^n(\zeta,\v).\]
Conversely, if $\v\in Y^4$ is such that $\v_{(\alpha)}\in Y^0$ for any $1\leq|\alpha|\leq n$, then $\v\in Y^n$ and
\[\E^n(\zeta,\v)\leq C(\mu,h_\star^{-1},h^\star,\beta\norm{\nabla b}_{H^{3}},\epsilon^2\F^4(\zeta,\v) )\ \F^n(\zeta,\v) .\]

\end{Lemma}
\begin{proof} By Lemma~\ref{L.EvsF0}, it suffices to prove the estimates replacing $\F^n$ with
\[\tilde\E^n(\zeta,\v)\eqdef \sum_{0\leq |\alpha|\leq N} \E^0(\zeta_{(\alpha)},\v_{(\alpha)}).\]
Recall that by definition, we have
\[
\zeta_{(\alpha)}\eqdef \partial^\alpha\zeta \quad \text{ and } \quad \v_{(\alpha)}\eqdef \partial^\alpha\v-\mu\epsilon\, \nabla(w \partial^\alpha\zeta ) \quad \text{ where } \quad w\eqdef - h \nabla\cdot\u+ \beta \nabla b\cdot \u.\]
Thus by Lemma~\ref{L.products} and Lemma~\ref{L.T-differentiable}, one has for any $m\in\{0,1,2\}$,
\[
 \sqrt{\mu}\norm{w}_{H^m} \leq C(\mu,h^\star,\beta\norm{\nabla b}_{H^{2}},\epsilon\norm{\nabla\zeta}_{H^{2}}) \norm{\u}_{X^m}\leq C(\mu,h_\star^{-1},h^\star,\beta\norm{\nabla b}_{H^{3}},\epsilon\norm{\zeta}_{H^{3}})\norm{\v }_{Y^m}.
\]
Now, by~\eqref{embed-Z} in Lemma~\ref{L.embeddings} and continuous Sobolev embedding $H^2\subset L^\infty$, one has
\[\mu\epsilon\norm{\nabla(w\zeta_{(\alpha)})}_{Y^0}\lesssim \epsilon\sqrt\mu\norm{w\zeta_{(\alpha)}}_{L^2}\lesssim \epsilon\sqrt\mu\norm{w}_{H^2}\norm{\zeta_{(\alpha)}}_{L^2}.\]
We immediately deduce the first inequality of the statement.

For the converse, we first use that if $|\alpha|\leq 2$, one has
\[\mu\epsilon\norm{\nabla(w\zeta_{(\alpha)})}_{Y^0}\lesssim \epsilon\sqrt\mu\norm{w\zeta_{(\alpha)}}_{L^2} \lesssim \epsilon\sqrt\mu\norm{\zeta_{(\alpha)} }_{H^2}\norm{w}_{L^2}\leq \sqrt\mu\norm{w}_{L^2} \norm{\epsilon\zeta}_{H^4}.\]
This yields for $n\in\{0,1,2\}$,
\[ \norm{\v }_{Y^n}^2=\sum_{0\leq |\alpha|\leq n}\norm{\partial^\alpha \v}_{Y^0}^2\leq C(\mu,h_\star^{-1},h^\star,\beta\norm{\nabla b}_{H^{3}},\epsilon\norm{\zeta}_{H^{4}}) \times \left( \sum_{0\leq |\alpha|\leq n}  \norm{\v_{(\alpha)}}_{Y^0}^2 \right).\]
Then, for $|\alpha|>2$, we use as above 
\[\mu\epsilon\norm{\nabla(w\zeta_{(\alpha)})}_{Y^0}\lesssim \epsilon\sqrt\mu\norm{w\zeta_{(\alpha)}}_{L^2}\lesssim \epsilon\sqrt\mu\norm{w}_{H^2}\norm{\zeta_{(\alpha)}}_{L^2}.\]
By the above control on $w$ and the previously obtained estimate, we deduce that for $n\in\{3,\dots,N\}$, 
\[ \norm{\v }_{Y^n}^2=\sum_{0\leq |\alpha|\leq n}\norm{\partial^\alpha \v}_{Y^0}^2\leq C(\mu,h_\star^{-1},h^\star,\beta\norm{\nabla b}_{H^{3}},\epsilon\norm{\zeta}_{H^{4}},\epsilon\norm{\v}_{Y^2}) \times \left( \sum_{0\leq |\alpha|\leq n}  \E^0(\zeta_{(\alpha)},\v_{(\alpha)}) \right),\]
so that the second inequality of the statement follows.
\end{proof}

In the following, we provide energy estimates for regular solutions of a mollified version of the Green-Naghdi system~\eqref{GN-v}, namely
 \begin{equation}\label{GN-v-mol}
   \left\{ \begin{array}{l}
   \partial_t\zeta+J^\iota\nabla\cdot(h\u) =r^{\iota}\\ \\
\partial_t \v +J^\iota \left(\nabla\zeta+\epsilon\u^\perp \curl \v+\frac\epsilon2 \nabla\big(\abs{\u}^2\big)-\mu\epsilon\nabla\big(\R[h,\u]+\R_b[h,\beta b,\u]\big)\right)=\r^{\iota}
   \end{array}\right.
   \end{equation}
where $\u\eqdef \mfT[h,\beta b]^{-1}(h\v)$ as well as its associated quasilinear system (see Proposition~\ref{P.quasi})
\begin{equation}\label{GN-quasi-mol}
   \left\{ \begin{array}{l}
   \partial_t\zeta_{(\alpha)}+J^\iota \left(\epsilon\nabla\cdot(\u \zeta_{(\alpha)}) +\nabla\cdot(h\u_{(\alpha)}) \right)=  r^{\iota}_{(\alpha)}\\ \\
   \partial_t \v_{(\alpha)} +J^\iota\left( \nabla\zeta_{(\alpha)} +\epsilon\u^\perp \curl \v_{(\alpha)}+\epsilon \nabla(\u\cdot\v_{(\alpha)})\right)=  \r^{\iota}_{(\alpha)}
   \end{array}\right.
\end{equation}
where we denote $\u_{(\alpha)}\eqdef \mfT[h,\beta b]^{-1}(h \v_{(\alpha)})$.
Here and thereafter, $J^\iota$ is a Friedrichs mollifier, defined as the Fourier multiplier $J^\iota=\varphi(\iota |D|)$, where $\iota\in(0,1)$ is a parameter and $\varphi$ is a smooth function taking values in $[0,1]$, compactly supported and equal to $1$ in a neighborhood of the origin. By convention we write $J^0\equiv \Id$ and $J^\iota(u_1,u_2)=(J^\iota u_1,J^\iota_2)$.
The following properties will be used repeatedly:
\begin{Lemma} \label{L.J}Let $n\leq m\in\NN$, $\iota,\iota_1,\iota_2\in(0,1)$, and define the Fourier multiplier $J^\iota$ as above. Then for any $f\in H^n(\RR^d)$ and $\v\in Y^n$, one has $J^\iota f\in H^{m}(\RR^d)$, $J^\iota\v\in Y^m$ and
\begin{equation}
\label{J-limit} \norm{f-J^{\iota} f  }_{H^{n}}+ \iota^{-1}\norm{f-J^{\iota} f  }_{H^{n-1}}+\norm{\v-J^{\iota} \v  }_{Y^{n}}+ \iota^{-1}\norm{\v-J^{\iota} \v  }_{Y^{n-1}}\to 0 \quad (\iota\to 0).
\end{equation}
The above holds for $f_n\to f$ in $H^n$ and  $\v_n\to\v $ in $Y^n$, uniformly with respect to $n\in\NN$.
Moreover, there exists a constant $C$, depending only on $\varphi$, such that
\begin{align}
\label{J-product} \norm{J^\iota f}_{H^n}&\leq \norm{f}_{H^n}, & \norm{J^\iota \v}_{Y^n}&\leq \norm{\v}_{Y^n}, \\
\label{J-smooth}   \norm{ J^{\iota} f}_{H^{m}}&\leq C\, \iota^{n-m}\norm{f}_{H^{n}}, &  \norm{ J^{\iota} \v}_{Y^{m}}&\leq C\, \iota^{n-m}\norm{\v}_{Y^{n}},\\
\label{J-diff} \norm{(J^{\iota_2}-J^{\iota_1})f}_{H^{n-1}}&\leq C\, |\iota^2-\iota^1|\norm{f}_{H^{n}}, &  \norm{(J^{\iota_2}-J^{\iota_1})\v}_{H^{n-1}}&
\leq C\, |\iota^2-\iota^1|\norm{\v}_{Y^{n}} .
\end{align}
Moreover, if $ f\in H^{3}(\RR^d)$ and $g\in H^{n-1}(\RR^d)$, there exists $C_n$, depending only on $\varphi$ and $n$ such that
\begin{equation}
\label{J-com-H}   \norm{[J^\iota,f]g}_{H^{n}}\leq C\, \norm{\nabla f}_{H^{2\vee {n-1}}}\norm{g}_{H^{n-1}}.
\end{equation}

\end{Lemma}
\begin{proof}
The first estimates in $H^n$ are straightforward by Fourier analysis (see~\cite[Lemma~5]{BonaSmith75}). The estimates in $Y^n$ follow by duality and using that $J^\iota$ is symmetric and commutes with spatial derivatives. The last estimate is a consequence of~\cite[Corollary~B.9]{Lannes} and~\eqref{J-product}.
\end{proof}

Propositions~\ref{P.energy0} and~\ref{P.energy} below provide {\em a priori} energy inequalities for sufficiently regular solutions of~\eqref{GN-v-mol} and~\eqref{GN-quasi-mol}. Again, these estimates are uniform with respect to $\iota\in(0,1)$, and in particular hold as well for $J^\iota=\Id$.

\begin{Proposition}\label{P.energy0}
Let $b\in \dot{H}^4(\RR^d)$, $(\zeta,\v)\in C^1([0,T];H^3(\RR^d)^{1+d})$ be such that~\eqref{cond-h0} holds and satisfy system~\eqref{GN-v-mol} with residuals $(r^{\iota},\r^{\iota})\in C^0([0,T];L^2(\RR^d)^{1+d})$. Then for any $t\in[0,T]$,
\[ \frac{\dd}{\dd t}\F(\zeta,\v)\leq {\bf F}\, \F(\zeta,\v) +{\bf C}\, \big(\F(r^{\iota},\r^{\iota})  \F(\zeta,\v)\big)^{1/2} \]
with ${\bf F}=C(\mu,h_\star^{-1},h^\star)F(\beta\norm{\nabla b}_{H^3},\epsilon\norm{\partial_t\zeta}_{H^3},\epsilon\norm{\nabla\zeta}_{H^2},\epsilon\norm{\v}_{H^3})$, ${{\bf C}=C(\mu,h_\star^{-1},h^\star,\beta\norm{\nabla b}_{H^2},\epsilon\norm{\nabla\zeta}_{H^2})}$, uniformly with respect to $\iota\in(0,1)$.
\end{Proposition}
\begin{proof} The regularity assumptions on the data are sufficient to ensure that all the terms and calculations (including integration by parts) below are well-defined.
We test the first equation of~\eqref{GN-v-mol} against $\zeta\in C^0([0,T];L^2(\RR^d))$, and the second one against $h\u=h\mfT[h,\beta b]^{-1}(h\v)\in C^0([0,T];X^0)$. It follows
\[\frac12 \frac{\dd}{\dd t} \F(\zeta,\v)=A^{1}+A^{2,\iota}+A^{3,\iota}+A^{4}\]
with, using the symmetry of the operator $\mfT[h,\beta b]^{-1}$,
\begin{align*}
A^{1}&\eqdef \frac12\Par{\v,[\partial_t, h\mfT[h,\beta b]^{-1}h]\v}_{L^2}\\
A^{2,\iota}&\eqdef -\epsilon \Par{ J^\iota \big(\u^\perp \curl \v+\frac12 \nabla\big(\abs{\u}^2\big)-\mu\nabla\big(\R[h,\u]+\R_b[h,\beta b,\u])\big), h\u}_{L^2}\\
A^{3,\iota}&\eqdef-\Par{J^\iota \nabla\cdot(h\u) ,\zeta}_{L^2}-\Par{ J^\iota \nabla\zeta ,h\u}_{L^2}\\
A^{4}&\eqdef\Par{r^{\iota},\zeta}_{L^2}+\Par{ \r^{\iota}, h\u }_{L^2}.
\end{align*}
Since $J^\iota$ is symmetric and commutes with differential operators, one has $A^{3,\iota}\equiv 0$ after integrating by parts. The contributions of $A^{1}$ and $A^{4}$ are treated in details in the proof of Proposition~\ref{P.energy}; see~\eqref{bound-A1} and~\eqref{bound-A4}, below, with Lemma~\ref{L.EvsF0}. As for $A^{2,\iota}$, one may obtain rough estimates as follows.

By continuous Sobolev embedding $H^2\subset L^\infty$ and~\eqref{product-H} in Lemma~\ref{L.products},
\[ \norm{\epsilon\nabla\cdot ( h\u)}_{L^\infty}\lesssim\norm{h\times(\epsilon\u)}_{H^3}\leq C(h^\star)F(\beta\norm{\nabla b}_{H^2},\epsilon\norm{\nabla\zeta}_{H^2},\epsilon\norm{\u}_{H^3})\]
and therefore, by~\eqref{J-product} in Lemma~\ref{L.J} and Cauchy-Schwarz inequality,
\[\abs{\epsilon \Par{ J^\iota (\frac12\abs{\u}^2),\nabla\cdot ( h\u)}_{L^2}}\leq C(h^\star)F(\beta\norm{\nabla b}_{H^2},\epsilon\norm{\nabla\zeta}_{H^2},\epsilon\norm{\u}_{H^3})\norm{\u}_{L^2}^2.\]
Similarly, recalling the definition~\eqref{def-R}-\eqref{def-Rb}, one obtains as above
\[\sqrt\mu\epsilon\norm{\R[h,\u]+\R_b[h,\beta b,\u]}_{L^2}\leq C(\mu,h_\star^{-1},h^\star)F(\beta\norm{\nabla b}_{H^2},\epsilon\norm{\nabla\zeta}_{H^2},\epsilon\norm{\u}_{X^3})\norm{\u}_{X^0}\]
and
\[\sqrt\mu\norm{\nabla\cdot ( h\u)}_{L^2}\leq \sqrt\mu\norm{\nabla h}_{L^\infty}\norm{\u}_{L^2}+\sqrt\mu\norm{h}_{L^\infty}\norm{\nabla\cdot\u}_{L^2}\leq C(\mu,h^\star,\beta\norm{\nabla b}_{H^2},\epsilon\norm{\nabla\zeta}_{H^2})\norm{\u}_{X^0}.\]
Finally, one has by Cauchy-Schwarz inequality and continuous Sobolev embedding $H^2\subset L^\infty$
\[ \abs{ \Par{J^\iota (\u^\perp \curl \v),h\u}_{L^2}}\leq  C(h^\star)\norm{\curl\v}_{H^2}\norm{\u}_{L^2}^2.\]
It follows, by~\eqref{embed-X} in Lemma~\ref{L.embeddings},
\[\abs{A^{2,\iota}}\leq C(\mu,h_\star^{-1},h^\star)F(\beta\norm{\nabla b}_{H^2},\epsilon\norm{\nabla\zeta}_{H^2},\epsilon\norm{\u}_{X^3},\epsilon\norm{\curl\v}_{H^2}) \norm{\u}_{X^0}^2.\]
We conclude by Lemma~\ref{L.T-differentiable} and~\eqref{product-Y} in Lemma~\ref{L.products}, as well as Lemma~\ref{L.EvsF0}.
\end{proof}

\begin{Proposition}\label{P.energy}
Let $b\in\dot{H}^4$, $(\zeta,\u)\in C^0([0,T];H^4(\RR^d)^{1+d})\cap C^1([0,T];H^3(\RR^d)^{1+d})$ be such that~\eqref{cond-h0} holds, and $(\zeta_{(\alpha)},\v_{(\alpha)})\in C^0([0,T];H^1(\RR^d)^{1+d})$ satisfying system~\eqref{GN-quasi-mol} with remainder terms $(r^{\iota}_{(\alpha)},\r^{\iota}_{(\alpha)}) \in C^0([0,T];L^2(\RR^d)^{1+d})$. Then one has for any $t\in[0,T]$,
\[ \frac{\dd}{\dd t}\F(\zeta_{(\alpha)},\v_{(\alpha)})\leq {\bf F}\,  \F(\zeta_{(\alpha)},\v_{(\alpha)}) +{\bf C}\, \big(\F(r^{\iota}_{(\alpha)},\r^{\iota}_{(\alpha)})  \F(\zeta_{(\alpha)},\v_{(\alpha)}) \big)^{1/2}\]
with ${\bf F}=C(\mu,h_\star^{-1},h^\star)F(\beta\norm{\nabla b}_{H^3},\epsilon\norm{\partial_t\zeta}_{H^3},\epsilon\norm{\nabla\zeta}_{H^3},\epsilon\norm{\u}_{H^3})$, ${\bf C}=C(\mu,h_\star^{-1},h^\star,\beta\norm{\nabla b}_{H^2},\epsilon\norm{\nabla\zeta}_{H^2})$, uniformly with respect to $\iota\in(0,1)$.
\end{Proposition}
\begin{proof} The regularity assumptions on the data are sufficient to ensure that all the terms and calculations (including integration by parts) below are well-defined.
We test the first equation of~\eqref{GN-quasi-mol} against $\zeta_{(\alpha)}\in C^0([0,T];L^2)$ and the second against $h\u_{(\alpha)}=h\mfT[h,\beta b]^{-1}(h\v_{(\alpha)})\in C^0([0,T];X^0)$. It follows
\[\frac12 \frac{\dd}{\dd t} \F(\zeta_{(\alpha)},\v_{(\alpha)})=A_{(\alpha)}^{1}+A_{(\alpha)}^{2,\iota}+A_{(\alpha)}^{3,\iota}+A_{(\alpha)}^{4}\]
with, using the symmetry of the operator $\mfT[h,\beta b]^{-1}$,
\begin{align*}
A_{(\alpha)}^{1}&\eqdef \frac12\Par{\v_{(\alpha)},[\partial_t, h\mfT[h,\beta b]^{-1}h]\v_{(\alpha)}}_{L^2}\\
A_{(\alpha)}^{2,\iota}&\eqdef-\epsilon\Par{J^\iota\nabla\cdot(\u \zeta_{(\alpha)}),\zeta_{(\alpha)}}_{L^2}-\epsilon \Par{ J^\iota ( \u^\perp \curl \v_{(\alpha)}+ \nabla(\u\cdot\v_{(\alpha)})), h\u_{(\alpha)}}_{L^2}\\
A_{(\alpha)}^{3,\iota}&\eqdef-\Par{J^\iota \nabla\cdot(h\u_{(\alpha)}) ,\zeta_{(\alpha)}}_{L^2}-\Par{ J^\iota \nabla\zeta_{(\alpha)} , h\u_{(\alpha)}}_{L^2}\\
A_{(\alpha)}^{4}&\eqdef\Par{r^{1,\iota}_{(\alpha)},\zeta_{(\alpha)}}_{L^2}+\Par{ r^{2,\iota}_{(\alpha)}, h\u_{(\alpha)} }_{L^2}.
\end{align*}
By Lemma~\ref{L.EvsF0}, there remains to estimate each contribution in terms of $\norm{\v_{(\alpha)}}_{Y^0},\norm{\u_{(\alpha)}}_{X^0},\norm{\zeta_{(\alpha)}}_{L^2}$.
\smallskip

\noindent{\em Estimate on $A_{(\alpha)}^1$.} We use the explicit formula for the commutator
\begin{align*}& \hspace*{-1cm}\frac12\Par{\v_{(\alpha)},[\partial_t, h\mfT[h,\beta b]^{-1}h]\v_{(\alpha)}}_{L^2}\\
&=\Par{\v_{(\alpha)},(\partial_t h) \mfT[h,\beta b]^{-1}\{h\v_{(\alpha)}\}}_{L^2}+\frac12\Par{\v_{(\alpha)},h[\partial_t, \mfT[h,\beta b]^{-1}]h\v_{(\alpha)}}_{L^2}\\
&=\Par{(\partial_t h)\v_{(\alpha)}, \u_{(\alpha)}}_{L^2}-\frac12\Par{\u_{(\alpha)},[\partial_t, \mfT[h,\beta b]]\u_{(\alpha)}}_{L^2}\\
&=\Par{(\partial_t h)\v_{(\alpha)}, \u_{(\alpha)}}_{L^2}-\frac12\Par{\u_{(\alpha)},(\partial_t h)\u_{(\alpha)}}_{L^2}-\frac\mu2\Par{\nabla\cdot \u_{(\alpha)},(h^2\partial_t h)\nabla\cdot\u_{(\alpha)}}_{L^2}\\
&\quad + \mu \Par{ (h\partial_t h) \nabla\cdot \u_{(\alpha)}, (\beta\nabla b)\cdot \u_{(\alpha)}}_{L^2}
\end{align*}
where we used the symmetry of $\mfT[h,\beta b]^{-1}$ and the definitions~\eqref{def-mfT},\eqref{def-T}. Since $\partial_t h=\epsilon\partial_t\zeta$,
we deduce by~\eqref{product-Y} in Lemma~\ref{L.products}, continuous Sobolev embedding $H^2\subset L^\infty$ and Cauchy-Schwarz inequality,
\begin{equation}\label{bound-A1} 
\abs{A_{(\alpha)}^1}\leq  C(\mu,h^\star,\beta\norm{\nabla b}_{H^2})\times \big(\norm{\epsilon \partial_t\zeta}_{H^3} \norm{\u_{(\alpha)}}_{X^0}\norm{\v_{(\alpha)}}_{Y^0}+\norm{\epsilon \partial_t\zeta}_{H^2} \norm{\u_{(\alpha)}}_{X^0}^2\big).\end{equation}
\smallskip

\noindent{\em Estimate on $A_{(\alpha)}^{2,\iota}$.}
First remark that, since $J^\iota$ is symmetric and commutes with differential operators, one has after integrating by parts
\[A_{(\alpha)}^{2,\iota,i}\eqdef-\epsilon\Par{J^\iota\nabla\cdot(\u \zeta_{(\alpha)}),\zeta_{(\alpha)}}_{L^2}
=\epsilon \Par{\zeta_{(\alpha)},\u\cdot \nabla J^\iota \zeta_{(\alpha)}}_{L^2}\]
and therefore, averaging the left-hand side and the right-hand side,
\[A_{(\alpha)}^{2,\iota,i}\eqdef-\frac12\epsilon\Par{J^\iota (\zeta_{(\alpha)}\nabla\cdot\u),\zeta_{(\alpha)}}_{L^2}
-\frac12\epsilon \Par{\zeta_{(\alpha)},[J^\iota,\u\cdot] \nabla  \zeta_{(\alpha)}}_{L^2}.\]
 By the product and commutator estimates~\eqref{J-product} and~\eqref{J-com-H} in Lemma~\ref{L.J}, and applying Cauchy-Schwarz inequality and continuous Sobolev embedding $H^2\subset L^\infty$, we get
\begin{equation}\label{bound-A2i} 
\abs{A_{(\alpha)}^{2,i}}\lesssim \epsilon\norm{ \u}_{H^3} \times \norm{\zeta_{(\alpha)}}_{L^2}^2,
\end{equation}
uniformly with respect to $\iota\in(0,1)$.

The second contribution, namely
\[A_{(\alpha)}^{2,\iota,ii}\eqdef \Par{ J^\iota ( \u^\perp \curl \v_{(\alpha)}+ \nabla(\u\cdot\v_{(\alpha)})), h\mfT[h,\beta b]^{-1}(h\v_{(\alpha)})}_{L^2} ,\]
is by far the most involved (notice that in the case of the water waves system, the corresponding term requires a specific attention as well; see~\cite[Prop.~3.30]{Lannes}). We prove in Lemma~\ref{L.commut-Gtransport}, below, that
\begin{equation}\label{bound-A2ii} 
\abs{A_{(\alpha)}^{2,\iota,ii}}\leq  C(\mu,h_\star^{-1},h^\star)F(\beta\norm{\nabla b}_{H^3},\epsilon\norm{\nabla\zeta}_{H^3},\epsilon\norm{\u}_{H^3}) \times \norm{\v_{(\alpha)}}_{Y^0}^2.
\end{equation}
\smallskip

\noindent{\em Estimate on $A_{(\alpha)}^{3,\iota}$.} Thanks to our choice of the symmetrizer and since $J^\iota$ is symmetric and commutes with differential operators, one has after integrating by parts
\begin{equation}\label{bound-A3} A_{(\alpha)}^{3,\iota}\equiv 0.\end{equation}
\smallskip

\noindent{\em Estimate on $A_{(\alpha)}^4$.} By Cauchy-Schwarz inequality and~\eqref{product-X} in Lemma~\ref{L.products}, one has
\begin{equation}\label{bound-A4} 
\abs{A_{(\alpha)}^4}\leq C(h^\star,\beta\norm{\nabla b}_{H^2},\epsilon\norm{\nabla\zeta}_{H^2}) \times  \big(\norm{\zeta_{(\alpha)}}_{L^2}^2+\norm{\u_{(\alpha)}}_{X^0}^2\big)^{1/2}\E^0(r^1_{(\alpha)},r^2_{(\alpha)}).\end{equation}
\smallskip

Altogether, estimates~\eqref{bound-A1}--\eqref{bound-A4}, with Lemma~\ref{L.EvsF0} yield the desired result.
\end{proof}

We conclude this section with the essential estimate in the proof of Proposition~\ref{P.energy}.
\begin{Lemma}\label{L.commut-Gtransport} Under the assumptions of Proposition~\ref{P.energy}, one has
\[\epsilon\abs{\Par{ J^\iota ( \u^\perp \curl \v_{(\alpha)}+ \nabla(\u\cdot\v_{(\alpha)})), h\mfT[h,\beta b]^{-1}(h\v_{(\alpha)})}_{L^2} } \leq {\bf F}\,\norm{\v_{(\alpha)}}_{Y^0}^2,\]
with ${\bf F}=C(\mu,h_\star^{-1},h^\star)F(\beta\norm{\nabla b}_{H^3},\epsilon\norm{\nabla\zeta}_{H^3},\epsilon\norm{\u}_{H^3})$, uniformly with respect to $\iota\in(0,1)$. 
\end{Lemma}
\begin{proof} 
The regularity assumptions on the data are sufficient to ensure that all the terms and calculations (including integration by parts) below are well-defined. We denote $\u_{(\alpha)}\eqdef \mfT[h,\beta b]^{-1}(h\v_{(\alpha)})$ and use the identity valid for any $U$ and $V$ sufficiently regular two-dimensional vector fields
\begin{equation} \label{identity}  \nabla(U\cdot V)=(U\cdot\nabla)V+(V\cdot\nabla)U-V^\perp \curl U-U^\perp \curl V.\end{equation}
Thus we have
\begin{align*}B^\iota&\eqdef \Par{J^\iota(  \u^\perp \curl \v_{(\alpha)}+ \nabla(\u\cdot\v_{(\alpha)})),h\u_{(\alpha)}}_{L^2} \\
&=\Par{(\u\cdot\nabla)\v_{(\alpha)} +(\v_{(\alpha)}\cdot\nabla)\u-\v_{(\alpha)}^\perp \curl \u,J^\iota(h\u_{(\alpha)})}_{L^2}\\
&=-\Par{(\u\cdot\nabla h)\v_{(\alpha)},J^\iota\u_{(\alpha)}}_{L^2}+\Par{(\u\cdot\nabla)(h\v_{(\alpha)}) +(h\v_{(\alpha)}\cdot\nabla)\u-(h\v_{(\alpha)})^\perp \curl \u, J^\iota\u_{(\alpha)}}_{L^2}\\
&\quad +\Par{(\u\cdot\nabla)\v_{(\alpha)} +(\v_{(\alpha)}\cdot\nabla)\u-\v_{(\alpha)}^\perp \curl \u,[J^\iota,h]\u_{(\alpha)}}_{L^2}.
\end{align*}
\medskip

The first term in $B^\iota$ is controlled by Cauchy-Schwarz inequality,~\eqref{product-H-0} and~\eqref{product-Y} in Lemma~\ref{L.products} and~\eqref{J-product} in Lemma~\ref{L.J}:
\begin{equation}\label{bound-B0} \epsilon|B_0^\iota|\eqdef \epsilon\abs{\Par{(\u\cdot\nabla h)\v_{(\alpha)},J^\iota \u_{(\alpha)}}_{L^2}}\leq F(\beta\norm{\nabla b}_{H^3},\epsilon\norm{\nabla\zeta}_{H^3},\epsilon\norm{\u}_{H^3})\norm{\v_{(\alpha)}}_{Y^0}\norm{\u_{(\alpha)}}_{X^0}.\end{equation}

For the second term, we plug the identity (recall the definition of $\mfT[h,\beta b]$ in~\eqref{def-T} and~\eqref{def-mfT})
\[ h\v_{(\alpha)}=h\u_{(\alpha)}-\frac{\mu}{3}\nabla(h^3\nabla\cdot\u_{(\alpha)})+\frac\mu2\Big(\nabla\big(h^2(\beta\nabla b)\cdot \u_{(\alpha)}\big)-h^2(\beta\nabla b)\nabla\cdot \u_{(\alpha)}\Big)+\mu h\beta^2(\nabla b\cdot \u_{(\alpha)})\nabla b\]
and consider separately the four contributions.

One has
\[B_1^\iota\eqdef  \Par{(\u\cdot\nabla)(h\u_{(\alpha)}) , J^\iota\u_{(\alpha)}}_{L^2} +\Par{(h\u_{(\alpha)}\cdot\nabla)\u-(h\u_{(\alpha)})^\perp \curl \u,J^\iota \u_{(\alpha)}}_{L^2} .\]
Integrating by parts the advection operator and averaging yields
\begin{align*}
&\Par{(\u\cdot\nabla)(h\u_{(\alpha)}), J^\iota\u_{(\alpha)}}_{L^2}\\
&\quad =-\Par{\u_{(\alpha)},(\u\cdot\nabla)(h J^\iota\u_{(\alpha)})}_{L^2}+\Par{(\u\cdot\nabla h)\u_{(\alpha)}, J^\iota\u_{(\alpha)}}_{L^2}-\Par{\u_{(\alpha)},(h\nabla\cdot\u)J^\iota\u_{(\alpha)}}_{L^2}\\
&\quad =\frac12 \Par{\u_{(\alpha)},  [J^\iota ,(h\u\cdot\nabla) ]\u_{(\alpha)}}_{L^2}+\frac12\Par{(\u\cdot\nabla h)\u_{(\alpha)},J^\iota\u_{(\alpha)}}_{L^2}-\frac12\Par{\u_{(\alpha)},(h\nabla\cdot\u)J^\iota\u_{(\alpha)}}_{L^2}.\end{align*}
All components are now estimated by Cauchy-Schwarz inequality and continuous Sobolev embedding $H^2\subset L^\infty$ as well as~\eqref{J-product},\eqref{J-com-H} in Lemma~\ref{L.J}:
\begin{equation}\label{bound-B1}
\epsilon |B_1^\iota|\leq C(h^\star)F(\beta\norm{\nabla b}_{H^2},\epsilon\norm{\nabla\zeta}_{H^2},\epsilon\norm{\u}_{H^3})\norm{\u_{(\alpha)}}_{L^2}^2.
\end{equation}

One has
\begin{align*}B_2^\iota&\eqdef -\frac\mu3 \Par{(\u\cdot\nabla)(\nabla(h^3\nabla\cdot\u_{(\alpha)})) +(\nabla(h^3\nabla\cdot\u_{(\alpha)})\cdot\nabla)\u-(\nabla(h^3\nabla\cdot\u_{(\alpha)}))^\perp \curl \u,J^\iota\u_{(\alpha)}}_{L^2}\\
&=\frac\mu3 \Par{\u\cdot\nabla(h^3\nabla\cdot\u_{(\alpha)}),J^\iota \nabla\cdot\u_{(\alpha)}}_{L^2}\\
&=-\frac\mu6\Par{(h^3\nabla\cdot\u-3h^2 \u\cdot\nabla h)\nabla\cdot\u_{(\alpha)},J^\iota\nabla\cdot\u_{(\alpha)}}_{L^2}+\frac\mu6 \Par{\nabla\cdot\u_{(\alpha)},[ J^\iota,h^3\u]\cdot\nabla (\nabla\cdot\u_{(\alpha)})}_{L^2},
\end{align*}
where we used the identity~\eqref{identity} with $U=\u$ and $V=\nabla(h^3\nabla\cdot\u_{(\alpha)})$ (notice that $\curl V=0$) and integration by parts.
We conclude as above
\begin{equation}\label{bound-B2}
 \epsilon|B_2^\iota|\leq \mu\, C(h^\star)F(\beta\norm{\nabla b}_{H^2},\epsilon\norm{\nabla\zeta}_{H^2},\epsilon\norm{\u}_{H^3})\norm{\nabla\cdot\u_{(\alpha)}}_{L^2}^2.
 \end{equation}

One has, denoting for readability $F\eqdef h^2(\beta\nabla b)$, 
\begin{align*}B_3^\iota&\eqdef \frac\mu2 \Par{(\u\cdot\nabla)(\nabla(F\cdot \u_{(\alpha)})) +(\nabla(F\cdot \u_{(\alpha)})\cdot\nabla)\u-(\nabla(F\cdot \u_{(\alpha)}))^\perp \curl \u,J^\iota\u_{(\alpha)}}_{L^2}\\
&\quad -\frac\mu2 \Par{(\u\cdot\nabla)(F\nabla\cdot \u_{(\alpha)}) +((F\nabla\cdot \u_{(\alpha)})\cdot\nabla)\u-(F\nabla\cdot \u_{(\alpha)})^\perp \curl \u,J^\iota\u_{(\alpha)}}_{L^2}\\
&=-\frac\mu2 \Par{\u\cdot\nabla(F\cdot \u_{(\alpha)}),J^\iota\nabla\cdot\u_{(\alpha)}}_{L^2}-\frac\mu2 \Par{(\u\cdot\nabla)(F\nabla\cdot \u_{(\alpha)}) ,J^\iota\u_{(\alpha)}}_{L^2}\\
&\quad -\frac\mu2\Par{((F\nabla\cdot \u_{(\alpha)})\cdot\nabla)\u-(F\nabla\cdot \u_{(\alpha)})^\perp \curl \u,J^\iota\u_{(\alpha)}}_{L^2}\\
&=\frac\mu2 \Par{F\cdot \u_{(\alpha)},\u\cdot (J^\iota\nabla \nabla\cdot\u_{(\alpha)})}_{L^2}-\frac\mu2 \Par{F (\u\cdot\nabla\nabla\cdot \u_{(\alpha)}) ,J^\iota\u_{(\alpha)}}_{L^2}\\
&\quad +\frac\mu2 \Par{F\cdot \u_{(\alpha)},(\nabla\cdot\u) J^\iota\nabla\cdot\u_{(\alpha)}}_{L^2}-\frac\mu2 \Par{(\nabla\cdot \u_{(\alpha)})(\u\cdot\nabla)F ,J^\iota\u_{(\alpha)}}_{L^2}\\
&\quad -\frac\mu2\Par{((F\nabla\cdot \u_{(\alpha)})\cdot\nabla)\u-(F\nabla\cdot \u_{(\alpha)})^\perp \curl \u,J^\iota\u_{(\alpha)}}_{L^2},
\end{align*}
where we used again the identity~\eqref{identity} with $U=\u$, $V=\nabla(F\cdot \u_{(\alpha)})$ and integration by parts.
 All the components are estimated as above by Cauchy-Schwarz inequality, continuous Sobolev embedding $H^2\subset L^\infty$,~\eqref{J-product}  as well as~\eqref{J-com-H} for the first line. It follows
\begin{equation}\label{bound-B3}
\epsilon |B_3^\iota|\leq \mu\, C(h^\star)F(\beta\norm{\nabla b}_{H^3},\epsilon\norm{\nabla\zeta}_{H^2},\epsilon\norm{\u}_{H^3})\norm{\u_{(\alpha)}}_{L^2}\norm{\nabla\cdot\u_{(\alpha)}}_{L^2}.
 \end{equation}

One has 
\begin{align*}B_4^\iota&\eqdef \mu \beta^2\Par{(\u\cdot\nabla)(h(\nabla b\cdot \u_{(\alpha)})\nabla b) +h(\nabla b\cdot \u_{(\alpha)})(\nabla b\cdot\nabla)\u-h(\nabla b\cdot \u_{(\alpha)})(\nabla b)^\perp \curl \u,J^\iota\u_{(\alpha)}}_{L^2}\\
&=\mu\beta^2\Par{\u\cdot\nabla(\nabla b\cdot \u_{(\alpha)}) ,h\nabla b\cdot J^\iota \u_{(\alpha)}}_{L^2}+\mu\beta^2\Par{(\nabla b\cdot \u_{(\alpha)}) (\u\cdot\nabla)(h\nabla b) ,J^\iota\u_{(\alpha)}}_{L^2}\\
&\quad + \mu \beta^2\Par{h(\nabla b\cdot \u_{(\alpha)})(\nabla b\cdot\nabla)\u-h(\nabla b\cdot \u_{(\alpha)})(\nabla b)^\perp \curl \u,J^\iota\u_{(\alpha)}}_{L^2}.
\end{align*}
All components but the first are estimated by Cauchy-Schwarz inequality, continuous Sobolev embedding $H^2\subset L^\infty$ and~\eqref{J-product}; and one has after integrating by parts,
\[\Par{\u\cdot\nabla(\nabla b\cdot \u_{(\alpha)}),h\nabla b\cdot J^\iota \u_{(\alpha)}}_{L^2}=-\Par{\nabla b\cdot \u_{(\alpha)},(\nabla\cdot(h\u))(\nabla b\cdot J^\iota \u_{(\alpha)})}_{L^2}-\Par{\nabla b\cdot \u_{(\alpha)},h\u\cdot\nabla(\nabla b\cdot J^\iota\u_{(\alpha)})}_{L^2}.\]
One may estimate the above by averaging and repeated use of the commutator estimate~\eqref{J-com-H}, and one obtains eventually
\begin{equation}\label{bound-B4}
 \epsilon|B_4^\iota|\leq \mu\, C(h^\star)F(\beta\norm{\nabla b}_{H^3},\epsilon\norm{\nabla\zeta}_{H^2},\epsilon\norm{\u}_{H^3})\norm{\u_{(\alpha)}}_{L^2}^2.
  \end{equation}

 We have one last term to estimate, namely
 \[B_5^\iota\eqdef \Par{(\u\cdot\nabla)\v_{(\alpha)} +(\v_{(\alpha)}\cdot\nabla)\u-\v_{(\alpha)}^\perp \curl \u,[J^\iota,h]\u_{(\alpha)}}_{L^2}.\]
 By~\eqref{J-com-H} and~\eqref{product-H-0} in Lemma~\ref{L.products}, and~\eqref{embed-Y} in Lemma~\ref{L.embeddings}, one has
 \[\epsilon\abs{\Par{(\v_{(\alpha)}\cdot\nabla)\u-\v_{(\alpha)}^\perp \curl \u,[J^\iota,h]\u_{(\alpha)}}_{L^2}}
 \leq  F(\beta\norm{\nabla b}_{H^2} ,\epsilon\norm{\nabla\zeta}_{H^2},\epsilon\norm{\u}_{H^3} )\norm{ \v_{(\alpha)}}_{Y^0}\norm{\u_{(\alpha)}}_{L^2}.\]
 Then, one has, integrating by parts,
 \[\epsilon\abs{\Par{(\u\cdot\nabla)\v_{(\alpha)} ,[J^\iota,h]\u_{(\alpha)}}_{L^2}}\leq \epsilon\abs{\Par{\v_{(\alpha)} ,(\nabla\cdot\u)[J^\iota,h]\u_{(\alpha)}}_{L^2}}+\epsilon\abs{\Par{\v_{(\alpha)} ,(\u\cdot\nabla)[J^\iota,h]\u_{(\alpha)}}_{L^2}}.\]
 The first term is estimated as above, and the by duality, since
 \begin{align*}\norm{(\u\cdot\nabla)[J^\iota,h]\u_{(\alpha)}}_{X^0}&\lesssim \epsilon\norm{(\u\cdot\nabla)[J^\iota,h]\u_{(\alpha)}}_{L^2}+\sqrt\mu \epsilon\norm{\nabla\cdot((\u\cdot\nabla)[J^\iota,h]\u_{(\alpha)})}_{L^2} \\
 &\lesssim \epsilon\norm{\u}_{L^\infty}\norm{[J^\iota,h]\u_{(\alpha)}}_{H^1}+\sqrt\mu \epsilon \norm{\Lambda \u}_{L^\infty}\norm{[J^\iota,h]\u_{(\alpha)}}_{H^1} +\sqrt\mu \epsilon\norm{(\u\cdot\nabla)\nabla\cdot[J^\iota,h]\u_{(\alpha)}}_{L^2} \\
 &\leq C(\mu)F(\beta\norm{\nabla b}_{H^3},\epsilon\norm{\nabla\zeta}_{H^3},\epsilon\norm{\u}_{H^3})\norm{\u_{(\alpha)}}_{X^0},
 \end{align*}
 where we used~\eqref{J-com-H} and continuous Sobolev embedding $H^2\subset L^\infty$. Thus we have
 \begin{equation}\label{bound-B5}
  \epsilon|B_5^\iota|\leq \ C(\mu)F(\beta\norm{\nabla b}_{H^3},\epsilon\norm{\nabla\zeta}_{H^3},\epsilon\norm{\u}_{H^3})\norm{\v_{(\alpha)}}_{Y^0}\norm{\u_{(\alpha)}}_{L^2}.
   \end{equation}
 
 Collecting estimates~\eqref{bound-B0}--\eqref{bound-B5} and using Lemma~\ref{L.EvsF0}, we deduce that
 \[\epsilon B^\iota=\epsilon(B_0^\iota+B_1^\iota+B_2^\iota+B_3^\iota+B_4^\iota+B_5^\iota)\]
  is estimated as in the statement. This concludes the proof.
\end{proof}

\section{Well-posedness}\label{S.WP}

We are now in position to prove our main results concerning the Cauchy problem for~\eqref{GN-v}.

\begin{Proposition}[Existence and uniqueness]\label{P.existence}
Let $N\geq 4$, $b\in \dot{H}^{N+2}$ and $(\zeta_0,\v_0)\in H^N\times Y^N$ satisfying~\eqref{cond-h0} with $h_\star,h^\star>0$, and $\curl\v_0\in H^{N-1}$. Then there exists $T>0$ and a unique $(\zeta,\v)\in {L^\infty(0,T;H^N\times Y^N)}\cap C([0,T];H^2\times(H^1)^d)$ satisfying~\eqref{GN-v} and $(\zeta,\v)\id{t=0}=(\zeta_0,\v_0)$. Moreover, one can restrict
\[ T^{-1}=C(\mu,h_\star^{-1},h^\star)F(\beta\norm{\nabla b}_{H^{N+1}},\epsilon\norm{\zeta_0}_{H^4},\epsilon \norm{\v_0}_{Y^4},\epsilon\norm{\curl\v_0}_{H^3} )>0 \]
such that, for any $t\in[0,T]$,~\eqref{cond-h0} holds with $\tilde h_\star=h_\star/2,\tilde h^\star=2h^\star$, and
\[ \E^N(\zeta,\v)+\norm{\curl\v}_{H^{N-1}}^2\leq {\bf C}_0\, \big(\E^N(\zeta_0,\v_0)+\norm{\curl\v_0}_{H^{N-1}}^2\big) 
\]
 with $ {\bf C}_0=C(\mu,h_\star^{-1},h^\star,\beta\norm{\nabla b}_{H^{N+1}},\epsilon\norm{\zeta_0}_{H^4},\epsilon\norm{\v_0}_{Y^4},\epsilon \norm{\curl\v_0}_{H^3} )$.
\end{Proposition}
\begin{proof}{\em Construction.} The construction of a solution is fairly classical, and follows the line of~\cite[Sec.~4.3.4]{Lannes}. Consider the mollified system~\eqref{GN-v-mol} with right-hand side $r^\iota=0,\r^\iota={\bf 0}$ and mollified initial data $(\zeta_0^\iota,\v_0^\iota)\eqdef (J^\iota\zeta_0,J^\iota\v_0)$. Obviously, for any $\iota\in (0,1)$, $\zeta_0^\iota,\v_0^\iota\in H^n$ with arbitrary large $n\in\NN$. By the Cauchy-Lipschitz theorem on Banach spaces, there exists a unique (smooth) solution to~\eqref{GN-v-mol} with initial data $(\zeta_0^\iota,\v_0^\iota)$, that we denote $U^\iota\eqdef (\zeta^\iota,\v^\iota)$, defined on the maximal time interval $[0,T^\iota)$. For any multi-index $\alpha\in\NN^d$, denote
\[\zeta^\iota_{(\alpha)}\eqdef \partial^\alpha \zeta^\iota \quad \text{ and } \quad \v^\iota_{(\alpha)}\eqdef  \partial^\alpha\v^\iota-\mu\epsilon\, \nabla(w^\iota \partial^\alpha\zeta^\iota)\]
where $w^\iota\eqdef -h^\iota \nabla\cdot\u^\iota+ \beta \nabla b\cdot \u^\iota$, $h^\iota\eqdef 1+\epsilon\zeta^\iota-\beta b$, $ \u^\iota\eqdef \mfT[h^\iota,\beta b]^{-1}(h^\iota\v^\iota)$. By proceeding exactly as in the proof of Proposition~\ref{P.quasi}, one easily checks that $U^\iota$ satisfies the quasilinear mollified system~\eqref{GN-quasi-mol} with
\[r^{\iota}_{(\alpha)}=J^\iota r_{(\alpha)}(\zeta^{\iota} ,\v^{\iota}) \quad \text{ and } \quad  \r^{\iota}_{(\alpha)} =J^\iota \r_{(\alpha)}(\zeta^{\iota} ,\v^{\iota})+ \mu \epsilon \nabla r^{\prime,\iota}_{(\alpha)},\]
where $r_{(\alpha)} $ and $\r_{(\alpha)} $ are given in Proposition~\ref{P.quasi}, and
\[ r^{\prime,\iota}_{(\alpha)}=-(\Id-J^\iota)\big(\zeta^{\iota}_{(\alpha)}\partial_t w^{\iota}\big)-[J^\iota,w^{\iota}]\partial^\alpha\nabla\cdot(h^{\iota}\u^{\iota}).\]
By using~\eqref{embed-Y} in Lemma~\ref{L.embeddings}, one has
\[\mu\epsilon \norm{ \nabla r^{\prime,\iota}_{(\alpha)} }_{Y^0}\lesssim \epsilon\sqrt\mu \norm{r^{\prime,\iota}_{(\alpha)}}_{L^2},\]
which, by~\eqref{J-product} and \eqref{J-com-H} in Lemma~\ref{L.J} and using bounds obtained in the proof of Proposition~\ref{P.quasi}, is easily seen
to satisfy the same estimate as $\r_{(\alpha)} $ (this is a simplification with regards to the proof of~\cite[Sec.~4.3.4]{Lannes}; see Remark~4.26 therein). Thus we have, for any $1\leq|\alpha|\leq N$,
\[
\norm{r_{(\alpha)}^\iota}_{L^2}+\norm{\r_{(\alpha)}^\iota}_{Y^0}\leq  {\bf F} \ \left(\norm{\zeta^\iota}_{H^{ |\alpha|}}+\norm{\v^\iota}_{Y^{ |\alpha|}}+\norm{\curl \v^\iota}_{H^{|\alpha|-1}}\right)
\]
with ${\bf F}=C(\mu,h_\star^{-1},h^\star)F\big(\beta\norm{\nabla b}_{H^{N+1}},\epsilon\norm{\nabla\zeta^\iota}_{H^{3}},\epsilon\norm{\v^\iota}_{Y^{4}},\epsilon\norm{\curl\v^\iota}_{H^{3}}\big)$.
It follows, applying Proposition~\ref{P.energy0} and Proposition~\ref{P.energy},
\[ \frac{\dd}{\dd t}\F^N(\zeta^\iota,\v^\iota)\leq {\bf F}\ \big(  \F^N(\zeta^\iota,\v^\iota) + \left(\norm{\zeta^\iota}_{H^{N}}+\norm{\v^\iota}_{Y^{N}}+\norm{\curl \v^\iota}_{H^{N-1}}\right) \F^N(\zeta^\iota,\v^\iota)^{1/2} \big)\]
with ${\bf F}$ as above, using that by~$\eqref{GN-v-mol}_1$,~\eqref{J-product} in Lemma~\ref{L.J}, Lemma~\ref{L.T-differentiable} and~\eqref{product-X},\eqref{product-Y} in Lemma~\ref{L.products},
 \[\norm{\partial_t\zeta^\iota}_{H^3}\leq\norm{h^\iota\u^\iota}_{H^4}\leq C(\mu,h_\star^{-1},h^\star,\epsilon\norm{\nabla \zeta^\iota}_{H^3},\beta\norm{\nabla b}_{H^3})\norm{\v^\iota}_{Y^4}.\]
 Notice that~\eqref{cond-h0} propagates for large time since
  \[ 1+\epsilon\zeta^\iota-\beta b=1+\epsilon\zeta_0^\iota-\beta b+\epsilon \int_0^t\partial_t\zeta^\iota\geq h_\star-\epsilon \int_0^t\norm{\partial_t\zeta}_{L^\infty}.\]
Finally, applying the operator $\curl$ to~$\eqref{GN-v-mol}_2$ (recall $r^\iota=0,\r^\iota={\bf 0}$), one has
\[
\partial_t \curl \v^\iota+\epsilon J^\iota\nabla\cdot(\u^\iota\curl \v^\iota)=0.
\]
Proceeding as in Proposition~\ref{P.quasi} to show that $\curl \partial^\alpha\v^\iota$ satisfies the same conservation law up to a tame remainder term, testing against $\curl \partial^\alpha\v^\iota$ to deduce energy estimates and summing over $0\leq|\alpha|\leq N-1$ yields
\[ \frac{\dd}{\dd t}\norm{\curl \v^\iota}_{H^{N-1}}^2 \lesssim    \epsilon\norm{\u^\iota}_{H^{4}} \norm{\curl \v^\iota}_{H^{N-1}}^2+\epsilon \norm{\curl \v}_{H^3}\norm{\u^\iota}_{H^{ N}}\norm{\curl \v^\iota}_{H^{N-1}}.\]

Altogether, using Gronwall-type estimates, Lemmata~\ref{L.EvsF} and~\ref{J-product} and straightforward arguments (see for instance~\cite[p. 109]{Lannes} for details), one deduces that there exists $T>0$ with
\[ T^{-1}=C(\mu,h_\star^{-1},h^\star)F(\beta\norm{\nabla b}_{H^{N+1}},\epsilon\norm{\zeta_0}_{H^4},\epsilon \norm{\v_0}_{Y^4},\epsilon\norm{\curl\v_0}_{H^3} )>0\]
 such that for any $\iota\in(0,1)$, $T^\iota>T$; and for any $t\in[0,T]$,
 \begin{equation}\label{est-in-proof}
 \E^N(\zeta^\iota,\v^\iota)+\norm{\curl\v^\iota}_{H^{N-1}}^2\leq {\bf C}_0\ \big(\E^N(\zeta_0,\v_0)+\norm{\curl\v_0}_{H^{N-1}}^2\big) 
 \end{equation}
 with $ {\bf C}_0=C(\mu,h_\star^{-1},h^\star,\beta\norm{\nabla b}_{H^{N+1}},\epsilon\norm{\zeta_0}_{H^4},\epsilon\norm{\v_0}_{Y^4},\epsilon \norm{\curl\v_0}_{H^3} )$,
 and $1+\epsilon\zeta^\iota-\beta b$ satisfies~\eqref{cond-h0} with $\tilde h_\star=h_\star/2$ and $\tilde h^\star=2h^\star$.

 Notice that the time interval $[0,T]$ and energy estimates are uniform with respect to $\iota\in(0,1)$. We shall prove below that the sequence $(\zeta^{\iota},\v^{\iota})$ defines a Cauchy sequence whose limit provides the desired solution.
 \medskip
 
 {\em Convergence.} Denote $\underline\zeta\eqdef \zeta^{\iota_2}-\zeta^{\iota_1}$ and $\underline\v\eqdef \v^{\iota_1}-\v^{\iota_2}$. Then $(\underline\zeta,\underline\v)$ satisfies
 \begin{equation}\label{GN-v-mol-Cauchy}
    \left\{ \begin{array}{l}
    \partial_t\underline\zeta+J^{\iota_2}\nabla\cdot(\underline h\,\underline\u) = J^{\iota_2} \underline r+(J^{\iota_1}-J^{\iota_2}) \overline r\\ \\
 \partial_t \underline\v +J^{\iota_2} \left(\nabla\underline\zeta+\epsilon\underline\u^\perp \curl \underline\v+\frac\epsilon2 \nabla\big(\abs{\underline\u}^2\big)-\mu\epsilon\nabla\big(\R[\underline h,\underline\u]+\R_b[\underline h,\beta b,\underline\u]\big)\right)\\
 \hfill  = J^{\iota_2} \underline \r+(J^{\iota_1}-J^{\iota_2}) \overline \r \hspace{-1cm}
    \end{array}\right.
    \end{equation}
 with notation $\underline h=1+\epsilon\underline\zeta-\beta b$ and $\underline\u\eqdef \mfT[\underline h,\beta b]^{-1}( h \underline\v)$, and
\begin{align*}
\underline r&=\nabla\cdot(\underline h\,\underline\u) +\nabla\cdot(h^{\iota_1}\u^{\iota_1})- \nabla\cdot(h^{\iota_2}\u^{\iota_2}) ,\qquad 
\overline r= \nabla\cdot(h^{\iota_1}\u^{\iota_1}),\\
\underline \r&=\epsilon \big(\underline\u^\perp \curl \underline\v+(\u^{\iota_1})^\perp \curl \v^{\iota_1}-(\u^{\iota_2})^\perp \curl \v^{\iota_2}\big) +\frac\epsilon2 \nabla\big(\abs{\underline\u}^2+\abs{\u^{\iota_1}}^2-\abs{\u^{\iota_2}}^2\big) \\
&\quad -\mu\epsilon\nabla\big(\R[\underline h,\underline\u]+\R[ h^{\iota_1},\u^{\iota_1}]-\R[ h^{\iota_2},\u^{\iota_2}]+\R_b[\underline h,\beta b,\underline\u] +\R_b[ h^{\iota_1},\beta b,\u^{\iota_1}]-\R_b[ h^{\iota_2},\beta b,\u^{\iota_2}]\big),\\
\overline \r&= \nabla\zeta^{\iota_1}+\epsilon(\u^{\iota_1})^\perp \curl \v^{\iota_1}+\frac\epsilon2 \nabla\big(\abs{\u^{\iota_1}}^2\big)-\mu\epsilon\nabla\big(\R[ h^{\iota_1},\u^{\iota_1}]+\R_b[ h^{\iota_1},\beta b,\u^{\iota_1}]\big).
\end{align*} 
By using the previously obtained energy estimates~\eqref{est-in-proof}, Lemma~\ref{L.T-differentiable} and Lemma~\ref{L.diff-T-1}, one has
\[ \norm{\underline\u}_{X^1}+\norm{\u^{\iota_1}-\u^{\iota_2}}_{X^1}\leq {\bf C}_0\ \E^1(\underline\zeta,\underline\v)^{1/2} \]
 where we denote, here and thereafter, $ {\bf C}_0=C(\mu,h_\star^{-1},h^\star,\beta\norm{\nabla b}_{H^{N+1}},\epsilon\norm{\zeta_0}_{H^4},\epsilon\norm{\v_0}_{Y^4},\epsilon \norm{\curl\v_0}_{H^3} )$. It follows, after straightforward computations and using Lemmata~\ref{L.embeddings},~\ref{L.products} and~\eqref{J-diff} in Lemma~\ref{L.J},
\[ \norm{J^{\iota_2}\underline r}_{L^2}+\norm{J^{\iota_2}\underline \r}_{Y^0}\leq  {\bf C}_0\ \big(\E^1(\underline\zeta,\underline\v)^{1/2}+\norm{\curl\underline\v}_{L^2}\big).\]
By~\eqref{J-diff} in Lemma~\ref{L.J}, we have immediately
 \[  \norm{(J^{\iota_1}-J^{\iota_2}) \overline r}_{L^2}+\norm{(J^{\iota_1}-J^{\iota_2})   \overline \r}_{Y^0}\leq {\bf C}_0\ |\iota^2-\iota^1| .\]

 Similarly, setting $1\leq|\alpha|\leq 2$, $\underline\zeta_{(\alpha)}\eqdef \partial^\alpha \zeta^{\iota_2}-\partial^\alpha \zeta^{\iota_1}$ and $\underline \v_{(\alpha)}\eqdef \v^{\iota_2}_{(\alpha)}-\v^{\iota_1}_{(\alpha)}$ satisfy
 \begin{equation}\label{GN-quasi-mol-Cauchy}
    \left\{ \begin{array}{l}
    \partial_t\underline\zeta_{(\alpha)}+J^{\iota_2} \left(\epsilon\nabla\cdot(\u^{\iota_2} \underline\zeta_{(\alpha)}) +\nabla\cdot\big(h^{\iota_2}\mfT[h^{\iota_2},\beta b]^{-1}(h^{\iota_2} \underline\v_{(\alpha)})\big)\right)=   J^{\iota_2} \underline r_{(\alpha)}+(J^{\iota_1}-J^{\iota_2})  \overline r_{(\alpha)}\hspace{-.5cm}\\ \\
    \partial_t \underline\v_{(\alpha)} +J^{\iota_2}\left( \nabla\underline\zeta_{(\alpha)} +\epsilon(\u^{\iota_2})^\perp \curl \underline\v_{(\alpha)}+\epsilon \nabla\big(\u^{\iota_2}\cdot \underline\v_{(\alpha)}\big)\right)\\
    \hfill=    J^{\iota_2} \underline \r_{(\alpha)}+(J^{\iota_1}-J^{\iota_2})  \overline \r_{(\alpha)}+ \mu \epsilon \nabla(r^{\prime,\iota_2}_{(\alpha)}-r^{\prime,\iota_1}_{(\alpha)})
    \end{array}\right.
 \end{equation}
 where 
\begin{align*}
\underline r_{(\alpha)}&= \nabla\cdot\big(\epsilon(\u^{\iota_1} -\u^{\iota_2}) \zeta^{\iota_1}_{(\alpha)}+h^{\iota_1}\mfT[h^{\iota_1},\beta b]^{-1}(h^{\iota_1}\v^{\iota_1}_{(\alpha)})-h^{\iota_2}\mfT[h^{\iota_2},\beta b]^{-1}(h^{\iota_2} \v^{\iota_1}_{(\alpha)})\big)\\
&\qquad  +r_{(\alpha)}(\zeta^{\iota_2} ,\v^{\iota_2})-r_{(\alpha)} (\zeta^{\iota_1} ,\v^{\iota_1}),\\
\overline r_{(\alpha)}&=\epsilon\nabla\cdot(\u^{\iota_1} \zeta^{\iota_1}_{(\alpha)}) +\nabla\cdot(h^{\iota_1}\u^{\iota_1}_{(\alpha)}) -r_{(\alpha)}(\zeta^{\iota_1} ,\v^{\iota_1}) ,\\
\underline \r_{(\alpha)}&=\epsilon (\u^{\iota_1}-\u^{\iota_2})^\perp\curl \v^{\iota_1}_{(\alpha)}+\epsilon\nabla\big( (\u^{\iota_1}-\u^{\iota_2})\cdot \v^{\iota_1}_{(\alpha)}\big)+\r_{(\alpha)}(\zeta^{\iota_2} ,\v^{\iota_2})-\r_{(\alpha)} (\zeta^{\iota_1} ,\v^{\iota_1}),\\
\overline \r_{(\alpha)}&= \epsilon (\u^{\iota_1})^\perp\curl \v^{\iota_1}_{(\alpha)}+\epsilon \nabla\big(\u^{\iota_1}\cdot \v^{\iota_1}_{(\alpha)}\big)-\r_{(\alpha)}(\zeta^{\iota_1} ,\v^{\iota_1}).
\end{align*} 
In order to estimate the right-hand side, let us first notice that since $|\alpha|+2\leq 4\leq N$, one has by~\eqref{est-in-proof},~\eqref{embed-Y} in Lemma~\ref{L.embeddings},~\eqref{product-H-0} in Lemma~\ref{L.products} and Lemma~\ref{L.T-differentiable}:
\[\norm{\zeta^{\iota_1}_{(\alpha)}}_{H^2}+\norm{\v^{\iota_1}_{(\alpha)}}_{Y^2} +\norm{\u^{\iota_1}_{(\alpha)}}_{X^2}\leq {\bf C}_0.\]
Moreover, we already indicated $\norm{\u^{\iota_1} -\u^{\iota_2}}_{X^1}\leq {\bf C}_0\ \E^1(\underline\zeta,\underline\v)^{1/2}$
and by~\eqref{est-r1r2-lipsch} in Proposition~\ref{P.quasi}, 
\[\norm{r_{(\alpha)}(\zeta^{\iota_1} ,\v^{\iota_1})-r_{(\alpha)} (\zeta^{\iota_2} ,\v^{\iota_2})}_{L^2}+\norm{ \r_{(\alpha)}(\zeta^{\iota_1} ,\v^{\iota_1})- \r_{(\alpha)} (\zeta^{\iota_2} ,\v^{\iota_2})}_{Y^0}
\leq {\bf C}_0\ \big( \E^2(\underline\zeta,\underline\v)^{1/2} + \norm{\curl\underline\v}_{H^1}\big).
\]
With these estimates in hand,~\eqref{J-product} in Lemma~\ref{L.J} and Lemmata~\ref{L.embeddings},~\ref{L.products},~\ref{L.T-invertible}~\ref{L.T-differentiable} and~\ref{L.diff-T-1} yield
\[ \norm{J^{\iota_2}\underline r_{(\alpha)}}_{L^2}+\norm{J^{\iota_2}\underline \r_{(\alpha)}}_{Y^0}\leq  {\bf C}_0\ \big( \E^2(\underline\zeta,\underline\v)^{1/2} + \norm{\curl\underline\v}_{H^1}\big)\]
and, by~\eqref{est-in-proof} and~\eqref{J-diff} in Lemma~\ref{L.J},
 \[  \norm{(J^{\iota_1}-J^{\iota_2}) \overline r_{(\alpha)}}_{L^2}+\norm{(J^{\iota_1}-J^{\iota_2})   \overline \r_{(\alpha)}}_{Y^0}\leq {\bf C}_0\ |\iota^2-\iota^1| .\]
Following the same remark as above, we control the last contribution:
\[\mu\epsilon\norm{\nabla r^{\prime,\iota_2}_{(\alpha)}-\nabla r^{\prime,\iota_1}_{(\alpha)}}_{Y^0}\leq \sqrt\mu\epsilon \norm{r^{\prime,\iota_2}_{(\alpha)}- r^{\prime,\iota_1}_{(\alpha)}}_{L^2}\leq {\bf C}_0\ \Big( \E^2(\underline\zeta,\underline\v)^{1/2} +|\iota^2-\iota^1|\Big).\]
 Thus applying Proposition~\ref{P.energy0} to~\eqref{GN-v-mol-Cauchy}, Proposition~\ref{P.energy} to~\eqref{GN-quasi-mol-Cauchy} and adapting Lemma~\ref{L.EvsF}, one has
 \[ \frac{\dd}{\dd t}\F^2(\underline\zeta,\underline\v)\leq {\bf C}_0\ \Big( \F^2(\underline\zeta,\underline\v) + \norm{\curl\underline\v}_{H^1}^2+|\iota^2-\iota^1|\F^2(\underline\zeta,\underline\v) ^{1/2}\Big).\]
with the notation $\F^2(\underline\zeta,\underline\v)\eqdef \sum_{0\leq |\alpha|\leq 2} \F[h^{\iota_2},\beta b](\underline\zeta_{(\alpha)},\underline\v_{(\alpha)})$. Notice also the identity
\[
\partial_t \curl \underline\v+\epsilon J^{\iota_2}\nabla\cdot(\u^{\iota_2}\curl \underline\v)=J^{\iota_2}\nabla\cdot((\u^{\iota_1}-\u^{\iota_2})\curl \underline\v^{\iota_1})+\epsilon( J^{\iota_1}-J^{\iota_2})\nabla\cdot(\u^{\iota_1}\curl \underline\v^{\iota_1}),
\]
so that, proceeding as above,
\[ \frac{\dd}{\dd t}\norm{\curl\underline\v}_{H^1}^2 \leq {\bf C_0}\ \Big( \F^2(\underline\zeta,\underline\v) + \norm{\curl\underline\v}_{H^1}^2+|\iota^2-\iota^1|\F^2(\underline\zeta,\underline\v) ^{1/2}\Big).\]
Applying Gronwall's Lemma and since, by~\eqref{J-diff} in Lemma~\ref{L.J}, the initial data satisfies
\[\norm{\zeta_0^{\iota_2}-\zeta_0^{\iota_1}}_{H^2}+\norm{\v_0^{\iota_2}-\v_0^{\iota_1}}_{H^2}=\norm{( J^{\iota_2}-J^{\iota_1})\zeta_0}_{H^2}+\norm{( J^{\iota_2}-J^{\iota_1})\v_0}_{H^2}\lesssim |\iota^2-\iota^1| \E^N(\zeta_0,\v_0)^{1/2},
\]
we find that
\[\F^2(\underline\zeta,\underline\v) ^{1/2}+\norm{\curl\underline\v}_{H^1} \leq {\bf C}_0 |\iota^2-\iota^1|(1+t)\exp( {\bf C}_0 t).\]
Slightly adapting the proof Lemma~\ref{L.EvsF} and thanks to~\eqref{est-in-proof} and Lemma~\ref{L.T-differentiable} and~\ref{L.diff-T-1}, we deduce
\[\norm{\zeta^{\iota_2}-\zeta^{\iota_1}}_{H^2}+\norm{\v^{\iota_2}-\v^{\iota_1}}_{Y^2}+\norm{\curl \v^{\iota_2}-\curl\v^{\iota_1}}_{H^1} +\norm{\u^{\iota_2}-\u^{\iota_1}}_{X^2}\leq {\bf C}_0 |\iota^2-\iota^1|(1+t)\exp( {\bf C}_0 t).\]
The Cauchy sequences are strongly convergent in low regularity Banach spaces $C([0,T];H^n)$, and are also bounded thus weakly convergent (up to a subsequence) in high regularity spaces by~\eqref{est-in-proof}. By uniqueness of the limit, there exists 
$(\zeta,\v,w,\u,v)\in L^\infty(0,T;H^N\times Y^N\times H^{N-1}\times X^N\times H^N)$ such that~\eqref{cond-h0} holds with $\tilde h_\star=h_\star/2$ and $\tilde h^\star=2h^\star$, satisfying the desired energy estimates and
\[\lim_{\iota\to 0}\sup_{t\in[0,T]}\big(\norm{\zeta^\iota-\zeta}_{H^2} +\norm{\v^\iota-\v}_{H^1}+\norm{\curl \v^\iota-w}_{H^2}+\norm{\u^\iota-\u}_{H^2}+\norm{\nabla\cdot\u^\iota-v}_{H^2}\big)=0.\]
By uniqueness of the limit, one has $w=\curl\v$, $v=\nabla\cdot\u$ and $\u=\mfT[h,\beta b]^{-1}(h\v)$.
The level of regularity in the above convergence result is sufficient to pass to the limit in~\eqref{GN-v-mol}, so that $(\zeta,\v)\in L^\infty(0,T;H^N\times Y^N)\cap C([0,T];H^2\times (H^1)^d)$ is a strong solution to~\eqref{GN-v}. That it satisfies the desired initial data is guaranteed by~\eqref{J-limit} in Lemma~\ref{L.J}. 
\medskip

{\em Uniqueness.} By considering $\underline\zeta,\underline\v$ the difference between two solutions with same initial data, and proceeding exactly as above (with fewer terms since mollifications are not involved), we find 
\[ \frac{\dd}{\dd t}\Big(\underline\F^2(\underline\zeta,\underline\v) +\norm{\curl\underline\v}_{H^1}^2\Big) \leq {\bf C}_0\ \Big( \underline\F^2(\underline\zeta,\underline\v) + \norm{\curl\underline\v}_{H^1}^2\Big).\]
 Applying Gronwall's estimate and since $\big(\underline\F^2(\underline\zeta,\underline\v) +\norm{\curl\underline\v}_{H^1}^2\big)\id{t=0}=0$, we deduce $\underline\zeta=0,\underline\v={\bf 0}$. This concludes the proof of Proposition~\ref{P.existence}.
\end{proof}

\begin{Proposition}[Stability]\label{P.stability}
Let the assumptions of Proposition~\ref{P.existence} be satisfied, and assume that $(\t\zeta,\t\v,\curl\t\v)\in L^\infty(0,\t T;H^{N+1}\times Y^{N+1}\times H^N)\cap C([0,\t T];H^{n}\times Y^{n}\times H^{n-1})$ satisfies~\eqref{cond-h0} and~\eqref{GN-v} with remainders $(\t r,\t\r,\curl\t\r)\in L^1(0,\t T;H^n\times Y^n \times H^{n-1})$ with $1\leq n\leq N$. Denote
\[ M_0=\norm{\zeta_0}_{H^N}+\norm{\v_0}_{Y^N} +\norm{\curl\v_0}_{H^N}, \quad \t M= \esssup_{t\in [0,T]} \big(\norm{\t\zeta}_{H^{N+1}}+\norm{\t\v}_{Y^{N+1}}+\norm{\curl\t\v}_{H^{N}}\big).\] Then there exists
\[ T^{-1}=C(\mu,h_\star^{-1},h^\star)F(\beta\norm{\nabla b}_{H^{N+1}},\epsilon M_0,\epsilon\t M)>0,\]
such that for any $t\in[0,\min(T,\t T)]$,
\begin{multline*}\big(\norm{\zeta-\t\zeta}_{H^n}+\norm{\v-\t\v}_{Y^n}+\norm{ \curl\v-\curl\t\v}_{H^{n-1}}\big)(t)\leq {\bf C} \big(\norm{\zeta-\t\zeta}_{H^n}+\norm{\v-\t\v}_{Y^n}+\norm{ \curl\v-\curl\t\v}_{H^{n-1}}\big)\id{t=0}\\
+ {\bf C} \, \int_0^t\big(\norm{\t r}_{H^n}+\norm{\t\r}_{Y^n}+\norm{\curl\t\r}_{H^{n-1}}\big)(t')\dd t'.
\end{multline*}
with ${\bf C}=C(\mu,h_\star^{-1},h^\star,\beta\norm{\nabla b}_{H^{N+1}},\epsilon M_0,\epsilon \t M)$.
\end{Proposition}
\begin{proof}
This Lipschitz stability property was already at stake in the convergence part of the proof of Proposition~\ref{P.existence}, and is shown in the same way. 

We denote $\underline \zeta\eqdef \zeta-\t\zeta$, $\underline \v\eqdef \v-\t\v$, $\underline h\eqdef 1+\epsilon\underline\zeta-\beta b$ and $\underline\u\eqdef \mfT[\underline h,\beta b]^{-1}(\underline h \underline\v)$; as well as $\underline\zeta_{(\alpha)}\eqdef \partial^\alpha \zeta-\partial^\alpha \t\zeta$, $\underline \v_{(\alpha)}\eqdef \v_{(\alpha)}-\t\v_{(\alpha)}$ and $\underline\u_{(\alpha)}\eqdef \mfT[h,\beta b]^{-1}(h\underline \v_{(\alpha)})$. Consider the system satisfied by $(\underline\zeta,\underline\v)$ and $(\underline\zeta_{(\alpha)},\underline\v_{(\alpha)})$. Compared with~\eqref{GN-v-mol-Cauchy} and~\eqref{GN-quasi-mol-Cauchy}, there are fewer terms in the right-hand side, but they require more precise estimates. For $1\leq |\alpha|\leq n$, one has
 \begin{equation}\label{GN-quasi-stab}
    \left\{ \begin{array}{l}
    \partial_t\underline\zeta_{(\alpha)}+ \epsilon\nabla\cdot(\u \underline\zeta_{(\alpha)}) +\nabla\cdot(h \underline\u_{(\alpha)}) =   \underline r_{(\alpha)}-\partial^\alpha \t r+r_{(\alpha)}(\zeta,\v)-r_{(\alpha)} (\t\zeta ,\t\v)\\ \\
    \partial_t \underline\v_{(\alpha)} +\nabla \underline\zeta_{(\alpha)} +\epsilon\u^\perp \curl \underline\v_{(\alpha)}+\epsilon \nabla(\u\cdot \underline\v_{(\alpha)})=   \underline \r_{(\alpha)}-\partial^\alpha \t\r+\mu \epsilon \nabla(\t w \partial^\alpha \t r)+\r_{(\alpha)}(\zeta ,\v)-\r_{(\alpha)} (\t\zeta ,\t\v)
    \end{array}\right.
 \end{equation}
 where $\t w\eqdef - \t h \nabla\cdot\t\u+ \beta \nabla b\cdot \t\u$, $\t h\eqdef 1+\epsilon\t\zeta-\beta b$, $\t\u \eqdef \mfT[\t h,\beta b]^{-1}(\t h\t\v) $, and
\begin{align*}
\underline r_{(\alpha)}&= \nabla\cdot\big(\epsilon(\t\u -\u) \t\zeta_{(\alpha)}+\t h\mfT[\t h,\beta b]^{-1}(\t h\t\v_{(\alpha)}) -h\mfT[h,\beta b]^{-1}(h\t\v_{(\alpha)})\big),\\
\underline \r_{(\alpha)}&=\epsilon(\t\u-\u)^\perp \curl \t\v_{(\alpha)}+\epsilon \nabla\big( (\t\u-\u)\cdot \t\v_{(\alpha)}\big).
\end{align*} 
 Since $|\alpha|\leq N$, one has by~\eqref{embed-Y} in Lemma~\ref{L.embeddings} and Lemmata~\ref{L.products} and~\ref{L.T-differentiable} to estimate $\norm{\t w}_{H^2}$,
\[ \epsilon\norm{\t\zeta_{(\alpha)}}_{H^{1}}+\epsilon\norm{\t\v_{(\alpha)}}_{Y^{1}}+\epsilon\norm{\curl\t\v_{(\alpha)}}_{L^2}\leq C(\mu,h_\star^{-1},h^\star)F( \beta\norm{\nabla b}_{H^3},\epsilon\norm{\t\zeta}_{H^{N+1}},\epsilon\norm{\t\v}_{Y^{N+1}},\epsilon\norm{\curl\t\v}_{H^N}).\]
Using Lemmata~\ref{L.embeddings},~\ref{L.products},~\ref{L.T-differentiable}, and~\ref{L.diff-T-1}, one checks that for $m\in\{0,1,2,3\}$,
\[\norm{\u-\t\u}_{H^m}\leq\norm{\u-\t\u}_{X^m}\leq {\bf C}\,\times\big(\norm{\underline\zeta}_{H^m}+\norm{\underline\v}_{Y^m}\big)\]
with ${\bf C}=C(\mu,h_\star^{-1},h^\star,\beta\norm{\nabla b}_{H^{3}},\epsilon\norm{\zeta}_{H^3},\epsilon\norm{\v}_{Y^3},\epsilon\norm{\t\zeta}_{H^{3}},\epsilon\norm{\t\v}_{Y^{4}})$.
It follows, using again Lemmata~\ref{L.embeddings},~\ref{L.products},~\ref{L.T-differentiable}, and adapting the proof of Lemma~\ref{L.diff-T-1} to replace $\norm{\v}_{Y^{2\vee n}}\norm{\zeta-\t\zeta}_{H^n}$ with $\norm{\v}_{Y^{n}}\norm{\zeta-\t\zeta}_{H^{2\vee n}}$, 
\[\norm{\underline r_{(\alpha)}}_{L^2}+\norm{\underline \r_{(\alpha)}}_{Y^0}\leq \epsilon\ {\bf C}\,\big(\norm{\underline\zeta}_{H^3}+\norm{\underline\v}_{Y^3}\big) \big( \norm{\t\zeta_{(\alpha)}}_{H^{1}}+\norm{\t\v_{(\alpha)}}_{Y^{1}}+\norm{\curl\t\v_{(\alpha)}}_{L^2}\big) ,\]
with ${\bf C}$ as above. 
Moreover,  by~\eqref{est-r1r2-lipsch} in Proposition~\ref{P.quasi}, one has
\[\norm{r_{(\alpha)}(\zeta,\v)-r_{(\alpha)} (\t\zeta ,\t\v)}_{L^2}+\norm{ \r_{(\alpha)}(\zeta,\v)- \r_{(\alpha)} (\t\zeta ,\t\v)}_{Y^0}
\leq \t{\bf F}\  \left(\norm{\underline\zeta}_{H^{|\alpha|}}+\norm{\underline\v}_{Y^{|\alpha|}}+\norm{\curl\underline\v}_{H^{|\alpha|-1}}\right)
\]
with 
\begin{multline*} \t{\bf F}=C(\mu,h_\star^{-1},h^\star)F\big(\beta\norm{\nabla b}_{H^{4\vee |\alpha|+1}},\epsilon\norm{\zeta}_{H^{4\vee|\alpha|}},\epsilon\norm{\v}_{Y^{4\vee|\alpha|}},\epsilon\norm{\curl\v}_{H^{3\vee|\alpha|-1}},\\
\epsilon\norm{\t\zeta}_{H^{4\vee|\alpha|}},\epsilon\norm{\t\v}_{Y^{4\vee |\alpha|}},\epsilon\norm{\curl\t\v}_{H^{3\vee|\alpha|-1}}\big).\end{multline*}
Finally, using~\eqref{embed-Z} in Lemma~\ref{L.embeddings}, one finds as above
\[\norm{\mu\epsilon \nabla(\t w\partial^\alpha \t r)}_{Y^0} \leq C(\mu) \norm{\epsilon\t w \partial^\alpha \t r}_{L^2}\leq \ {\bf C}\, \ \norm{\partial^\alpha \t r}_{L^2}.\]
By Proposition~\ref{P.energy}, and using the identity $\partial_t\zeta=-\nabla\cdot(h\u)$ to control $\norm{\partial_t\zeta}_{H^3}$, we obtain
\[ \frac{\dd}{\dd t}\F(\underline\zeta_{(\alpha)},\underline\v_{(\alpha)})\leq {\bf F}\,  \big(\F(\underline\zeta_{(\alpha)},\underline\v_{(\alpha)}) +\E^{3\vee|\alpha|}(\underline\zeta,\underline\v) +\norm{\curl\underline\v}_{H^{|\alpha|-1}}^2\big) +{\bf C} \big(\F(\partial^\alpha \t r,\partial^\alpha \t\r)  \F(\underline\zeta_{(\alpha)},\underline\v_{(\alpha)}) \big)^{1/2},\]
with ${\bf F}=C(\mu,h_\star^{-1},h^\star)F(\beta\norm{\nabla b}_{H^{N+1}},\epsilon\norm{\zeta}_{H^N},\epsilon\norm{\v}_{Y^N},\epsilon\norm{\curl\t \v}_{H^{ N-1}},\epsilon\norm{\t \zeta}_{H^{ N+1}},\epsilon\norm{\t \v}_{Y^{ N+1}},\epsilon\norm{\curl\t \v}_{H^{ N}})$ and ${\bf C}=C(\mu,h_\star^{-1},h^\star,\beta\norm{\nabla b}_{H^{N+1}},\epsilon\norm{\zeta}_{H^N},\epsilon\norm{\v}_{Y^N},\epsilon\norm{\curl\t \v}_{H^{ N-1}},\epsilon\norm{\t \zeta}_{H^{ N+1}},\epsilon\norm{\t \v}_{Y^{ N+1}},\epsilon\norm{\curl\t \v}_{H^{ N}})$.

In order to control $\curl \v$, we notice that
\[\partial_t\curl\underline\v+\epsilon\nabla\cdot(\u \curl\underline\v)=\epsilon\nabla\cdot \big((\t\u-\u)\curl\t\v\big)-\curl\t\r,\]
so that standard energy estimates yield
\[\frac{\dd}{\dd t}\Big(\norm{\curl\underline\v}_{H^{n-1}}^2\Big)\lesssim \Big(\epsilon \norm{\u}_{H^{2\vee n}}\norm{\curl\underline\v}_{H^{n-1}}+\epsilon\norm{\t\u-\u}_{H^n}\norm{\curl\t\v}_{H^{2\vee n}}  +\norm{\curl\t\r}_{H^{n-1}}\Big)\norm{\curl\underline\v}_{H^{n-1}}.\]
Finally, one easily checks that, as in Lemma~\ref{L.EvsF}, that
\[ \E^n(\underline\zeta,\underline\v)\leq {\bf C}\, \sum_{0\leq|\alpha|\leq n}\F(\underline\zeta_{(\alpha)},\underline\v_{(\alpha)}) \quad ; \quad \sum_{0\leq |\alpha|\leq n}\F(\underline\zeta_{(\alpha)},\underline\v_{(\alpha)})\leq {\bf C}\, \E^n(\underline\zeta,\underline\v).\]

Thus for any $n\geq 3$, adding the above energy estimates for $1\leq |\alpha|\leq n$, the corresponding one based on Proposition~\ref{P.energy0} when $\alpha=(0,0)$, and by Gronwall's Lemma, we find
\begin{multline*} \big(\norm{\underline\zeta}_{H^n}+\norm{\underline\v}_{Y^n}+\norm{\curl\v}_{H^{n-1}}\big)(t)\leq \big(\norm{\underline\zeta\id{t=0}}_{H^n}+\norm{\underline\v\id{t=0}}_{Y^n}+\norm{\curl\v\id{t=0}}_{H^{n-1}}\big) e^{{\bf F} t}\\
+ {\bf C} \int_0^t \big(\norm{\t r}_{H^n}+\norm{\t\r}_{Y^n}+\norm{\curl\t\r}_{H^{n-1}}\big)(t') e^{{\bf F}(t-t')}\dd t'.\end{multline*}
with ${\bf C},{\bf F}$ as above. The proposition is proved for $n\geq 3$, using the energy estimate of Proposition~\ref{P.existence}. The case $n\leq 2$ is obtained in the same way, but using the estimates
\[\norm{\underline r_{(\alpha)}}_{L^2}+\norm{\underline \r_{(\alpha)}}_{Y^0}\leq \epsilon\ {\bf C}\,\big(\norm{\underline\zeta}_{H^1}+\norm{\underline\v}_{Y^1}\big) \big( \norm{\t\zeta_{(\alpha)}}_{H^{3}}+\norm{\t\v_{(\alpha)}}_{H^{3}}+\norm{\curl\t\v_{(\alpha)}}_{H^2}\big) \]
 (notice that in that case, $|\alpha|\leq 2\leq N+1-3$).
\end{proof}

We now conclude this section with continuity results, completing the well-posedness of the Cauchy problem for system~\eqref{GN-v} in the sense of Hadamard.
\begin{Proposition}[Well-posedness]\label{P.continuity}
Under the hypotheses of Proposition~\ref{P.existence}, the unique strong solution to~\eqref{GN-v} satisfies $(\zeta,\v,\curl\v)\in C([0,T];H^N\times Y^N\times H^{N-1})$. Moreover, the mapping $(\zeta_0,\v_0\curl\v_0)\in H^N\times Y^N\times H^{N-1}\mapsto  (\zeta,\v,\curl\v)\in {C([0,T];H^N\times Y^N\times H^{N-1})}$ is continuous. 

More precisely, given $(\zeta_0,\v_0,\curl\v_0)\in  H^N\times Y^N\times H^{N-1}$ satisfying~\eqref{cond-h0} and a sequence $(\zeta_{0,n},\v_{0,n},\curl\v_{0,n})\to (\zeta_0,\v_0,\curl\v_0) $ in ${H^N\times Y^N\times H^{N-1}}$, then there exists  $n_0\in\NN$ and one can set $T^{-1}=C(\mu,h_\star^{-1},h^\star)\times F(\beta\norm{\nabla b}_{H^{N+1}},\epsilon\norm{\zeta_0}_{H^N},\epsilon\norm{\v_0}_{Y^N},\norm{\curl\v_0}_{H^{N-1}})$ such that for all $n\geq n_0$, there exists a unique $(\zeta_n,\v_n)\in C([0,T];H^N\times Y^N)$ satisfying~\eqref{GN-v} with initial data $(\zeta_n,\v_n)\id{t=0}=(\zeta_{0,n},\v_{0,n})$, and one has
\[ \lim_{n\to\infty}\sup_{t\in[0,T]} \Big(\norm{\zeta_n-\zeta}_{H^N}+\norm{\v_n-\v}_{Y^N}+\norm{\curl\v_n-\curl\v}_{H^{N-1}}\Big)=0.\]
\end{Proposition}
\begin{proof} Our proof is based on the Bona-Smith technique~\cite{BonaSmith75}. For any $\iota\in (0,1)$, denote $(\zeta^\iota,\v^\iota)$ (resp. $(\zeta_{n}^\iota,\v_{n}^\iota)$) the unique solution to~\eqref{GN-v} with mollified initial data $(\zeta_0^\iota,\v_0^\iota)\eqdef (J^\iota\zeta_0,J^\iota\v_0)$ (resp. $(\zeta_{0,n}^\iota,\v_{0,n}^\iota)\eqdef (J^\iota\zeta_{0,n},J^\iota\v_{0,n})$) in $C([0,T];H^2\times (H^1)^d)\cap L^\infty(0,T;H^N\times Y^N)$, as provided by Proposition~\ref{P.existence}; see Lemma~\ref{L.J} for the definition of $J^\iota$ and relevant properties. In particular, by~\eqref{J-product}, one can restrict $n_0\in\NN$ such that the energy estimate and lower bound on $T$ stated in Proposition~\ref{P.existence} hold uniformly with respect to $\iota\in(0,1)$ and $n\geq n_0$.

We then proceed as in the proof of Proposition~\ref{P.stability}, with $(\tilde\zeta,\tilde \v)=(\zeta^\iota,\v^\iota)$ and $(r,\r)={\bf 0}$. We thus find that the difference, $\underline \zeta\eqdef \zeta-\zeta^\iota$ and $\underline\v=\v-\v^\iota$ satisfies, for any $3\leq m\leq N$,
\begin{multline*} \frac{\dd}{\dd t}\big(\F^m(\underline\zeta,\underline\v)+\norm{\curl\underline\v}_{H^{m-1}}^2\big) \leq {\bf F}\,  \big(\F^m(\underline\zeta,\underline\v)+\norm{\curl\underline\v}_{H^{m-1}}^2\big) \\
+\epsilon {\bf C}\, \F^m(\underline\zeta,\underline\v)^{1/2} \E^3(\underline\zeta,\underline\v)^{1/2} \big(\norm{\zeta^\iota}_{H^{m+1}}+\norm{\v^\iota}_{Y^{m+1}}+\norm{\curl\v^\iota}_{H^{m}}\big) ,\end{multline*}
 with $\F^m(\underline\zeta,\underline\v)\eqdef \sum_{0\leq |\alpha|\leq m} \F[h,\beta b](\underline\zeta_{(\alpha)},\underline\v_{(\alpha)})$ where $\underline\zeta_{(\alpha)}\eqdef \partial^\alpha \zeta-\partial^\alpha \t\zeta$, $\underline \v_{(\alpha)}\eqdef \v_{(\alpha)}-\t\v_{(\alpha)}$, and
\begin{align*}{\bf F}&=C(\mu,h_\star^{-1},h^\star)F\big(\beta\norm{\nabla b}_{H^{N+1}},\epsilon\norm{\zeta}_{H^{N}},\epsilon\norm{\v}_{Y^{N}},\epsilon\norm{\curl\v}_{H^{N-1}},
\epsilon\norm{\zeta^\iota}_{H^{N}},\epsilon\norm{\v^\iota}_{Y^{N}},\epsilon\norm{\curl\v^\iota}_{H^{N-1}}\big),\\
{\bf C}&=C\big(\mu,h_\star^{-1},h^\star,\beta\norm{\nabla b}_{H^{N+1}},\epsilon\norm{\zeta}_{H^{N}},\epsilon\norm{\v}_{Y^{N}},\epsilon\norm{\curl\v}_{H^{N-1}},
\epsilon\norm{\zeta^\iota}_{H^{N}},\epsilon\norm{\v^\iota}_{Y^{N}},\epsilon\norm{\curl\v^\iota}_{H^{N-1}}\big).\end{align*}
By~\eqref{J-product} in Lemma~\ref{L.J} and the energy estimate in Proposition~\ref{P.existence}, one has ${\bf F}+{\bf C}\leq {\bf C}_0$ with
\[ {\bf C}_0=C(\mu,h_\star^{-1},h^\star,\beta\norm{\nabla b}_{H^{N+1}},\epsilon\norm{\zeta_0}_{H^N},\epsilon\norm{\v_0}_{Y^N},\epsilon \norm{\curl\v_0}_{H^{N-1}} ).\]
By~\eqref{J-smooth} in Lemma~\ref{L.J} and the energy estimate in Proposition~\ref{P.existence}, we find
\[ \norm{\zeta^\iota}_{H^{N+1}}+\norm{\v^\iota}_{Y^{N+1}}+\norm{\curl\v^\iota}_{H^{N}}\leq {\bf C}_0\, \big(\norm{J^\iota\zeta_0}_{H^{N+1}}+\norm{J^\iota\v_0}_{Y^{N+1}}+\norm{\curl J^\iota\v_0}_{H^{N}}\big) \leq \iota^{-1} {\bf C}_0.
\]
Moreover, since $3+1\leq N$, the above differential energy inequality with $m=3$ reads
\[  \frac{\dd}{\dd t}\big(\F^N(\underline\zeta,\underline\v)+\norm{\curl\underline\v}_{H^{N-1}}^2\big) \leq {\bf C}_0\,  \big(\F^N(\underline\zeta,\underline\v)+\norm{\curl\underline\v}_{H^{N-1}}^2 +  \E^3(\underline\zeta,\underline\v)\big),\]
so that, adapting Lemma~\ref{L.EvsF} and Gronwall's Lemma, and  finally applying~\eqref{J-limit} in Lemma~\ref{L.J},
\[ \sup_{t\in[0,T]}\iota^{-1}\E^3(\underline\zeta,\underline\v)\leq {\bf C}_0 \exp({\bf C}_0 T) \times  \iota^{-1}\E^3(\underline\zeta_0,\underline\v_0) \to 0 \quad (\iota\to0).\]
Thus applying Gronwall's Lemma to the differential energy inequality with $m=N$ and again adapting Lemma~\ref{L.EvsF} yields
\[\lim_{\iota\to0}\sup_{t\in[0,T]} \big(\norm{ \zeta-\zeta^\iota}_{H^N}+\norm{\v-\v^\iota}_{Y^N} +\norm{\curl\v-\curl\v^\iota}_{H^{N-1}}\big)= 0 ,\]
Using that, for any $\iota\in(0,1)$, $(\zeta^\iota,\v^\iota) \in C([0,T];H^N\times Y^N)$ (by the smoothness of the initial data, Proposition~\ref{P.existence}, and integrating~\eqref{GN-v} with respect to time), we deduce $(\zeta,\v) \in C([0,T];H^N\times Y^N)$.
\medskip

Now let us turn to the continuity of the flow map. The proof above yields
\[\lim_{\iota\to0}\sup_{t\in[0,T]} \big(\norm{ \zeta_n-\zeta_n^\iota}_{H^N}+\norm{\v_n-\v_n^\iota}_{Y^N} +\norm{\curl\v_n-\curl\v_n^\iota}_{H^{N-1}}\big)= 0\]
uniformly with respect to $n\geq n_0$ (notice in particular the uniformity with respect to $n$ in Lemma~\ref{L.J}). Moreover, proceeding in the same way, we find the following estimate for $(\underline\zeta^\iota,\underline\v^\iota)\eqdef (\zeta_n^\iota-\zeta^\iota,\v_n^\iota-\v^\iota)$ with any given $\iota\in(0,1)$: 
\[\frac{\dd}{\dd t}\big(\F^N(\underline\zeta^\iota,\underline\v^\iota)+\norm{\curl\underline\v^\iota}_{H^{N-1}}^2\big) \leq {\bf C}_0\,  \big(\F^N(\underline\zeta^\iota,\underline\v^\iota)+\norm{\curl\underline\v^\iota}_{H^{N-1}}^2
+ \iota^{-1}\F^N(\underline\zeta^\iota,\underline\v^\iota)^{1/2} \E^3(\underline\zeta^\iota,\underline\v^\iota)^{1/2}\big).\]
By Gronwall's Lemma, adapting Lemma~\ref{L.EvsF} and~\eqref{J-product} in Lemma~\ref{L.J}, we immediately deduce
\[\lim_{n\to\infty}\sup_{t\in[0,T]} \big(\norm{\zeta^\iota-\zeta_n^\iota}_{H^N}+\norm{\v^\iota-\v_n^\iota}_{H^N}+\norm{\curl\underline\v}_{H^{N-1}}\big)=0.\]
The continuity of the flow map follows from the above limits and triangular inequality.
\end{proof}

\section{Proof of the main results}\label{S.Formulations-comparison}

We proved in the previous section the well-posedness and stability of the Cauchy problem for system~\eqref{GN-v}. We show in this section how to transcribe these results to the original formulation of the Green-Naghdi system~\eqref{GN-u}. It is claimed in~\cite{CamassaHolmLevermore96,Matsuno16} that~\eqref{GN-v} and~\eqref{GN-u} are equivalent, after ``judiciously differentiating by parts'' and ``lengthy calculations''. Since this fact is of considerable importance in our work and is rather tedious to check, we detail the calculations below. We then conclude this section with the proof of Theorems~\ref{T.WP} and~\ref{T.justification}.
\begin{Proposition}\label{P.GNuvsGNv}
Let $b\in \dot{H}^5(\RR^d)$ and $\zeta\in C^0 ([0,T];H^4(\RR^d))$ be such that~\eqref{cond-h0} holds. 

If $\u \in C^0([0,T]; X^4)$ is such that $(\zeta,\u)$ satisfies~\eqref{GN-u}, then $\v\eqdef h^{-1}\mfT[h,\beta b]\u\in C^0([0,T];Y^4)$ is uniquely defined and $(\zeta,\v)$ satisfies~\eqref{GN-v}.

If $\v \in C^0([0,T]; Y^4)$ is such that $(\zeta,\v)$ satisfies~\eqref{GN-v}, then $\u\eqdef \mfT[h,\beta b]^{-1}(h\v)\in C^0([0,T];X^4)$ is uniquely defined and $(\zeta,\u)$ satisfies~\eqref{GN-u}.
\end{Proposition}
\begin{proof} 
By Lemmata~\ref{L.embeddings},~\ref{L.products} and~\ref{L.com-Hn}, one easily checks that $\v\eqdef h^{-1}\mfT[h,\beta b]\u\in C^0([0,T];Y^4)$. Conversely, that $\u\eqdef \mfT[h,\beta b]^{-1}(h\v)$ is well-defined and $\u\in C^0([0,T];X^4)$ follows by Lemmata~\ref{L.T-invertible} and~\ref{L.T-differentiable}. The regularity of time derivatives is provided by the equations~\eqref{GN-u} or~\eqref{GN-v}. This regularity is sufficient to ensure that all the identities below hold in, say, $L^2(\RR^d)$.

Let us first notice that, by the identity $\frac12\nabla(\abs{\u}^2)=(\u\cdot\nabla) \u-\u^\perp \curl \u$, all terms of order $\O(\mu)$ in~\eqref{GN-v} and~\eqref{GN-u} agree. Notice also that, as pointed out in Appendix~\ref{S.Formulations-Hamiltonian}, system~\eqref{GN-v} can be rewritten as~\eqref{GN-vter}. It follows that to complete the proof, we only need to show that
\begin{equation}
\label{goal} \big[\partial_t,\T[h,\beta b]\big]\u+\epsilon \u^\perp \curl(\T[h,\beta b] \u)+ \epsilon \nabla \big(\u\cdot \T[h,\beta b] \u-\frac{1}{2}w^2\big)\\=\epsilon\Q[h,\u]+ \epsilon\Q_b[h,\beta b,\u],
\end{equation}
with $w\eqdef  (\beta\nabla b)\cdot\u-h\nabla\cdot\u$.
We clarify below why~\eqref{goal} holds. Let us first decompose
\begin{align*}
\mathcal T[h,\beta b]\u&= \left\{\frac{-1}{3h}\nabla(h^3\nabla\cdot \u)+\frac\beta{2h}\nabla\big(h^2\nabla b\cdot \u\big) \right\} \ +\ \beta \left\{ -\frac1{2} h\nabla b\nabla\cdot \u+\beta\nabla b(\nabla b\cdot \u)\right\}\\
&\eqdef \frac1{h}\nabla f_1+\beta f_2\nabla b.
\end{align*}
Thus we may use the identity valid for any sufficiently regular scalar functions $f,g$:
\[\u^\perp \curl(f\nabla g)+ \nabla (\u\cdot (f\nabla g))= (\u\cdot\nabla )(f\nabla g)+f (\nabla g \cdot\nabla)\u -(\nabla g)^\perp \curl(\u)
=(\u\cdot\nabla f)\nabla g+f\nabla(\nabla g\cdot\u)\]
to deduce
\begin{multline}\label{id1}
\u^\perp \curl(\T[h,\beta b] \u)+ \nabla (\u\cdot \T[h,\beta b] \u)\\
=\frac{-1}{h^2}(\u\cdot\nabla h)\nabla f_1+\frac1h \nabla(\nabla f_1\cdot\u)+\beta (\u\cdot\nabla f_2)\nabla b+ \beta f_2\nabla(\nabla b\cdot \u).
\end{multline}
Now, we write, using that $b$ is time-independent and replacing $\partial_t\zeta=-\nabla\cdot(h\u)$,
\begin{align} \big[\partial_t,\T[h,\beta b]\big]\u&=\frac{\epsilon\nabla\cdot(h\u)}{h^2}\nabla \Big(-\frac13 h^3\nabla\cdot\u+\frac\beta2h^2\u\cdot\nabla b\Big)+\frac1h\nabla\big(h^2(\epsilon\nabla\cdot(h\u)\nabla\cdot\u)\big)\nn\\&\qquad -\frac\beta{h}\nabla\big(h(\epsilon\nabla\cdot(h\u)\u\cdot\nabla b)\big)+\frac\beta2(\epsilon\nabla\cdot(h\u))(\nabla\cdot\u)\nabla b\nn\\
& =\frac{\epsilon\nabla\cdot(h\u)}{h^2} \nabla f_1+\frac{\epsilon}h\nabla\big((h\nabla\cdot(h\u))(h\nabla\cdot\u-\beta \u\cdot\nabla b)\big)+\frac{\epsilon\beta}2(\nabla\cdot(h\u))(\nabla\cdot\u)\nabla b.
\label{id2}
\end{align}

By~\eqref{id1} and~\eqref{id2}, the desired identity~\eqref{goal} becomes
\begin{multline}\label{new-goal}
\frac{\nabla\cdot\u}{h} \nabla f_1 +\frac1{h}\nabla\left(\big(\nabla f_1+(h^2\nabla\cdot\u)\nabla h - \beta h\nabla\cdot(h\u)\nabla b\big)\cdot\u + h^3(\nabla\cdot\u)^2\right)\\
+\frac{\beta}2(\nabla\cdot(h\u))(\nabla\cdot\u)\nabla b+\beta(\u\cdot\nabla f_2)\nabla b+\beta f_2\nabla(\nabla b\cdot \u) -\frac12\nabla\big((\beta\u\cdot\nabla b-h\nabla\cdot\u)^2\big)
\\
=\Q[h,\u]+ \Q_b[h,\beta b,\u].
\end{multline}
The non-topographical contributions in~\eqref{new-goal} (\ie setting $\beta=0$, including in $f_1$) are easily seen to match:
\begin{multline*} \frac{-1}3\frac{\nabla\cdot\u}{h} \nabla (h^3\nabla\cdot\u) +\frac1{h}\nabla\left(\frac{-1}{3}h^3 \big(\nabla(\nabla\cdot\u)\big)\cdot\u +h^3(\nabla\cdot\u)^2\right)-\frac12\nabla\big((h\nabla\cdot\u)^2\big)
\\
= \frac{-1}{3h}\nabla\Big(h^3\big((\u\cdot\nabla)(\nabla\cdot\u)-(\nabla\cdot\u)^2\big)\Big)=\Q[h,\u].
\end{multline*}
The remaining contributions in~\eqref{new-goal} satisfy,
denoting for readability $f=\nabla\cdot\u$ and $g=\beta \nabla b\cdot \u$, 
\begin{align*}&\frac12\frac{f}{h} \nabla (h^2 g) +\frac1{h}\nabla\left(\frac12\u\cdot\nabla(h^2 g )-  h(\u \cdot \nabla h+h f)g\right)\\
&\qquad +\frac{\beta}2( \u\cdot \nabla h+h f)f \nabla b+\beta(\u\cdot\nabla (g-\frac12 hf))\nabla b+ (g-\frac12 hf)\nabla g-\frac12\nabla\big(g^2-2 ghf \big)
\\
&=\frac12\frac{f}{h} \nabla (h^2 g) +\frac1{h}\nabla\left(\frac12 h^2\u\cdot \nabla g- h^2 f g\right)
+\frac{\beta}2hf^2 \nabla b+\beta(\u\cdot (\nabla g-\frac12 h\nabla f))\nabla b-\frac12 hf\nabla g+\nabla\big(ghf \big)\\
&=\frac1{2h}\nabla\left( h^2\u\cdot \nabla g\right)-\frac\beta2 h (\u\cdot \nabla f)\nabla b +\frac{\beta}2hf^2 \nabla b+\beta(\u\cdot \nabla g)\nabla b \ +\ \text{ cancellations}
\\
&= \Q_b[h,\beta b,\u].
\end{align*}
Thus the identity~\eqref{goal} holds, and the Proposition is proved. 
\end{proof}

\paragraph{Proof of Theorem~\ref{T.WP}} Let $N\geq 4$, $b\in \dot{H}^{N+2}$ and $(\zeta_0,\u_0)\in H^N\times X^N$ satisfying~\eqref{cond-h0} with $h_\star,h^\star>0$. Denote $h_0=1+\epsilon\zeta_0-\beta b$ and 
\[\v_0=h_0^{-1}\mfT[h_0,\beta b]\u_0=\u_0+\mu \T[h_0,\beta b]\u_0,\]
where we recall that the operator $\T$ is defined in~\eqref{def-T}. Using Lemma~\ref{L.embeddings},~\ref{L.products} and~\ref{L.com-Hn}, one easily checks that $\v_0\in Y^N,\curl\v_0\in H^{N-1}$ and
\[\norm{\v_0}_{Y^N}+\norm{\curl\v_0}_{H^{N-1}}\leq C(\mu,h_\star^{-1},h^\star,\beta\norm{\nabla b}_{H^{N+1}},\epsilon\norm{\zeta}_{H^4},\epsilon\norm{\u_0}_{X^4}) \big( \norm{\zeta}_{H^N}+\norm{\u_0}_{X^N}\big).\]
One may thus apply Propositions~\ref{P.existence} and~\ref{P.continuity}: there exists $T>0$ and $(\zeta,\v)\in C([0,T];H^N\times Y^N)$ strong solution to~\eqref{GN-v}; and one may restrict
\[ T^{-1}=C(\mu,h_\star^{-1},h^\star)F(\beta\norm{\nabla b}_{H^{N+1}},\epsilon\norm{\zeta_0}_{H^N},\epsilon \norm{\u_0}_{X^N} ) \]
such that, for any $t\in[0,T]$,~\eqref{cond-h0} holds with $\tilde h_\star=h_\star/2,\tilde h^\star=2h^\star$, and
\[ \E^N(\zeta,\v)+\norm{\curl\v}_{H^{N-1}}^2\leq {\bf C}_0\, \big(\E^N(\zeta_0,\v_0)+\norm{\curl\v_0}_{H^{N-1}}^2\big) 
\]
 with $ {\bf C}_0=C(\mu,h_\star^{-1},h^\star,\beta\norm{\nabla b}_{H^{N+1}},\epsilon\norm{\zeta_0}_{H^N},\epsilon\norm{\u_0}_{X^N})$. By Lemmas~\ref{L.embeddings},\ref{L.T-differentiable} and~\ref{L.diff-T-1} as well as Proposition~\ref{P.GNuvsGNv}, setting
 $\u\eqdef \mfT[h,\beta b]^{-1}(h\v)$ defines $(\zeta,\u)\in C([0,T];H^N\times X^N)$ strong solution to~\eqref{GN-u}, and one has
 \[ \sup_{t\in[0,T]} \big(\norm{\zeta}_{H^N}^2+\norm{\u}_{X^N}^2 \big) \leq {\bf C}_0 \big(\norm{\zeta_0}_{H^N}^2+\norm{\u_0}_{X^N}^2 \big) \]
  with $ {\bf C}_0=C(\mu,h_\star^{-1},h^\star,\beta\norm{\nabla b}_{H^{N+1}},\epsilon\norm{\zeta_0}_{H^4},\epsilon\norm{\u_0}_{X^4})$. 
  
 We thus constructed a strong solution to the Cauchy problem for~\eqref{GN-u} with initial data $(\zeta_0,\u_0)$. The uniqueness of the solution follows from the uniqueness in Proposition~\ref{P.existence} and Proposition~\ref{P.GNuvsGNv}. The continuity of the flow map follows from Proposition~\ref{P.continuity} and Lemmas~\ref{L.embeddings},\ref{L.T-differentiable} and~\ref{L.diff-T-1}.
 
 \paragraph{Proof of Theorem~\ref{T.justification}}
 By the assumptions of Theorem~\ref{T.justification}, one has $\zeta_0\in H^N,\nabla\psi_0 \in  H^{N}$ and therefore, by Lemmata~\ref{L.embeddings},~\ref{L.products} and~\ref{L.T-differentiable}, $\u_0\in X^{N}$ and
 \[\norm{\u_0}_{X^{N}}\leq C(\mu,h_\star^{-1},h^\star,\beta\norm{\nabla b}_{H^{N-1}},\epsilon\norm{\zeta_0}_{H^{N}})\norm{\nabla\psi_0}_{H^{N}}.\]
Thus Theorem~\ref{T.WP} applies, $(\zeta_{\rm GN},\u_{\rm GN})$ is well-defined, and one can restrict $T$ as in the Proposition to ensure that
 \[ \sup_{t\in[0,T]}\Big(\norm{\zeta_{\rm GN}}_{H^{N}}+\norm{\u_{\rm GN}}_{X^{N}}\Big)\leq {\bf C}_0,\]
 with ${\bf C}_0=C(\mu,h_\star^{-1},h^\star,\beta\norm{\nabla b}_{H^{N+1}},\epsilon\norm{\zeta_0}_{H^{N}},\epsilon\norm{\nabla \psi_0}_{H^{N}})$.
 
 Now, slightly adapting the proof of~\cite[Prop.~5.8]{Lannes} and denoting $h_{\rm ww}\eqdef 1+\epsilon\zeta_{\rm ww}-\beta b$ and
 \[\v_{\rm ww}\eqdef \nabla \psi_{\rm ww}, \qquad \u_{\rm ww}\eqdef \mfT[h_{\rm ww},\beta b]^{-1}(h_{\rm ww}\v_{\rm ww}),\]
 one finds that $(\zeta_{\rm ww},\u_{\rm ww})\in C(0,T;H^N\times X^{N})$ satisfies~\eqref{GN-u} up to remainder terms $r_{\rm ww},\r_{\rm ww}$, with
 \[\sup_{t\in[0,T]}\big(\norm{r_{\rm ww}}_{H^{N-6}}+\norm{\r_{\rm ww}}_{H^{N-6}}\big)\leq \mu^2\, {\bf C}_{\rm ww},\]
with ${\bf C}_{\rm ww}=C(\mu,h_\star^{-1},\beta\norm{ b}_{H^{N}},\epsilon\norm{\zeta_{\rm ww}}_{H^{N}},\epsilon\norm{\nabla \psi_{\rm ww}}_{H^{N}})$. Proposition~\ref{P.GNuvsGNv} immediately extends to non-trivial remainder terms, and it follows that $(\zeta_{\rm ww},\v_{\rm ww})\in C([0,T];H^{N}\times Y^{N})$ satisfies~\eqref{GN-v} up to the small remainder terms $r_{\rm ww},\r_{\rm ww}$. We apply Proposition~\ref{P.stability} and deduce that
 \[\big( \norm{\zeta_{\rm ww}-\zeta_{\rm GN}}_{H^{N-6}}+\norm{\v_{\rm ww}-\v_{\rm GN}}_{Y^{N-6}}\big)(t)\leq \mu^2\, {\bf C}\, t,\]
 with ${\bf C}$ as in the statement. This concludes the proof.

\appendix

\section{Energy estimates from the original formulation}\label{S.direct-estimates}

As mentioned in the introduction, on can obtain energy estimates directly from system~\eqref{GN-u}, rather than from system~\eqref{GN-v}, as carried out in this work. We roughly sketch the different steps below.

\paragraph{Quasilinearization of the system}
Let $(\zeta,\u)\in C([0,T] ;H^{|\alpha|}\times X^{|\alpha|})$ satisfies~\eqref{GN-u}, with $\alpha$ a non-zero multi-index and $T>0$. Assume that $b$ is sufficiently smooth, $|\alpha|$ is sufficiently large, and~\eqref{cond-h0} holds. 
Then  $\zeta_{(\alpha)}\eqdef \partial^\alpha\zeta$ and $\u_{(\alpha)}\eqdef \partial^\alpha\u$ satisfy
\begin{equation}\label{GN-u-quasi}
   \left\{ \begin{array}{l}
   \partial_t\zeta_{(\alpha)}+\epsilon\nabla\cdot(\u \zeta_{(\alpha)}) +\nabla\cdot(h\u_{(\alpha)}) =  r_{(\alpha)}\\ \\
\big(\Id+\mu \T[h,\beta b]\big)\partial_t\u_{(\alpha)} +\nabla\zeta_{(\alpha)}+\epsilon(\u\cdot\nabla)\u_{(\alpha)}
 +\mu\epsilon \Q_{(\alpha)}[h,\beta b,\u]\u_{(\alpha)}=\r_{(\alpha)},
\end{array}\right.
\end{equation}
with $h=1+\epsilon\zeta-\beta b$ and (abusing notations)
\begin{multline*} \Q_{(\alpha)}[h,\beta b,\u]\u_{(\alpha)}\eqdef \frac{-1}{3h}\nabla\Big(h^3\big((\u\cdot\nabla)(\nabla\cdot\u_{(\alpha)})\big)\Big)\\
+\frac{\beta}{2h}\Big(\nabla\big( h^2(\u\cdot\nabla) ( \u_{(\alpha)}\cdot \nabla b)\big) -h^2\big((\u\cdot\nabla)(\nabla\cdot\u_{(\alpha)})\big)\nabla b\Big)
+\beta^2 \big((\u\cdot\nabla)(\u_{(\alpha)}\cdot\nabla b)\big)\nabla b
\end{multline*}
and
where  $(r_{(\alpha)},\r_{(\alpha)})\in C([0,T] ; L^2\times Y^0)$ satisfies
\begin{equation}\label{est-u-r1r2}
\norm{r_{(\alpha)}}_{L^2}+\norm{\r_{(\alpha)}}_{Y^0}\lesssim \ \norm{\zeta}_{H^{ |\alpha|}}+\norm{\u}_{X^{ |\alpha|}}.
\end{equation}
The system~\eqref{GN-u-quasi} satisfied by $(\zeta_{(\alpha)},\u_{(\alpha)})$ is nothing but the linearized system~\eqref{GN-u} around $(\zeta,\u)$, from which order-zero operators have been discarded. The estimate~\eqref{est-u-r1r2} would follow as in the proof of Proposition~\ref{P.quasi}, and in particular using quasilinearization formulas derived in Section~\ref{S.preliminary}.

\paragraph{A priori energy estimates} For sufficiently smooth and finite-energy solutions of~\eqref{GN-u-quasi}, we add the $L^2$-inner product of the first equation with $\zeta_{(\alpha)}$ and the one of the second equation with $h\u_{(\alpha)}$. After some cancellations, integrations by parts and rearrangements, we find
\begin{equation}\label{ODE-u}
\frac{\dd}{\dd t}\mathcal F_{(\alpha)} + \epsilon \mathcal G_{(\alpha)}
= \int_{\RR^d}  r_{(\alpha)}  \zeta_{(\alpha)} +h \r_{(\alpha)} \cdot \u_{(\alpha)} \,\dd x,\end{equation}
where
\[\mathcal F_{(\alpha)}\eqdef \frac12 \int_{\RR^d} \zeta_{(\alpha)}^2  + h|\u_{(\alpha)}|^2  + \ \mu\, h \T[h,\beta b] \u_{(\alpha)} \cdot\u_{(\alpha)} \, \dd x\]
and
\begin{multline*}\mathcal G_{(\alpha)}\eqdef \frac12 \int_{\RR^d} (\nabla\cdot\u) \zeta_{(\alpha)}^2 -\big(\partial_t\zeta+\nabla\cdot (h\u)\big) |\u_{(\alpha)}|^2
-\frac\mu3 \big(3 h^2 \partial_t \zeta +\nabla\cdot (h^3\u)\big)  (\nabla\cdot\u_{(\alpha)})^2\\
+\mu \big(2h\partial_t \zeta+\nabla\cdot ( h^2\u) \big)(\beta\nabla b\cdot\u_{(\alpha)} )\nabla\cdot\u_{(\alpha)}-\mu  \big(\partial_t \zeta+\nabla\cdot(h\u)\big) (\beta\nabla b\cdot\u_{(\alpha)} )^2\, \dd x. \end{multline*}
By Lemma~\ref{L.T-invertible} and Cauchy-Schwarz inequality, we find that
\[ \norm{\zeta_{(\alpha)}}_{L^2}^2+\norm{\u_{(\alpha)}}_{X^0}^2 \lesssim \mathcal F_{(\alpha)} \quad \text{ and } \quad \mathcal G_{(\alpha)} \lesssim \norm{\zeta_{(\alpha)}}_{L^2}^2+\norm{\u_{(\alpha)}}_{X^0}^2.\]
Using~\eqref{est-u-r1r2} and again Cauchy-Schwarz inequality, Gronwall's Lemma to the differential equation~\eqref{ODE-u} yields (locally in time) the control of the energy $\F_{(\alpha)}$. Proceeding as in Sections~\ref{S.energy} and~\ref{S.WP}, one may then set up a Picard iteration scheme which yields the strong local well-posedness of the Cauchy problem for system~\eqref{GN-u}.

\section{Derivation of  the Green-Naghdi system}\label{S.Formulations-Hamiltonian}

Our work is based on a non-standard formulation of the Green-Naghdi system. We would like to motivate the relevance of this formulation (the verification of the equivalence between the different formulations is provided in Proposition~\ref{P.GNuvsGNv}). Below, we formally derive the non-standard formulation of the Green-Naghdi system from the Hamiltonian formulation of the water waves system, by approximating the associated Hamiltonian functional. This study, which was essentially provided in~\cite{CamassaHolmLevermore96}, has the advantage of revealing in a very straightforward way the Hamiltonian structure of the non-standard formulation of the Green-Naghdi system (and therefore the associated conservation laws) and giving a natural physical interpretation of the variables at stake. 

Let us first recall the canonical Hamiltonian structure of the water waves system as brought to light by~\cite{Zakharov68} and Craig-Sulem~\cite{CraigSulemSulem92,CraigSulem93}. Define the following Hamiltonian functional 
\begin{equation}\label{def-H}
\mathcal{H}(\zeta,\psi)\eqdef \frac12\int_{\RR^d} \zeta^2+\frac1\mu \psi\, G^{\mu}[\epsilon\zeta,\beta b] \psi
\end{equation}
where $
\psi(t,X)=\phi(t,X,\epsilon\zeta(t,X))$ is the trace of the velocity potential at the surface, and $G^\mu$ is the Dirichlet-to-Neumann operator, defined by
\[G^\mu: \varphi \mapsto \sqrt{1+\mu|\epsilon\nabla\zeta|^2} (\partial_n \phi)\id{z=\epsilon\zeta}=(\partial_z\phi)\id{z=\epsilon\zeta}-\mu(\epsilon\nabla\zeta)\cdot(\nabla_X\phi)\id{z=\epsilon\zeta},\]
 where $\phi$ is the unique solution (see \eg \cite{Lannes} for a detailed and rigorous analysis) to 
 \begin{equation}\label{Laplace} \left\{\begin{array}{l}
 \mu\Delta_X\phi+\partial_z^2\phi=0 \quad \text{ in } \{(X,z)\in\RR^{d+1}, \  -1+\beta b(X)\leq z\leq \epsilon\zeta(X)\},\\
 \phi(X,\epsilon\zeta(X))=\varphi\quad \text{and }\quad (\partial_z \phi-\mu(\beta \nabla b)\cdot(\nabla_X\phi))(X,-1+\beta b(X))=0.
 \end{array}\right.\end{equation}
 The operator $G^\mu$ is well-defined provided $h\eqdef 1+\epsilon\zeta-\beta b\geq h_\star>0$, and one can then show that the Zakharov/Craig-Sulem formulation of the water waves system simply reads
 \begin{equation}\label{WW-psi}\partial_t\begin{pmatrix} \zeta\\\psi\end{pmatrix}=\begin{pmatrix} 0& \Id \\ -\Id & 0\end{pmatrix}\begin{pmatrix} \delta_\zeta\mathcal H\\\delta_\psi\mathcal H\end{pmatrix}.\end{equation}
 
 If one reformulates (in dimension $d=2$) the above system using, instead of the canonical variables $(\zeta,\psi)$, the variables $\zeta$ and  $\v=(v_1,v_2)^\top\eqdef \nabla\psi $, then one obtains
 \begin{equation}\label{WW-v}\partial_t \begin{pmatrix} \zeta \\ v_1 \\ v_2 \end{pmatrix} =- 
  \begin{pmatrix} 0&  \partial_1  & \partial_2\\
  \partial_1  & 0 & -q \\
  \partial_2 & q & 0\end{pmatrix}
  \begin{pmatrix}
  \delta_\zeta \mathcal{H}  \\ \delta_{v_1} \mathcal{H} \\ \delta_{v_2} \mathcal{H}
    \end{pmatrix} . \end{equation}
 where $q=\frac{\curl \v}{h}$. Of course, in our situation, $q\equiv 0$ since $\v=\nabla\psi$, but this contribution is kept for the analogy with the Euler or Saint-Venant Hamiltonian formalism ; see \eg~\cite{Shepherd90}. Keeping this contribution turns out to be necessary for comparing with the standard formulation of the Green-Naghdi system in the general setting; see Proposition~\ref{P.GNuvsGNv}.
 \medskip
 
Recall that one has the identity~\cite[Prop.~3.35]{Lannes}
\begin{equation}\label{Gvsu} \frac1\mu G^\mu[\epsilon\zeta,\beta b]\psi \ = \ -\nabla\cdot (h\u), \qquad  \u\eqdef \frac1{1+\epsilon\zeta-\beta b}\int_{-1+\beta b}^{\epsilon\zeta}\nabla_X\phi(\cdot,z)\ \dd z. \end{equation}
so that the first equation in~\eqref{WW-psi} or~\eqref{WW-v} simply reads
\[\partial_t\zeta+\nabla\cdot(h\u)=0,\]
which is the first equation of the Green-Naghdi system. The system is then completed by constructing an evolution equation for $\u$, containing only differential operators, which is approximately satisfied by exact solutions of the water waves system, through asymptotic expansions with respect to $\mu\to 0$. This equation has different equivalent formulations in the literature as the equations have been rediscovered several times; in this paper we use~\eqref{GN-u}, originating from~\cite[(26)]{LannesBonneton09} and justified in the sense of consistency in~\cite[Prop.~5.8]{Lannes}.

Our strategy here is different: we obtain equations written with the original variables $\zeta,\psi$ (or $\zeta,\v$) by using an asymptotic expansion of the Hamiltonian functional $\mathcal H$, and plugging it in~\eqref{WW-psi} or~\eqref{WW-v}. The strategy of deriving the Green-Naghdi system using an approximate Hamiltonian functional or Lagrangian is not new: it was already used in particular in~\cite{Whitham67,CraigGroves94} (leading however to an ill-posed system of Green-Naghdi type) and in~\cite{MilesSalmon85} to derive the original Green-Naghdi system; see also~\cite{CamassaHolmLevermore96,KimBaiErtekinEtAl01,ChhayDutykhClamond16,ClamondDutykhMitsotakis17}. 

Let us recall the Dirichlet-to-Neumann expansion~\cite[Remark 3.39]{Lannes}
\begin{equation}\label{dev-G} \frac1\mu G^\mu[\epsilon\zeta,\beta b]\psi=-\nabla\cdot(h\nabla\psi)+\mu\nabla\cdot(h\mathcal T[h,\beta b]\nabla\psi\big)+\O(\mu^2)\end{equation}
with the notation
\[
\mathcal T[h,b]V\eqdef \frac{-1}{3h}\nabla(h^3\nabla\cdot V)+\frac1{2h}\Big(\nabla\big(h^2\nabla b\cdot V\big)-h^2\nabla b\nabla\cdot V\Big)+\nabla b(\nabla b\cdot V).
\]
It would therefore be natural to consider the approximate Hamiltonian functional from~\eqref{def-H}
\[
\mathcal{H}(\zeta,\psi)\approx \frac12\int_{\RR^d} \zeta^2+\psi\Big(-\nabla\cdot(h\nabla\psi)+\mu\nabla\cdot(h\mathcal T[h,\beta b]\nabla\psi\big)\Big).
\]
However, plugging this approximation into~\eqref{WW-psi} or~\eqref{WW-v} yields an ill-posed system (in the sense that the linearized system around the trivial solution $\zeta=0,\psi=0$, in the flat-bottom case, exhibits unstable modes whose amplitude grows exponentially and arbitrarily rapidly for large frequencies). It is interesting to note that the obtained system corresponds to the one exhibited in~\cite[(10)-(11)]{Whitham67} and \cite[(14)-(15)]{CraigGroves94} (in the one-dimension and flat-bottom situation) and, as pointed out in~\cite[(1.8a),(1.8b)]{MilesSalmon85}, it reduces to the original (ill-posed) Boussinesq system when the amplitude is small, that is withdrawing $\O(\mu\epsilon)$ terms. 

The ill-posedness of the aforementioned systems can be traced back from the fact that the approximate Hamiltonian functional is no longer positive, whereas the original one is; see~\cite[Prop.~3.9 and 3.12]{Lannes}. This issue can be avoided as follows: by~\eqref{Gvsu} and \eqref{dev-G}, one has
\[ \u=\nabla\psi-\mu \mathcal T[h,\beta b]\nabla\psi+\O(\mu^2), \quad \text{thus} \quad \nabla\psi=\u+\mu \mathcal T[h,\beta b]\u+\O(\mu^2).\]
Now, we notice that the operator $\mfT[h,\beta b]$ defined by
\[
   \mfT[h,\beta b]\u\ \eqdef\ h\u+\mu h\T[h,\beta b]\u 
\]
is a topological isomorphism (see Lemma~\ref{L.T-invertible}), and therefore
\begin{equation}\label{uvspsi} \u=\mfT[h,\beta b]^{-1}(h\nabla\psi)+\O(\mu^2).\end{equation}
It is now natural to use the following approximation:
\begin{equation}\label{def-HGN}\mathcal{H}(\zeta,\psi)= \frac12\int_{\RR^d} \zeta^2+  (\nabla\psi) \cdot (h\u) \approx \frac12\int_{\RR^d} \zeta^2+( h\nabla\psi)\cdot \mfT[h,\beta b]^{-1}(h\nabla\psi)  \eqdef \mathcal H_{\rm GN}(\zeta,\psi).
\end{equation}
Now, plugging the new approximate Hamiltonian in~\eqref{WW-psi} and~\eqref{WW-v} yields, respectively,
  \begin{equation}\label{GN-psi2}
   \left\{ \begin{array}{l}
   \partial_t\zeta+\nabla\cdot(h\u) =0,\\ \\
\partial_t \psi +\zeta+\frac\epsilon2 \abs{\u}^2=\mu\epsilon\big(\R[h,\u]+\R_b[h,\beta b,\u]\big),
   \end{array}\right.
   \end{equation}
   and
    \begin{equation}\label{GN-vbis}
      \left\{ \begin{array}{l}
      \partial_t\zeta+\nabla\cdot(h\u) =0,\\ \\
   \big(\partial_t+\epsilon\u^\perp \curl\big) \v+\nabla\zeta+\frac\epsilon2\nabla(\abs{\u}^2)=\mu\epsilon\nabla \big(\R[h,\u]+\R_b[h,\beta b,\u]\big),
      \end{array}\right.
      \end{equation}
where $\v=\nabla\psi$, $\u\eqdef \mfT[h,\beta b]^{-1}(h\nabla\psi)$, $\curl (v_1,v_2)\eqdef \partial_1v_2-\partial_2v_1$, $(u_1,u_2)^\perp\eqdef(-u_2,u_1)$, and
      \begin{align*}
      \R[h,\u]&\eqdef \frac{\u}{3h}\cdot\nabla(h^3\nabla\cdot\u)+\frac12 h^2(\nabla\cdot\u)^2, \\
       \R_b[h,\beta b,\u]&\eqdef -\ \frac12  \left(\frac{\u}{h}\cdot\nabla\big(h^2(\beta\nabla b\cdot\u)\big)+ h(\beta\nabla b\cdot\u) \nabla\cdot\u+(\beta\nabla b\cdot\u)^2\right).
      \end{align*}
System~\eqref{GN-vbis} is the system we study, and we show in Proposition~\ref{P.GNuvsGNv} that it is equivalent to the standard formulation of the Green-Naghdi system, namely~\eqref{GN-u}. System~\eqref{GN-psi2} is immediately deduced in the situation $\curl \v=0$, and inherit the canonical Hamiltonian structure of the Zakharov/Craig-Sulem formulation of the water waves system. Notice that one can rewrite system~\eqref{GN-vbis} as
     \begin{equation}\label{GN-vter}
      \left\{ \begin{array}{l}
      \partial_t\zeta+\nabla\cdot(h\u) =0,\\ \\
   \big(\partial_t+\epsilon\u^\perp \curl\big) \v+\nabla\left(\zeta+\epsilon\u\cdot\v-\frac\epsilon2\u\cdot\u-\frac{\epsilon\mu}{2}w^2\right)=0,
      \end{array}\right.
      \end{equation}
  with $w= (\beta\nabla b)\cdot\u-h\nabla\cdot\u$.
To our knowledge, system~\eqref{GN-vter}, as a new formulation for the Green-Naghdi system, has been first brought to light in~\cite[(4.3)-(4.4)]{CamassaHolmLevermore96} (in the flat bottom case, the formulation~\eqref{GN-psi2} appears in~\cite[(6.5)]{MilesSalmon85} and \cite[(9.12)]{Salmon88} but is quickly disregarded in favor of the aforementioned ill-posed model). It appears also in~\cite[(5.14)-(5.15)]{KimBaiErtekinEtAl01} (in the irrotational setting),~\cite[(30)]{GavrilyukKalischKhorsand15} (in the flat bottom situation) and~\cite[(2.9)-(2.34)-(2.35)]{Matsuno16}. As a matter of fact, the latter references point out that system~\eqref{GN-vter} echoes a formulation of the water waves system. Indeed, system~\eqref{WW-v} may be equivalently written as
     \begin{equation}\label{WW-vbis}
      \left\{ \begin{array}{l}
      \partial_t\zeta+\nabla\cdot(h\u) =0,\\ \\
   \big(\partial_t+\epsilon U^\perp \curl\big)  \v+\nabla\left(\zeta+\epsilon U\cdot\v-\frac\epsilon2 U\cdot U-\frac{\epsilon\mu}{2}w^2\right)=0,
      \end{array}\right.
\end{equation}
  where $\v,U,w$ are determined from the velocity potential, $\phi$, by
  \[\v=\nabla\big(\phi\id{z=1+\epsilon\zeta}\big), \quad  \u=\int_{\beta b}^{1+\epsilon\zeta}\nabla\phi\ \dd z, \quad U=(\nabla_X\phi)\id{z=1+\epsilon\zeta}, \quad w=\frac1\mu(\partial_z\phi)\id{z=1+\epsilon\zeta}.\]
System~\eqref{WW-vbis} is determined by the sole variables $\zeta$ and $\v$ (and $b$), after solving the Laplace problem~\eqref{Laplace}. Now, by the identity~\eqref{Gvsu} and chain rule, one has
\[
 U=\v-\mu\epsilon w\nabla\zeta \quad \text{ and } \quad w=\epsilon U\cdot\nabla \zeta-\nabla\cdot(h\u).
\]
It follows in particular
\[U=\v+\O(\mu)=\u+\O(\mu) \quad \text{ and } \quad w= (\beta\nabla b)\cdot\u-h\nabla\cdot\u+\O(\mu), \]
and therefore~\eqref{GN-vter} is immediately seen as a $\O(\mu^2)$ approximation of~\eqref{WW-vbis}, with the abuse of notation $\u\eqdef \mfT[h,\beta b]^{-1}(h\v)$ and $ w\eqdef  (\beta\nabla b)\cdot\u-h\nabla\cdot\u$ being justified by the above approximations.


\bibliographystyle{abbrv}

\begin{thebibliography}{}

\end{thebibliography}


\begin{thebibliography}{10}

\bibitem{Alazard08}
T.~Alazard.
\newblock A minicourse on the low {M}ach number limit.
\newblock {\em Discrete Contin. Dyn. Syst. Ser. S}, 1(3):365--404, 2008.

\bibitem{AlazardMetivier09}
T.~Alazard and G.~Métivier.
\newblock Paralinearization of the {D}irichlet to {N}eumann operator, and regularity of three-dimensional water waves.
\newblock {\em Comm. Partial Differential Equations}, 34(12):1632--1704, 2009.

\bibitem{Alinhac89}
S.~Alinhac.
\newblock Existence d'ondes de raréfaction pour des systèmes quasi-linéaires hyperboliques multidimensionnels.
\newblock {\em Comm. Partial Differential Equations}, 14(2):173--230, 1989.

\bibitem{Alvarez-SamaniegoLannes08}
B.~Alvarez-Samaniego and D.~Lannes.
\newblock Large time existence for 3{D} water waves and asymptotics.
\newblock {\em Invent. Math.}, 171(3):485--541, 2008.

\bibitem{Alvarez-SamaniegoLannes08a}
B.~Alvarez-Samaniego and D.~Lannes.
\newblock A {N}ash-{M}oser theorem for singular evolution equations.
  {A}pplication to the {S}erre and {G}reen-{N}aghdi equations.
\newblock {\em Indiana Univ. Math. J.}, 57(1):97--131, 2008.

\bibitem{Barthelemy04}
E.~Barth{\'e}lemy.
\newblock Nonlinear shallow water theories for coastal waves.
\newblock {\em Surveys in Geophysics}, 25(3-4):315--337, 2004.

\bibitem{BellecColinRicchiuto}
S.~Bellec, M.~Colin, and M.~Ricchiuto.
\newblock {D}iscrete asymptotic equations for long wave propagation.
\newblock {\em SIAM J. Numer. Anal.}, 54:3280--3299, 2016.

\bibitem{Benzoni-GavageSerre07}
S.~Benzoni-Gavage and D.~Serre.
\newblock {\em Multidimensional hyperbolic partial differential equations.
  {F}irst-order systems and applications}.
\newblock Oxford Mathematical Monographs. The Clarendon Press Oxford University
  Press, Oxford, 2007.

\bibitem{BonaSmith75}
J.~L. Bona and R.~Smith.
\newblock The initial-value problem for the {K}orteweg-de {V}ries equation.
\newblock {\em Philos. Trans. Roy. Soc. London Ser. A}, 278(1287):555--601,
  1975.

\bibitem{BonnetonChazelLannesEtAl11}
P.~Bonneton, F.~Chazel, D.~Lannes, F.~Marche, and M.~Tissier.
\newblock A splitting approach for the fully nonlinear and weakly dispersive
  {G}reen-{N}aghdi model.
\newblock {\em J. Comput. Phys.}, 230(4):1479--1498, 2011.

\bibitem{BreschMetivier10}
D.~Bresch and G.~M{\'e}tivier.
\newblock Anelastic limits for {E}uler-type systems.
\newblock {\em Appl. Math. Res. Express. AMRX}, (2):119--141, 2010.

\bibitem{CamassaHolmLevermore96}
R.~Camassa, D.~D. Holm, and C.~D. Levermore.
\newblock Long-time effects of bottom topography in shallow water.
\newblock {\em Phys. D}, 98(2-4):258--286, 1996.
\newblock Nonlinear phenomena in ocean dynamics (Los Alamos, NM, 1995).

\bibitem{CamassaHolmLevermore97}
R.~Camassa, D.~D. Holm, and C.~D. Levermore.
\newblock Long-time shallow-water equations with a varying bottom.
\newblock {\em J. Fluid Mech.}, 349:173--189, 1997.

\bibitem{CastroLannesa}
A.~Castro and D.~Lannes.
\newblock Fully nonlinear long-waves models in presence of vorticity.
\newblock {\em J. Fluid Mech.}, 759:642--675, 2014.

\bibitem{ChhayDutykhClamond16}
M.~Chhay, D.~Dutykh, and D.~Clamond.
\newblock On the multi-symplectic structure of the {S}erre-{G}reen-{N}aghdi.
  equations.
\newblock {\em J. Phys. A}, 49(3):03LT01, 7, 2016.

\bibitem{ClamondDutykhMitsotakis17}
D.~Clamond, D.~Dutykh, and D.~Mitsotakis.
\newblock Conservative modified Serre-Green-Naghdi equations with improved dispersion characteristics.
\newblock {\em Commun. Nonlinear Sci. Numer. Simul.}, 45:245--257, 2017.

\bibitem{CraigGroves94}
W.~Craig and M.~D. Groves.
\newblock Hamiltonian long-wave approximations to the water waves problem.
\newblock {\em Wave Motion}, 19(4):367--389, 1994.

\bibitem{CraigSulem93}
W.~Craig and C.~Sulem.
\newblock Numerical simulation of gravity waves.
\newblock {\em J. Comput. Phys.}, 108(1):73--83, 1993.

\bibitem{CraigSulemSulem92}
W.~Craig, C.~Sulem, and P.-L. Sulem.
\newblock Nonlinear modulation of gravity waves: a rigorous approach.
\newblock {\em Nonlinearity}, 5(2):497--522, 1992.

\bibitem{DucheneIsrawiTalhouk16}
V.~Duch{\^e}ne, S.~Israwi, and R.~Talhouk.
\newblock A new class of two-layer {G}reen-{N}aghdi systems with improved
  frequency dispersion.
\newblock {\em Stud. Appl. Math.}, 137(3):356--415, 2016.

\bibitem{Gallagher05}
I.~Gallagher.
\newblock R\'esultats r\'ecents sur la limite incompressible.
\newblock {\em Ast\'erisque}, (299):Exp. No. 926, vii, 29--57, 2005.
\newblock S{\'e}minaire Bourbaki. Vol. 2003/2004.

\bibitem{GavrilyukKalischKhorsand15}
S.~Gavrilyuk, H.~Kalisch, and Z.~Khorsand.
\newblock A kinematic conservation law in free surface flow.
\newblock {\em Nonlinearity}, 28(6):1805--1821, 2015.


\bibitem{GavrilyukTeshukov01}
S.~Gavrilyuk and V.~Teshukov.
\newblock Generalized vorticity for bubbly liquid and dispersive shallow water equations 
\newblock {\em Contin. Mech. Thermodyn.}, 13:365--382, 2001.

\bibitem{GreenNaghdi76}
A.~E. Green and P.~M. Naghdi.
\newblock A derivation of equations for wave propagation in water of variable
  depth.
\newblock {\em J. Fluid Mech.}, 78(02):237--246, 1976.

\bibitem{Holm88}
D.~D. Holm.
\newblock Hamiltonian structure for two-dimensional hydrodynamics with
  nonlinear dispersion.
\newblock {\em Phys. Fluids}, 31(8):2371--2373, 1988.

\bibitem{Iguchi09}
T.~Iguchi.
\newblock A shallow water approximation for water waves.
\newblock {\em J. Math. Kyoto Univ.}, 49(1):13--55, 2009.

\bibitem{Israwi10a}
S.~Israwi.
\newblock Derivation and analysis of a new 2{D} {G}reen-{N}aghdi system.
\newblock {\em Nonlinearity}, 23(11):2889--2904, 2010.

\bibitem{Israwi11}
S.~Israwi.
\newblock Large time existence for 1{D} {G}reen-{N}aghdi equations.
\newblock {\em Nonlinear Analysis: Theory, Methods \& Applications},
  74(1):81--93, 2011.

\bibitem{Kato75}
T.~Kato.
\newblock The {C}auchy problem for quasi-linear symmetric hyperbolic systems.
\newblock {\em Arch. Ration. Mech. Anal.}, 58(3):181--205, 1975.

\bibitem{KimBaiErtekinEtAl01}
J.~W. Kim, K.~J. Bai, R.~C. Ertekin, and W.~C. Webster.
\newblock A derivation of the {G}reen-{N}aghdi equations for irrotational
  flows.
\newblock {\em J. Engrg. Math.}, 40(1):17--42, 2001.

\bibitem{Lannes}
D.~Lannes.
\newblock {\em The water waves problem}, volume 188 of {\em Mathematical
  Surveys and Monographs}.
\newblock American Mathematical Society, Providence, RI, 2013.
\newblock Mathematical analysis and asymptotics.

\bibitem{LannesBonneton09}
D.~Lannes and P.~Bonneton.
\newblock Derivation of asymptotic two-dimensional time-dependent equations for
  surface water waves propagation.
\newblock {\em Physics of Fluids}, 21(1):016601, 2009.

\bibitem{LannesMarche15}
D.~Lannes and F.~Marche.
 \newblock A new class of fully nonlinear and weakly dispersive {G}reen-{N}aghdi models for efficient {2D} simulations.
\newblock {\em   J. Comput. Phys.}, 282:238--268, 2015. 

\bibitem{LeGavrilyukHank10}
O.~Le Métayer, S.~Gavrilyuk, and S.~Hank.
 \newblock A numerical scheme for the {G}reen-{N}aghdi model.
\newblock {\em J. Comput. Phys.}, 229:2034--2045, 2010. 

\bibitem{LevermoreOliverTiti96}
C.~D. Levermore, M.~Oliver, and E.~S. Titi.
\newblock Global well-posedness for models of shallow water in a basin with a
  varying bottom.
\newblock {\em Indiana Univ. Math. J.}, 45(2):479--510, 1996.

\bibitem{Li02}
Y.~A. Li.
\newblock Hamiltonian structure and linear stability of solitary waves of the
  {G}reen-{N}aghdi equations.
\newblock {\em J. Nonlinear Math. Phys.}, 9(suppl. 1):99--105, 2002.
\newblock Recent advances in integrable systems (Kowloon, 2000).

\bibitem{Li06}
Y.~A. Li.
\newblock A shallow-water approximation to the full water waves problem.
\newblock {\em Comm. Pure Appl. Math.}, 59(9):1225--1285, 2006.

\bibitem{Matsuno16}
Y.~Matsuno.
\newblock Hamiltonian structure for two-dimensional extended Green–Naghdi
  equations.
\newblock {\em Proc. R. Soc. Lond. Ser. A Math. Phys. Eng. Sci.},
  472(2190):20160127--, 2016.

\bibitem{MilesSalmon85}
J.~Miles and R.~Salmon.
\newblock Weakly dispersive nonlinear gravity waves.
\newblock {\em J. Fluid Mech.}, 157:519--531, 1985.

\bibitem{MelinandMesognon}
B.~Mélinand and B.~Mésognon-Gireau.
\newblock {Large time existence for Water-Waves with large bathymetry}.
\newblock In preparation.

\bibitem{Mesognon-Gireau}
B.~Mésognon-Gireau.
\newblock {The Cauchy problem on large time for the Water Waves equations with
  large topography variations}.
\newblock {\em Ann. Inst. H. Poincaré Anal. Non Linéaire}, 34(1):89--118, 2017.

\bibitem{Mesognon-Gireaub}
B.~Mésognon-Gireau.
\newblock {The Cauchy problem on large time for a Boussinesq-Peregrine equation
  with large topography variations}.
\newblock {\em Adv. Differential Equations}, 22(7/8):457--504, 2017.

\bibitem{Mesognon-Gireaua}
B.~Mésognon-Gireau.
\newblock {The singular limit of the Water-Waves equations in the rigid lid
  regime}.
\newblock ArXiv preprint:1512.02424.

\bibitem{Oliver97}
M.~Oliver.
\newblock Classical solutions for a generalized {E}uler equation in two
  dimensions.
\newblock {\em J. Math. Anal. Appl.}, 215(2):471--484, 1997.

\bibitem{Peregrine67}
D.~H. Peregrine.
\newblock Long waves on a beach.
\newblock {\em J. Fluid Mech.}, 27:815--827, 3 1967.

\bibitem{Salmon88}
R.~Salmon.
\newblock Hamiltonian fluid mechanics.
\newblock {\em Annual Review of Fluid Mechanics}, 20(1):225--256, 1988.

\bibitem{Seabra-SantosRenouardTemperville87}
F.~J. Seabra-Santos, D.~P. Renouard, and A.~M. Temperville.
\newblock Numerical and experimental study of the transformation of a solitary
  wave over a shelf or isolated obstacle.
\newblock {\em J. Fluid Mech.}, 176:117--134, 3 1987.

\bibitem{Serre53}
F.~Serre.
\newblock Contribution {\`a} l'{\'e}tude des {\'e}coulements permanents et
  variables dans les canaux.
\newblock {\em La Houille Blanche}, 6:830--872, 1953.

\bibitem{Shepherd90}
T.~G. Shepherd.
\newblock {S}ymmetries, conservation laws, and {H}amiltonian structure in
  geophysical fluid dynamics.
\newblock {\em Advances in Geophysics}, 32:287--338, 1990.

\bibitem{SiriwatKaewmaneeMeleshko16}
P.~Siriwat, C.~Kaewmanee, and S.~V. Meleshko.
\newblock Symmetries of the hyperbolic shallow water equations and the
  {G}reen-{N}aghdi model in lagrangian coordinates.
\newblock {\em International Journal of Non-Linear Mechanics}, 2016.

\bibitem{SuGardner69}
C.~H. Su and C.~S. Gardner.
\newblock Korteweg-de {V}ries equation and generalizations. {III}. {D}erivation
  of the {K}orteweg-de {V}ries equation and {B}urgers equation.
\newblock {\em J. Mathematical Phys.}, 10:536--539, 1969.

\bibitem{TaylorIII}
M.~E. Taylor.
\newblock {\em Partial differential equations. {III} {N}onlinear equations},
  volume 117 of {\em Applied Mathematical Sciences}.
\newblock Springer-Verlag, New York, 1997.

\bibitem{Whitham67}
G.~B. Whitham.
\newblock Variational methods and applications to water waves.
\newblock {\em Proc. R. Soc. Lond. Ser. A Math. Phys. Eng. Sci.}, 299, 06 1967.

\bibitem{Zakharov68}
V.~E. Zakharov.
\newblock Stability of periodic waves of finite amplitude on the surface of a
  deep fluid.
\newblock {\em J. Appl. Mech. Tech. Phys.}, 9:190--194, 1968.

\end{thebibliography}

\end{document}